\documentclass[11pt,leqno]{amsart}

\usepackage{amsmath, amsfonts, amsthm, amssymb, graphicx, amscd,color,colortbl,}
\usepackage{float,epsf,subfig}
\usepackage[english]{babel}

\makeatletter
\renewcommand\maketitle{\par
  \begingroup
    \renewcommand\thefootnote{\@fnsymbol\c@footnote}%
    \def\@makefnmark{\rlap{\@textsuperscript{\normalfont\@thefnmark}}}%
    \long\def\@makefntext##1{\parindent 1em\noindent
            \hb@xt@1.8em{%
                \hss\@textsuperscript{\normalfont\@thefnmark}}##1}%
    \if@twocolumn
      \ifnum \col@number=\@ne
        \@maketitle
      \else
        \twocolumn[\@maketitle]%
      \fi
    \else
      \newpage
      \global\@topnum\z@   
      \@maketitle
    \fi
    \thispagestyle{plain}\@thanks
  \endgroup
  \setcounter{footnote}{0}%
  \global\let\thanks\relax
  \global\let\maketitle\relax
  \global\let\@maketitle\relax
  \global\let\@thanks\@empty
  \global\let\@author\@empty
  \global\let\@date\@empty
  \global\let\@title\@empty
  \global\let\title\relax
  \global\let\author\relax
  \global\let\date\relax
  \global\let\and\relax
}
\renewcommand*\@maketitle{%
  \newpage
  \null
  \vskip 2em%
  \begin{center}%
  \let \footnote \thanks
    {\LARGE \@title \par}%
    \vskip 1.5em%
    {\large
      \lineskip .5em%
      \begin{tabular}[t]{c}%
        \@author
      \end{tabular}\par}%
    \vskip 1em%
    {\large \@date}%
  \end{center}%
  \par
  \vskip 1.5em}
\renewcommand{\author}[1]{\gdef\@author{#1}}
\makeatother

\numberwithin{equation}{section} \theoremstyle{plain}
\newtheorem{theorem}{Theorem}[section]
\newtheorem{lemma}[theorem]{Lemma}
\newtheorem{proposition}[theorem]{Proposition}

\theoremstyle{definition}
\newtheorem{definition}[theorem]{Definition}

\theoremstyle{remark}
\newtheorem{remark}[theorem]{Remark}

\newcommand{\e}{\varepsilon}

\renewcommand{\Xi}{P}

\newcommand{\nnu}{{\boldsymbol\nu}}

\def\bes{\begin{equation*}}
\def\ees{\end{equation*}}

\newcommand{\be}{\begin{equation}}
\newcommand{\ee}{\end{equation}}
\newcommand{\bee}{\begin{eqnarray}}
\newcommand{\eee}{\end{eqnarray}}

\begin{document}
\title[Shock Diffraction by Convex Cornered Wedges]{Shock Diffraction by Convex Cornered Wedges
for the Nonlinear Wave System}
\author{Gui-Qiang G. Chen \qquad Xuemei Deng \qquad Wei Xiang \\
Accepted for publication in [Arch. Rational Mech. Anal.]}

\address{Gui-Qiang G. Chen,  Mathematical Institute, University of Oxford,
         Oxford, OX2 6GG, UK; School of Mathematical Sciences, Fudan University,
 Shanghai 200433, China;  Department of Mathematics, Northwestern University,
         Evanston, IL 60208, USA}
 \email{\tt chengq@maths.ox.ac.uk}

\address{Xuemei Deng, Mathematical Institute, University of Oxford,
         Oxford, OX2 6GG, UK; School of Mathematical Sciences, Xiamen University,
         Xiamen, Fujian 361005, China}
\email{\tt dmeimeisx@yahoo.com.cn}

\address{Wei Xiang, School of Mathematical Sciences, Fudan University,
 Shanghai 200433, China; and Mathematical Institute, University of Oxford,
         Oxford, OX2 6GG, UK}
         \email{\tt 071018004@fudan.edu.cn; xiang@maths.ox.ac.uk}

\keywords{Shock diffraction, convex cornered wedges, nonlinear wave system, existence, regularity, mixed type,
elliptic-hyperbolic mixed, sonic boundary, free boundary, oblique derivative condition, optimal regularity.}

\subjclass[2010]{
Primary: 35M10, 25M12, 35B65, 35L65, 35L70, 35J70, 76H05, 35L67, 35R35;
Secondary: 35L15, 35L20, 25J67, 76N10, 76L05}
\date{\today}
\maketitle

\begin{abstract}
We are concerned with rigorous mathematical analysis of shock diffraction by two-dimensional
convex cornered wedges in compressible fluid flow, through the
nonlinear wave system. This shock diffraction problem can be
formulated as a boundary value problem for second-order nonlinear
partial differential equations of mixed elliptic-hyperbolic type in
an unbounded domain. It can be further reformulated as a free
boundary problem for nonlinear degenerate elliptic equations of
second order with a degenerate oblique derivative boundary
condition. We establish a global theory of existence and optimal
regularity for this shock diffraction problem. To achieve this, we
develop several mathematical ideas and techniques, which are also
useful for other related problems involving similar analytical
difficulties.
\end{abstract}

\section{Introduction}
We are concerned with rigorous mathematical analysis of shock
diffraction by two-dimensional cornered wedges whose angles are less
than $\pi$ in compressible fluid flow, through the nonlinear wave
system. The study of the shock diffraction problem dates back to the
1950's and starts from the work of Bargman \cite{Bargman}, Lighthill
\cite{Lighthill1,Lighthill2}, Fletcher-Weimer-Bleakney \cite{FWB},
and Fletcher-Taub-Bleakney \cite{FTB} via asymptotic or experimental
analysis. Also see Courant-Friedrichs \cite{CFr} and Whitham
\cite{Whitham}.

In this paper, we develop several mathematical ideas and techniques
through the nonlinear wave system to establish a rigorous theory of
existence and regularity of solutions to the diffraction problem.
The nonlinear wave system consists of three conservation laws, which
takes the form:
\begin{equation} \label{1.1a}
\begin{array}{rr}
\rho_{t}+m_{x_1}+n_{x_2}=0,\\
m_{t}+p_{x_1}=0,\\
n_{t}+p_{x_2}=0,
\end{array}
\end{equation}
for $(t,\mathbf{x})\in[0,\infty)\times \mathbb{R}^2,  \mathbf{x}\in
\mathbb{R}^2$, where $\rho$ stands for the density,  $p$ for the
pressure, $(m,n)$ for the momenta in the $(x_1,x_2)$--coordinates.
The pressure-density constitutive relation is
\begin{equation}\label{pressure}
p(\rho)=\rho^\gamma/\gamma, \qquad \gamma>1,
\end{equation}
by scaling without loss of generality. Then the sonic speed
$c=c(\rho)$ is determined by
$$
c^2(\rho):=p'(\rho)=\rho^{\gamma-1}.
$$
Notice that $c(\rho)$ is a positive, increasing function for all
$\rho>0$.

The two-dimensional nonlinear wave system \eqref{1.1a} is derived
from the compressible isentropic gas dynamics by neglecting the
inertial terms, {\it i.e.}, the quadratic terms in the velocity; see
Canic-Keyfitz-Kim \cite{sbe2}. Also see Zheng \cite{z} for a related
hyperbolic system, the pressure gradient system of conservation
laws; the same arguments developed in this paper can be carried
through to establish a corresponding theory of existence and
regularity for the pressure gradient system.

Let $S_0$ be the vertical planar shock in the $(t, {\bf
x})$--coordinates, $t\in \mathbb{R}_+:=[0, \infty),
\mathbf{x}=(x_1,x_2)\in\mathbb{R}^2$, with the left constant state
$U_{1}=(\rho_{1},m_{1},0)$ and the right state
$U_{0}=(\rho_{0},0,0)$, satisfying
$$
m_1=\sqrt{\big(p(\rho_{1})-p(\rho_{0})\big)(\rho_{1}-\rho_{0})}>0,
\qquad \rho_1>\rho_0.
$$
When $S_0$ passes through a convex cornered wedge:
$$
W:=\{(x_1, x_2)\, :\, x_2<0,  -\infty<x_1\le x_2\,{\rm
ctan}\,\theta_w\},
$$
shock diffraction occurs, where the wedge angle $\theta_w$ is
between $-\pi$ and $0$; see Fig. \ref{cog}. Then the shock
diffraction problem can be formulated as the following mathematical
problem:

\begin{figure}[h]
  \centering
  \includegraphics[width=0.45\textwidth]{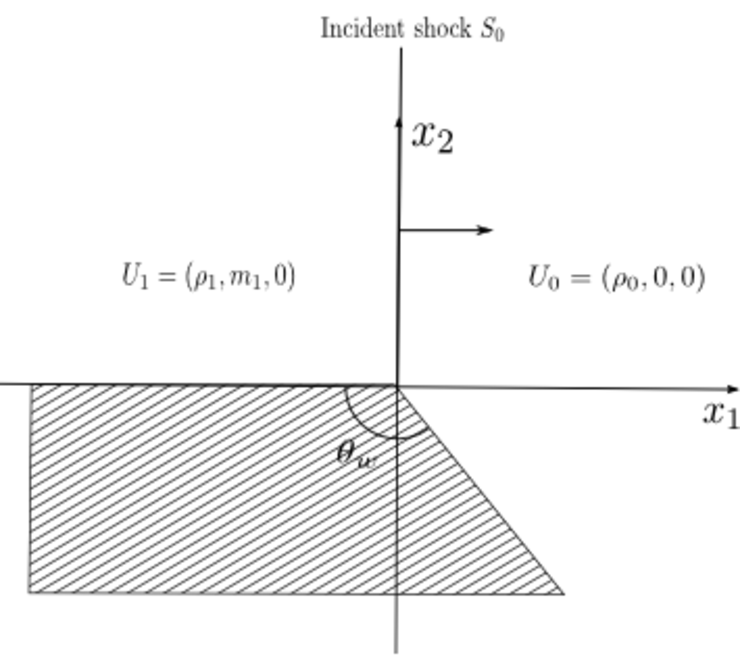} \qquad
  \includegraphics[width=0.45\textwidth]{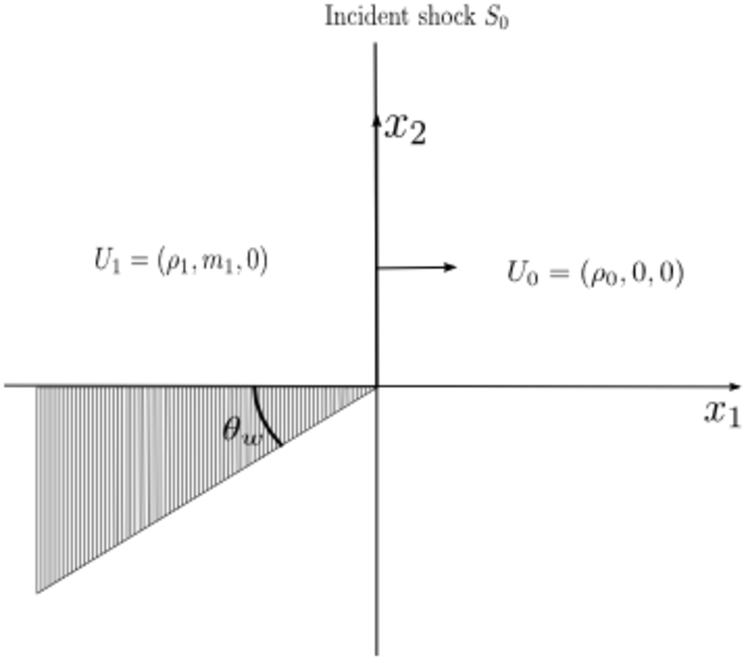}
  \caption{Initial-boundary value problem}
  \label{cog}
\end{figure}

\medskip
{\bf Problem 1 (Initial-Boundary Value Problem)}.  {\it Seek a
solution of system \eqref{1.1a} with the initial condition at $t=0$:
\begin{equation}\label{initial-condition}
(\rho, m, n)|_{t=0} =\begin{cases} (\rho_0, 0, 0) \quad &\mbox{\rm
in}\,\,\{-\pi+\theta_w\le \arctan\big(\frac{x_2}{x_1}\big)\le
\frac{\pi}{2}\},\\ (\rho_1, m_1, 0) \quad &\mbox{\rm in}\,\,
\{x_1<0, x_2>0\},
\end{cases}
\end{equation}
and the slip boundary condition along the wedge boundary $\partial
W$:
\begin{equation}\label{boundary-condition}
(m, n)\cdot \nnu|_{\partial W}=0,
\end{equation}
where $\nnu$ is the exterior unit normal to $\partial W$ (see Fig.
{\rm \ref{cog}}).}

\medskip
Notice that the initial-boundary value problem
\eqref{1.1a}--\eqref{boundary-condition}
 is invariant under the self-similar scaling:
\begin{equation}\label{2.2.2}
(t, {\bf x})\to (\alpha t, \alpha {\bf x}) \qquad
\quad\mbox{for}\quad \alpha\ne 0.
\end{equation}
Thus, we seek self-similar solutions with the form
\begin{equation}\label{2.2.3}
(\rho, m, n)(t, {\bf x})=(\rho, m, n)(\xi, \eta) \qquad\quad
\mbox{for}\quad  (\xi, \eta)=\frac{{\bf x}}{t}.
\end{equation}

In the self-similar coordinates $(\xi, \eta)$, system \eqref{1.1a}
can be rewritten as
\begin{equation} \label{1.5a}
\begin{array}{rr}
(m-\xi \rho)_\xi +(n-\eta\rho)_\eta+2 \rho=0,\\
\big(p(\rho)-\xi m\big)_\xi-(\eta m)_\eta+2m=0,\\
(\xi n)_\xi-\big(p(\rho)-\eta n\big)_\eta-2n=0.
\end{array}
\end{equation}
In the polar coordinates $(r,\theta), r=\sqrt{\xi^2+\eta^2}$, the
system can be further written as
\begin{equation}\label{1.5b}
\,\,\,
\partial_{r}\begin{pmatrix}
r\rho-\cos\theta\, m-\sin\theta\, n\\ rm-\cos\theta\, p(\rho)\\
rn-\sin\theta\, p(\rho)
 \end{pmatrix}+\partial_{\theta}
\begin{pmatrix}
\sin\theta\, m-\cos\theta\, n\\ \sin\theta\, p(\rho)\\
-\cos\theta\, p(\rho)
\end{pmatrix}
=\begin{pmatrix}
\rho+\frac{\cos\theta}{r}m+\frac{\sin\theta}{r}n \\  m+\frac{\cos\theta}{r}p(\rho)\\
n+\frac{\sin\theta}{r}p(\rho)
 \end{pmatrix}.
\end{equation}

The location of the incident shock $S_0$ for large $r\gg 1$ is:
\begin{equation}\label{ShockLocation:1}
\xi=\xi_1=\sqrt{\frac{p(\rho_1)-p(\rho_0)}{\rho_1-\rho_0}}>0.
\end{equation}

Then Problem 1 can be reformulated as a boundary value problem in an
unbounded domain:

\medskip
{\bf Problem 2 (Boundary Value Problem)}. {\it Seek a solution of
system \eqref{1.5a}, or equivalently \eqref{1.5b}, with the
asymptotic boundary condition when $r\to \infty$:
\begin{equation}\label{initial-condition-2}
(\rho, m, n)\to \begin{cases} (\rho_0, 0, 0)\quad\,
&\mbox{\rm in}\,\,\{\xi>\xi_1, \eta>0\}\cup\{-\pi+\theta_w\le\arctan\big(\frac{\eta}{\xi}\big)\le 0\},\\
(\rho_1, m_1, 0) \quad\, &\mbox{\rm in}\,\, \{\xi<\xi_1,\eta>0\},
\end{cases}
\end{equation}
and the slip boundary condition along the wedge boundary $\partial
W$:
\begin{equation}\label{boundary-condition-ss}
(m, n)\cdot \nnu|_{\partial W}=0,
\end{equation}
where $\nnu$ is the exterior unit normal to $\partial W$ (see Fig.
{\rm \ref{cog-2}}).}
\begin{figure}[h]
  \centering
  \includegraphics[width=0.47\textwidth]{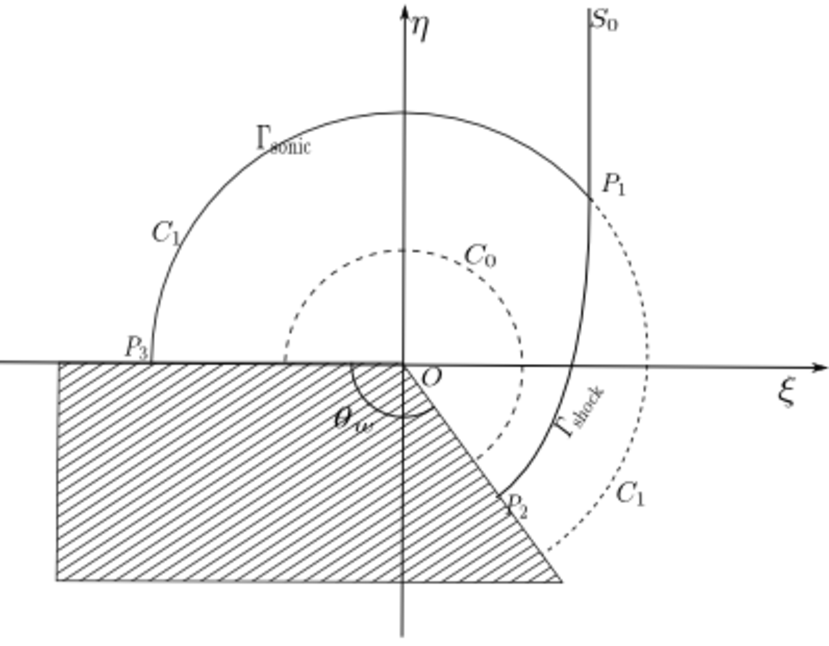}\qquad
  \includegraphics[width=0.47\textwidth]{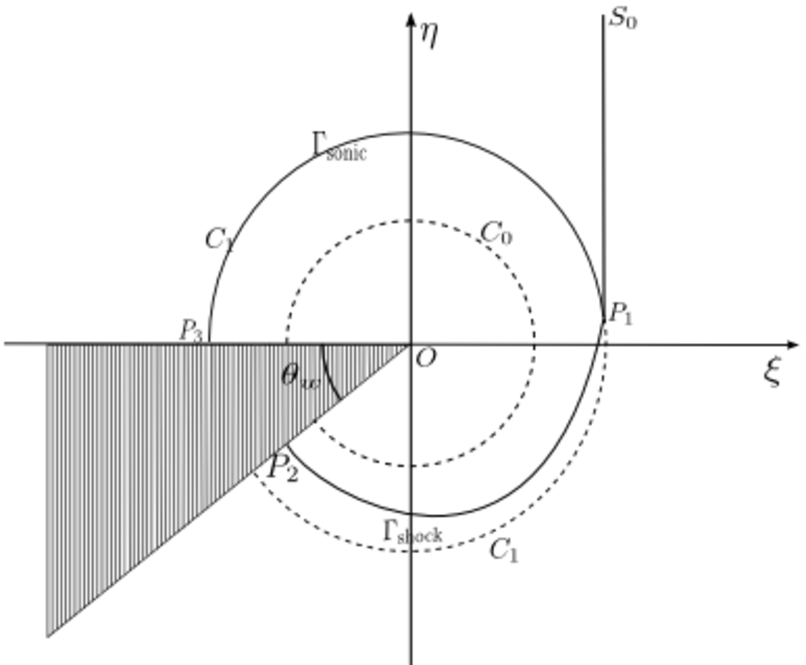}
  \caption{Shock diffraction configuration}
  \label{cog-2}
\end{figure}

\bigskip
For a smooth solution $U=(\rho,m,n)$ to \eqref{1.5a}, we may
eliminate $m$ and $n$ in \eqref{1.1a} to obtain a second-order
nonlinear equation for $\rho$: \begin{equation} \label{1.6a}
\big((c^{2}-\xi^{2})\rho_{\xi}-\xi\eta\rho_{\eta}+\xi\rho\big)_{\xi}
+\big((c^{2}-\eta^{2})\rho_{\eta}-\xi\eta\rho_{\xi}+\eta\rho
\big)_{\eta} -2\rho=0. \end{equation}

Correspondingly, equation \eqref{1.6a} in the polar coordinates $(r,
\theta), r=\sqrt{\xi^2+\eta^2}$, takes the form
\begin{equation}\label{1.8b}
\big((c^{2}-r^{2})\rho_{r}\big)_{r}+\frac{c^{2}}{r}\rho_{r}
+(\frac{c^{2}}{r^{2}}\rho_{\theta})_{\theta}=0.\end{equation}

In the self-similar coordinates, as the incident shock $S_0$ passes
through the wedge corner, $S_0$ interacts with the sonic circle
$\Gamma_{sonic}$ of state (1): $r=r_1$, and becomes a transonic
diffracted shock $\Gamma_{shock}$, and the flow in the domain
$\Omega$ behind the shock and inside $\Gamma_{sonic}$ becomes
subsonic. In Section 2, we reduce {\it Problem 2} for shock
diffraction into a one-phase free boundary problem, {\it Problem 3},
for second-order elliptic equation in the domain $\Omega$ with the
free boundary $\Gamma_{shock}$, degenerate boundary
$\Gamma_{sonic}$, and slip boundary $\partial W\cap \Omega$. In this
paper, we focus on the existence of global solutions of shock
diffraction and the optimal regularity of the solution across the
sonic circle $\Gamma_{sonic}$.

There are two additional difficulties to establish the global
existence of solutions, besides the ellipticity degenerates at the
sonic circle $\Gamma_{\text{sonic}}$. The first is that the oblique
derivative boundary condition degenerates at $P_2$, that is,
$\beta_2$ may equal to $0$, for which a one-point Dirichlet boundary
condition has to be identified to ensure the uniqueness of
solutions. The second difficulty is that the diffracted shock may
coincide with the sonic circle $C_0:=\{r=c(\rho_0)\}$ of state (0)
in the iteration where the oblique derivative boundary condition
fails again. Then we can not employ directly the results in
Liebermann \cite{g1}--\cite{g4} to show the existence of solutions
for the fixed boundary value problem. One of our strategies here is
to enforce an additional condition $r(\theta)\geq c(\rho_0)+\delta$ on
the diffracted shock curve with $\delta$ small enough and modify
slightly the approximate shock curve to overcome the difficulty.

The approach used in this paper for establishing the global
existence of solutions is first to regularize the equation by adding
the regularized differential operator $\varepsilon\Delta$ ($\Delta$
denotes the Laplace operator in the self-similar coordinates) to
make the equation uniformly elliptic; and then to rely on the Perron
method, as in \cite{gl1}, to show the global existence of solutions
for the fixed boundary value problem; and finally to apply the
Schauder fixed point theorem to show the existence of global
solutions for the free boundary problem. Moreover, we obtain uniform
estimates for the global solutions with respect to $\delta,
\varepsilon>0$ so that we can pass the limits $\delta\rightarrow0$
and $\varepsilon\rightarrow0$ to establish the existence of
solutions of the free boundary problem for the original system. In
particular, we prove that the diffracted shock is uniformly
transonic, that is, the strength of the shock is positive even at
point $P_2$.

In order to establish the optimal regularity across the sonic
boundary $\Gamma_{sonic}$, we write equation \eqref{1.8b} in terms
of the function
$$
\psi:=c^2(\rho_1)-c^2(\rho)
$$
in the $(x,y)-$coordinates, which is specified in \S 5, defined near
$\Gamma_{\text{sonic}}$ such that $\Gamma_{\text{sonic}}$ becomes a
segment on $\{x=0\}$, with the form
\begin{equation}\label{1.10}
\,\,\, (2c_{1}x-\psi)\psi_{xx}+\psi_{yy}
+c_{1}\psi_{x}-\psi_{x}^{2}-\frac{1}{(\gamma-1) c^2_1}\psi^2_y=0\
\quad \text{in}\ x>0\ \text{and near}\ x=0, \end{equation} plus
``small" terms, since $\rho$ and $\psi$ have the same regularity in
$\Omega$. Then we employ the approach in Bae-Chen-Feldman \cite{bcm}
to analyze the features of equation \eqref{1.10} and prove the
$C^{1,\alpha}$-regularity of solutions of the shock diffraction
problem in the elliptic region up to part
$\overline{\Gamma_{\text{sonic}}}\backslash P_1$ of the sonic shock.
As a corollary, we establish that the $C^{0,1}-$regularity is
actually optimal across the sonic boundary $\Gamma_{\text{sonic}}$
from the elliptic region $\Omega$ to the hyperbolic region of state
(1), that is, the optimal regularity at the degenerate elliptic
boundary.

We remark that the existence problem for a shock interaction with
the right cornered wedge (90-degree) was studied by Kim \cite{e}, in
which some important features and behavior of solutions have been
exhibited. As far as we have known, for the shock diffraction by a
convex cornered wedges whose angles are between $-\pi$ and $0$ in
compressible fluid flow, no rigorous complete global mathematical
results have been available, since the early work by Bargman
\cite{Bargman}, Lighthill \cite{Lighthill1,Lighthill2},
Fletcher-Taub-Bleakney \cite{FTB}, and Fletcher-Weimer-Bleakney
\cite{FWB}. The results established in this paper is the first
rigorous complete mathematical results through the nonlinear wave
system for the global existence and optimal regularity of solutions
of shock diffraction by any convex cornered wedge.

A closely related problem, shock reflection-diffraction by a concave
cornered wedges, has been systematically analyzed in Chen-Feldman
\cite{cm0,cm1,cm2} and Bae-Chen-Feldman \cite{bcm}, where the
existence of regular shock reflection-diffraction configurations has
been established up to the sonic wedge-angle for potential flow.
Also see Canic-Keyfitz-Kim \cite{sbe1,sbe2} for the unsteady
transonic small disturbance equation and the nonlinear wave system, Zheng \cite{z} for the pressure-gradient system, and Serre \cite{Serre} for the Euler equations for a Chaplygin gas.

The organization of this paper is as follows. In $\S 2$, we
reformulate the shock diffraction problem into a free boundary
problem for the nonlinear second-order equation \eqref{1.1a} in both
the self-similar and polar coordinates, and present the statement of
our main theorem for the existence and optimal regularity of the
global solution. In $\S 3$, we first formulate the regularized
approximate free boundary problem by adding a regularized
differential operator with $\varepsilon\Delta\rho$ to the original
equation and the assumption $c(\bar{\rho})\geq c(\rho_0)+\delta$,
where $\bar{\rho}$ is the data given at point $P_2$. Then we
establish the existence of solutions to the regularized free
boundary problem for the uniformly elliptic equation in the polar
coordinates, and so does in the self-similar coordinates, as
approximate solutions to the original free boundary problem. In $\S
4$, we proceed to the limits $\varepsilon\rightarrow0$ and
$\delta\rightarrow0$ to establish the global existence of solutions
of the original problem in the self-similar coordinates. In $\S 5$,
we establish the optimal $C^{0,1}$-regularity  of the solution
$\rho$ across the degenerate sonic boundary. In $\S 6$, we establish
a corresponding theorem for the existence and regularity of
solutions of the shock diffraction problem for the nonlinear wave
system.

\section{Mathematical Formulation and Main Theorem}

In this section, we derive mathematical formulation of the shock
diffraction problem as a free boundary problem for a nonlinear
degenerate elliptic equation of second order and present our main
theorem of this paper. In particular, we employ the Rankine-Hugoniot
relations to set up a boundary condition along the free boundary
(shock) and derive other boundary conditions along the wedge
boundaries in the polar coordinates.

\subsection{Rankine-Hugoniot Conditions and Oblique Derivative Boundary Condition on the Diffracted Shock}

Consider system \eqref{1.5b} in the polar coordinates. Then the
Rankine-Hugoniot relations, {\it i.e.}, the jump conditions, are
\begin{eqnarray}
&&\frac{dr}{d\theta}=
r\frac{\sqrt{r^2-\bar{c}^2(\rho,\rho_0)}}{\bar{c}(\rho,\rho_0)}, \label{2.34a}\\[2mm]
&&[p][\rho]=[m]^2+[n]^2, \label{2.6}
\end{eqnarray}
with
$\bar{c}(\rho,\rho_0)=\sqrt{\frac{p(\rho)-p(\rho_0)}{\rho-\rho_0}}$,
where we have chosen the plus branch so that $\frac{dr}{d\theta}>0$.
Differentiating \eqref{2.6} along $\Gamma_{\text{shock}}$ and using
the equations obtained above with careful calculation, we finally
obtain \begin{equation} \sum_{i=1}^2
\beta_iD_i\rho:=\beta_1\rho_r+\beta_2\rho_{\theta}=0, \end{equation}
where $\beta=(\beta_1,\beta_2)$ is a function of $(\rho_0, \rho,
r(\theta), r'(\theta))$ with
\begin{equation}\label{2.9}
\beta_{1}=r'(\theta)\big(c^{2}(r^2-\bar{c}^2)-3\bar{c}^2(c^{2}-r^2)\big),\quad
\beta_{2}=3c^{2}(r^{2}-\bar{c}^2)-\bar{c}^2(c^{2}-r^{2}).
\end{equation}
Thus, the obliqueness becomes
\begin{equation}\label{mu-def}
0\ne \beta\cdot(1,-r'(\theta))=-2r^2(c^2-\bar{c}^{2})r'(\theta)
=:\mu,
\end{equation}
where $(1,-r'(\theta))$ is the outward normal to $\Omega$ on
$\Gamma_{\text{shock}}$. Note that $\mu$ becomes zero when
$r'(\theta)=0$, that is, $r=\bar{c}(\rho,\rho_0)$. When the
obliqueness fails, we have
$$
\beta_1=0, \qquad \beta_2=-\bar{c}^2(c^2-r^2)<0,
$$
since $c^2(\rho)>\bar{c}^2(\rho,\rho_0)=r^2$ if $\rho>\rho_0$.

\medskip
We define $Q$ to be the governing second-order quasilinear operator
in the subsonic domain $\Omega$: \begin{equation}\label{2.40a}
Q\rho:=\big((c^{2}-r^{2})\rho_{r}\big)_{r}+\frac{c^{2}}{r}\rho_{r}
+\big(\frac{c^{2}}{r^{2}}\rho_{\theta}\big)_{\theta}=0,
\end{equation} and $M$ to be the oblique derivative boundary operator:
\begin{equation}\label{2.41a}
M\rho:=\beta_{1}\rho_{r}+\beta_{2}\rho_{\theta}=0\qquad\,\,\text{on}
\quad\Gamma_{\text{shock}}:=\{(r(\theta),\theta)\,
:\,\theta_w\leq\theta\leq\theta_1\}. \end{equation} The second
condition on $\Gamma_{\text{shock}}$ is the shock evolution
equation: \begin{equation}\label{2.42a}
\frac{dr}{d\theta}=r\frac{\sqrt{r^2-\bar{c}^2(\rho,\rho_0)}}{\bar{c}(\rho,\rho_0)}
:=g(r,\theta,\rho(r,\theta)), \qquad\,\,\, r(\theta_1)=r_1,
\end{equation} where $(r_1,\theta_1)$ are the polar coordinates of
$P_1=(\xi_1,\eta_1)$.

At point $P_2$, $r'(\theta_w)=0$, $M$ does not satisfy the oblique
derivative boundary condition at this point. We may alternatively
express this as a one-point Dirichlet condition by solving
$r(\theta_w)=\bar{c}(\rho(r(\theta_w), \theta_w),\rho_0)$. In order
to deal with this equation, we introduce the  notation:
\begin{equation} a=(\bar{c}_{b})^{-1}(r)\qquad  \text{when}\,\,
\bar{c}_b:=\bar{c}(a,b)=r \,\,\,\mbox{for fixed $b$}. \end{equation}
 Thus, we have
\begin{equation} \label{2.47a}
\rho(P_2)=\bar{\rho}=(\bar{c}_{\rho_0})^{-1}(r(\theta_w)).
\end{equation}

\subsection{Boundary Condition on the Wedge}
The boundary condition on the wedge is the slip boundary condition,
{\it i.e.}, $(m,n)\cdot\nnu=0$. Differentiating it along the wedge,
and combining this with the second and third equations in
\eqref{1.1a}, we conclude that $\rho$ satisfies \begin{equation}
\label{2.45a} \rho_{\nnu}=0\qquad\text{on}\,\,
\Gamma_0:=\partial\Omega\cap(\{\theta=\pi\}\cup
\{\theta=\theta_w\}). \end{equation}

\subsection{Boundary Condition on $\Gamma_{sonic}$ of State (1)}
The Dirichelt boundary condition on $\Gamma_{sonic}$:
\begin{equation} \label{2.43a} \rho=\rho_1\qquad\text{on}\,\,
\Gamma_{\text{sonic}}:=\partial\Omega\cap\partial B_{c_1}(0).
\end{equation} On the Dirichlet boundary $\Gamma_{\text{sonic}}$, the
equation $Q\rho=0$ becomes degenerate elliptic from the inside of
$\Omega$.

\subsection{Reformulation of the Shock Diffraction Problem}

With the derivation of the free boundary condition on
$\Gamma_{shock}$ and the fixed boundary conditions on
$\Gamma_{sonic}$ and the wedge $\Gamma_0$, {\it Problem 2} is
reduced to the following free boundary problem in the domain
$\Omega$ for the second order equation \eqref{2.40a}, with $(m, n)$
correspondingly determined by \eqref{1.5b}.

\medskip
{\bf Problem 3 (Free Boundary Problem)}. {\it Seek a solution of the
second-order nonlinear equation \eqref{2.40a} for the density
function $\rho$ in the domain $\Omega$, satisfying the free boundary
conditions \eqref{2.41a}--\eqref{2.47a} on $\Gamma_{shock}$, the
Neumann boundary condition \eqref{2.45a} on the wedge $\Gamma_0$,
and the Dirichlet boundary condition \eqref{2.43a} on the degenerate
boundary $\Gamma_{sonic}$, the sonic circle of state {\rm (1)} ({\it
cf.} Fig. {\rm \ref{cog-2}}).}

\subsection{Main Theorem}
For the free boundary problem, {\it Problem 3}, we have the
following results, which form the main theorem of this paper.

\begin{theorem}[\text{\bf Main Theorem}]\label{1}\label{2.1}
Let the wedge angle $\theta_w$ be between $-\pi$ and $0$. Then there
exists a global solution $\rho(r,\theta)$ in the domain $\Omega$
with the free boundary $r=r(\theta), \theta\in [\theta_w,
\theta_1]$, of Problem {\rm 3}:
$$
\rho\in C^{2+\alpha}(\Omega)\cap C^{\alpha}(\overline{\Omega}),
\quad r\in C^{2+\alpha}([\theta_w,\theta_1))\cap
C^{1,1}([\theta_w,\theta_1]).
$$
Moreover, the solution $(\rho(r,\theta), r(\theta))$ satisfies the
following properties:
\begin{enumerate}
\item[\rm (i)] $\rho> \rho_0$ on the shock $\Gamma_{shock}$, that is, the shock $\Gamma_{shock}$
is separated from the sonic circle $C_0$ of state {\rm (0)};

\item[\rm (ii)] The shock $\Gamma_{shock}$ is strictly convex up to point $P_1$,
except point $P_2$, in the self-similar coordinates $(\xi, \eta)$;
\item[\rm (iii)] The solution is $C^{1,\alpha}$ up to $\Gamma_{sonic}$ and
Lipschitz continuous across $\Gamma_{\text{sonic}}$;
\item[\rm (iv)] The Lipschitz regularity of solutions
across $\Gamma_{sonic}$ and at $P_1$ from the inside is optimal.
\end{enumerate}
\end{theorem}

In particular, Theorem \ref{2.1} implies the following facts:
\begin{enumerate}
\item[(a)] The diffracted shock $\Gamma_{shock}$ definitely is not degenerate at point $P_2$.
This has been an open question even when the wedge angle is
$\frac{\pi}{2}$ as in \cite{e}, though it is physically plausible.

\item[(b)] The curvature of the diffracted shock $\Gamma_{shock}$ away from point $P_2$ is strictly convex
and has a jump at point $P_1$ from a positive value to zero, while
the strict convexity of the curvature fails at $P_2$.

\item[(c)] The optimal regularity of solutions
across $\Gamma_{sonic}$ and at $P_1$ from the inside is $C^{0,1}$,
{\it i.e.}, Lipschitz continuity.
\end{enumerate}

We establish {Theorem \ref{1}} in two main steps. First, we solve
the regularized approximate free boundary problem for $Q$ involving
two small parameters $\varepsilon$ and $\delta$, introduced in \S 3.
Then we analyze the limits $\varepsilon\rightarrow0$ and
$\delta\rightarrow0$, and prove that the limits yield a solution of
{\it Problem 3}, {\it i.e.}, \eqref{2.40a}--\eqref{2.43a}, in \S 4.
The optimal regularity is established in \S 5.

\section{Regularized Approximate Problem}

In this section we first formulate the regularized approximate free
boundary problem and establish the existence of solutions to this
problem as approximate solutions to the original problem. To solve
the free boundary problem, we formulate the fixed point argument in
terms of the position of the free boundary. There are two main
difficulties in establishing the existence of solutions: The first
is that the ellipticity degenerates at the sonic circle
$\Gamma_{\text{sonic}}$; and the second is that the free boundary
$\Gamma_{\text{shock}}$ may coincide with the sonic circle $C_0$ of
state (0) in an iteration, which would make the iteration
impossible. We overcome these difficulties as described below.

\subsection{Approximate Problem and Existence Theorem for Approximate Solutions}

For fixed $\varepsilon>0$, introduce a regularized operator:
$$
Q^{\varepsilon}:=Q+\varepsilon\triangle,
$$
where $\triangle$ represents the Laplace operator in the
self-similar coordinates. For a given curve $r(\theta)$, we first
solve the fixed boundary value problem \eqref{2.40a}--\eqref{2.41a},
\eqref{2.45a}--\eqref{2.43a}, and \eqref{2.47a} with $Q$ replaced by
$Q^\varepsilon$; then we obtain a new shock position
$\tilde{r}(\theta)$ by integrating \eqref{2.42a}:
\begin{equation}\label{2.48}
\tilde{r}(\theta)=r_1+\int_{\theta_1}^{\theta}g(r(s),s,\rho(s,r(s)))ds
\qquad\text{for}\ \theta\in[\theta_w,\theta_1), \end{equation} where
$g$ is defined in \eqref{2.42a}. Note that, on the right side of
\eqref{2.48}, we evaluate all the quantities along the old shock
position $r(\theta)$.

With this, it seems that the free boundary could be obtained by
solving a fixed boundary problem and then by integrating the shock
evolution equation. However, we face the second difficulty as
indicated above, that is, $\tilde{r}(\theta)$ may meet the sonic
circle $C_0$ of state (0). To overcome the new difficulty, we
introduce another small, positive parameter $\delta$ which is fixed
and define the iteration set of $r$,
$\mathcal{K}^{\varepsilon,\delta}$, which is a closed, convex subset
of the H\"{o}lder space $C^{1+\alpha_1}([\theta_w,\theta_1])$, where
$\alpha_1$ depends on $\varepsilon$ and $\delta$ to be specified
later. The functions in $\mathcal{K}^{\varepsilon,\delta}$ satisfy

\medskip
\noindent
 ({\bf $K_1$}) \,\, $r(\theta_1)=r_1$;

\noindent ({\bf $K_2$}) \,\, $r'(\theta_w)=0$;

\noindent ({\bf $K_3$}) \,\, $c(\rho_0)+\delta\leq r(\theta_w)$;

\noindent ({\bf $K_4$}) \,\, $0\leq
r'(\theta)\leq\frac{r^2_1}{\bar{c}(\rho_0)}$ \,\, for
$\theta_w\leq\theta\leq\theta_1$.

\medskip
\noindent When the difficulty occurs, we modify $\tilde{r}(\theta)$
slightly somewhere as
$r(\theta)=c(\rho_0)+\delta+A(\theta-\theta_w)^3+B(\theta-\theta_w)^n$,
where $A,B$, and $n$ will be uniquely determined. Then we define a
mapping on $\mathcal{K}^{\varepsilon,\delta}$:
$$
J\,: \, r\rightarrow\tilde{r}.
$$

\medskip
We now restate the regularized approximate problem as follows: For
fixed $\varepsilon,\delta>0$, the equation for $\rho$ in the
subsonic region is \begin{equation}\label{3.1}
Q^{\varepsilon}\rho=\big((c^{2}-r^{2}+\varepsilon)\rho_{r}\big)_{r}
+\frac{c^2+\varepsilon}{r}\rho_{r}
+\big(\frac{c^{2}+\varepsilon}{r^{2}}\rho_{\theta}\big)_{\theta}=0;
\end{equation} the shock evolution equation remains the same when $r\geq
c(\rho_0)+2\delta$: \begin{equation}\label{3.2}
\begin{cases}
\frac{dr}{d\theta}=g(r,\theta,\rho), \\[1.5mm]
r(\theta_1)=r_1;
\end{cases}
\end{equation}
and \begin{equation}\label{3.2a}
r(\theta)=c(\rho_0)+\delta+A(\theta-\theta_w)^3+B(\theta-\theta_w)^n
\end{equation}
 for some constants $A,B$, and $n$ on the boundary when
\eqref{3.2} does not hold; the remaining boundary conditions as
before are \begin{equation}\label{3.3}
M\rho=\beta\cdot\nabla\rho=0\qquad \text{on}\
\Gamma_{\text{shock}}=\{(r,\theta)\,:\, \theta_w<\theta<\theta_1\},
\end{equation} \begin{equation}\label{3.4} \rho=\rho_1\quad \text{on}\
\Gamma_{\text{sonic}},\qquad\,\,\, \rho_{\bf \nu}=0\quad \text{on}\
\Gamma_0, \end{equation} where $ \nnu$ is the outward normal to
$\Omega$ at $\Gamma_0$; and \begin{equation}\label{3.5}
\rho(P_2)=\bar{\rho}=(\bar{c}_{\rho_0})^{-1}(r(\theta_w)).
\end{equation}

Let $V=\{P_1,P_2,O,P_3\}$ denote the corners of $\Omega$, and
$V'=V\backslash\{P_2\}$. Set $\Omega'=\overline{\Omega}\backslash
(V\cup \Gamma_{\text{shock}})$. For $\Xi\in V$, we define the corner
region
$$
\Omega_{\Xi}(\sigma):=\{x\in\Omega\,:\,\text{dist}(x,\Xi)\leq
\sigma\}, \qquad \Omega_V(\sigma):=\cup_{P\in
V}\Omega_{\Xi}(\sigma).
$$
We define a region that is close to $\Gamma_{\text{shock}}$, but
does not contain corner $P_1$ by taking a covering of
$\Gamma_{\text{shock}}$  with a ball of radius $\delta$ centered at
the points on $\Gamma_{\text{shock}}$ which are bounded away from
$P_1$. Define
$$
\Gamma'(\sigma):=\{\Xi\in \Gamma_{\text{shock}}\, :\,
 \text{dist}(\Xi,P_1)>\sigma\}
$$
and
$$
\Gamma(\sigma)=\{x\in \Omega\cap \big(\cup_{\Xi\in
\Gamma'(\sigma)}B_{\sigma}(\Xi)\big)\},
$$
where $B_{\sigma}(\Xi)$ is a ball of radius $\sigma$ centered at
$\Xi$. We then define \begin{equation} \label{3.22}
C_{b}^{a}\equiv\{u\, :\, \|u\|^{b}_{a}
:=\sup_{\sigma>0}\big(\sigma^{a+b}\|u\|_{a,\bar{\Omega}\backslash(\Gamma(\sigma)\cup
\Omega_{V'}(\sigma))}\big)<\infty\}. \end{equation}

We focus now on the proof of the following existence theorem in this
section.

\begin{theorem}\label{4}
For any $\varepsilon\in(0,\varepsilon_0)$ and
$\delta\in(0,\delta_0)$ for some $\varepsilon_0,\delta_0>0$, there
exists a solution
$(\rho^{\varepsilon,\delta},r^{\varepsilon,\delta})\in
C_{(-\gamma_1)}^{2+\alpha}(\Omega^{\varepsilon,\delta})\times
C^{1+\alpha}([\theta_w,\theta_1])$ to the regularized free boundary
problem \eqref{3.1}--\eqref{3.5} such that
 \begin{equation}\label{3.21}
\rho_0<\bar{\rho}^{\varepsilon,\delta}\leq\rho^{\varepsilon,\delta}<\rho_1,\
\quad\ c^2(\rho^{\varepsilon,\delta})\geq r^2\qquad\,\,\, \text{in}\
\overline{\Omega}^{\varepsilon,\delta} \end{equation}
 for some
$\alpha,\gamma\in(0,1)$, which depend on $\varepsilon$, $\delta$,
and the data $(\rho_0, \rho_1, \theta_w)$. Furthermore, the solution
satisfies \eqref{3.2} at the points of
$\Gamma^{\varepsilon,\delta}_{\text{shock}}$ where
$r^{\varepsilon,\delta}\geq c(\rho_0)+2\delta$. The function
$r^{\varepsilon,\delta}(\theta)$, defining the position of the free
boundary $\Gamma^{\varepsilon,\delta}_{\text{shock}}$, is in
$\mathcal{K}^{\varepsilon,\delta}$. Here
$\Omega^{\varepsilon,\delta}$ is bounded by
$\Gamma^{\varepsilon,\delta}_{\text{shock}}$,
$\Gamma_{\text{sonic}}$, and $\Gamma_0$.
\end{theorem}

We establish {Theorem \ref{4}} in the following steps whose details
are given in the following four subsections.

\medskip
{\it Step 1}. Since the governing equation \eqref{3.1} is nonlinear
and the ellipticity is not known a priori, we impose a cut-off
function in the equation $Q^{\varepsilon}\rho=0$.

We introduce a smooth increasing function $\zeta\in C^{\infty}$ such
that \begin{equation} \zeta(s):=
\begin{cases}
s \qquad&\text{if}\ s\geq0,\\
-\frac{1}{2}\varepsilon  \qquad&\text{if}\ s<-\varepsilon,
\end{cases}
\end{equation}
and $|\zeta'(s)|\leq1$. We then consider the following  modified
governing equation:
\begin{equation}\label{3.7a}
\begin{array}{lll}
Q^{\varepsilon,+}\rho&=&\big((\zeta(c^{2}-r^{2})+\varepsilon)\rho_{r}\big)_{r}
+(\frac{c^{2}+\varepsilon}{r^{2}}\rho_{\theta})_{\theta}
+\big(\frac{1}{r}(\zeta(c^{2}-r^{2})+\varepsilon)+r\big)\rho_{r}
\\[3mm]
&=&\sum_{i=1}^2
D_{i}\big(a^{\varepsilon}_{ii}(r,\theta,\rho)D_{i}\rho\big)
+b^{\varepsilon}(r,\rho)D_{r}\rho=0\qquad\quad\text{in}\,\,\,\Omega.
\end{array}
\end{equation}

\medskip
{\it Step 2}. We make some estimates for a solution to the linear
problem with fixed boundary $\Gamma_{\text{shock}}$ defined by
$r(\theta)\in \mathcal{K}^{\varepsilon,\delta}$ and establish the
Schauder estimates on $\Gamma_{\text{shock}}$.

\medskip
{\it Step 3}. We employ a technique in Lieberman \cite{gl1} to solve
the problem with the oblique derivative boundary condition
$M\rho=0$. Using the H\"{o}lder gradient bounds to the linear
problem, we establish the existence results for the linear fixed
boundary problem in the polar coordinates, via the Perron method
developed in \cite{gl1}.

\medskip
{\it Step 4}. We apply the Schauder fixed point theorem to conclude
the existence of a solution to the free boundary problem with the
oblique derivative boundary condition. Finally we remove the cut-off
function by the a priori estimates to conclude the results.

\subsection{Proof of Theorem \ref{4}: Regularized Linear Fixed Boundary Value Problem}
Replace $\rho$ in the coefficients $a^{\varepsilon}_{ii}$ and
$b^{\varepsilon}$ in \eqref{3.7a} and $\beta_i$ in \eqref{3.3} by a
function $w$ in a set $\mathcal{W}$ that is defined with respect to
a given boundary component
$\Gamma^{\varepsilon,\delta}_{\text{shock}}$ and depends on the
given values $\rho_0$, $\rho_1$, and
$\bar{\rho}^{\varepsilon,\delta}=(\bar{c}_{\rho_0})^{-1}(r^{\varepsilon,\delta}(\theta_w))$:

\smallskip
\begin{definition}\label{6}
The elements $w$ of $\mathcal{W}\subset C_{(-\gamma_1)}^{2}$ satisfy

({\bf W1}) $\rho_0<\bar{\rho}^{\varepsilon,\delta}\leq w\leq
\rho_1$, $w=\rho_1$ on $\Gamma_{\text{sonic}}$, $w_{\nnu}=0$ on
$\Gamma_0$, and $w(P_2)=\bar{\rho}^{\varepsilon,\delta}$;

\medskip
({\bf W2}) $\|w\|_{\alpha_0}\leq K_0$,
$\|w\|_{2+\alpha_0,\Omega_{\text{loc}}}\leq K_0$,
 and $\|w\|_{1+\mu,\Gamma(d_0)}\leq K_0$;

\medskip
({\bf W2}) $\|w\|^{(-\gamma_1)}_2\leq K_1$.
\end{definition}

\smallskip
The weighted H\"{o}lder space is defined by \eqref{3.22}. The values
of $\gamma_1, \alpha_0,\mu\in(0,1)$, and $d_0$, as well as the
values of $K_0$ and $K_1$, will be specified later. Obviously,
$\mathcal{W}$ is closed, bounded, and convex.

The quasilinear equation \eqref{3.7a} and the oblique derivative
boundary condition \eqref{3.3} are now replaced by the linear
equation and linear oblique derivative boundary condition on
$\Gamma^{\varepsilon,\delta}_{\text{shock}}:=\{(r(\theta),\theta)\,:\,\theta_w\leq\theta\leq\theta_1\}$:
\begin{equation}
\begin{array}{lll} \label{3.23}
&L^{\varepsilon,+}u:=\sum_{i=1}^2D_i\big(a^{\varepsilon}_{ii}(\Xi,w)D_iu\big)
+b^{\varepsilon}(\Xi,w)D_ru=0\qquad &\text{in}\ \Omega,\\[2mm]
&Mu:=\beta_{1}(\Xi,w)D_ru+\beta_{2}(\Xi,w)D_{\theta}u=0\quad
&\text{on}\ \Gamma^{\varepsilon,\delta}_{\text{shock}},
\end{array}
\end{equation}
with given $r(\theta)\in\mathcal{K}^{\varepsilon,\delta}\subset
C^{1+\alpha_1}([\theta_w,\theta_1])\cap C^{2}((\theta_w,\theta_1))$
and $w\in\mathcal{W}$, where the repeated indices are summed as
usual. Because of the cut-off function $\zeta$, $L^{\varepsilon,+}$
is uniformly elliptic in $\Omega$ with the ellipticity ratio
depending on the data and $\varepsilon$.

In this section, we demonstrate the key point that, thanks to the
uniform distance between the sonic circle $C_0$ of the right state
$(0)$ and $\Gamma_{\text{shock}}$, for a given function
$w\in\mathcal{W}$, the solution $u$ to the linear equation
\eqref{3.23} with the remaining boundary conditions:
\begin{equation}\label{3.24} u=\rho_1\,\,\, \text{on}\
\Gamma_{\text{shock}},\qquad u_{\mathbf{\nu}}=0\,\,\, \text{on}\
\Gamma_{0},\qquad u(P_2)=\bar{\rho}^{\varepsilon,\delta},
\end{equation}
satisfies the H\"{o}lder and Schauder estimates in $\Omega'$ \, and \,
the
uniform bound in $C^{1+\mu}(\Gamma(d_0))$ near
$\Gamma^{\varepsilon,\delta}_{\text{shock}}$ for any
$\mu<\min\{\gamma_1,\alpha_1\}$. This bound gives rise to enough
compactness to establish the existence of a solution to the
quasilinear problem by applying the Schauder fixed point theorem.

\medskip
First, we state the Schauder estimates up to the fixed boundary
$\Gamma_{\text{sonic}}$ with the Dirichlet boundary condition and to
$\Gamma_0$ with the Neumann boundary condition, and the H\"{o}lder
estimates at the corners in $V'$.

\begin{lemma}\label{7}
Assume that $\Gamma_{\text{shock}}$ is parameterized as
$\{(r(\theta),\theta)\}$ with
$r(\theta)\in\mathcal{K}^{\varepsilon,\delta}$ for some $\alpha_1$
and that $w\in\mathcal{W}$ for given $K_0$, $K_1$, $\alpha_0$, and
$\gamma$. Then there exist $\gamma_V,\alpha_{\Omega}\in(0,1)$ such
that any solution $u\in
C^{2+\alpha_{\Omega}}_{\text{loc}}(\Omega')\cap
C^{\gamma_V}(\Omega_{V'}(d_0))$ to the linear problem
\eqref{3.23}--\eqref{3.24} satisfies \begin{equation}\label{3.26}
\|u\|_{\gamma,\Omega_{V'}(d_0)}\leq C_1\|u\|_0 \qquad\mbox{for any
$\gamma\leq \gamma_V$,} \end{equation}
 and \begin{equation}\label{3.27}
\|u\|_{2+\alpha,\Omega'_{\text{loc}}}\leq C_2\|u\|_0 \qquad
\mbox{for any $\alpha\leq\alpha_{\Omega}$}. \end{equation}
 The exponent
$\gamma_V$ depends on the data $\rho_0$, $\rho_1$, and $\theta_w$;
and both $\alpha_{\Omega}$ and $\gamma_V$ depend on $\varepsilon$
but are independent of $\alpha_1$ and $\gamma_1$. The constant $C_2$
is independent of $K_1$ but depends on $K_0$.
\end{lemma}

\begin{proof}
The corner estimates at $P_1$ and $P_3$ directly follow from the
results in Theorem 1, Lieberman \cite{g4}. Near the origin, the
wedge angle is larger than $\pi$; thanks to the symmetry of the
governing equation in the $\theta$-axis with form \eqref{1.8b}, we
derive the corner estimate as follows.

First we flat out the boundary by introducing the transformation:
$$
(r', \theta')=(r, \frac{\pi}{\pi-\theta_w}(\theta-\theta_w)), \qquad
(\xi', \eta')=(r'\cos\theta', r'\sin\theta').
$$
Then the governing equation in the $(r',\theta')$-coordinates takes
the form
$$
\begin{array}{lll}
\tilde{Q}^{\varepsilon,+}\rho
&=&\big((\zeta(c^2(w)-r^2)+\varepsilon)\rho_r\big)_r
+\big(\frac{\zeta(c^2(w)-r^2)+\varepsilon}{r}+r\big)\rho_r\\[1.5mm]
&&+\frac{\pi^2}{(\pi-\theta_w)^2}
\big(\frac{c^2(w)+\varepsilon}{r^2}\rho_{\theta}\big)_{\theta},
\end{array}
$$
and
\begin{equation}\label{3.27a}
\begin{array}{lll}
\tilde{Q}^{\varepsilon,+}\rho
&=&\Big(\big((\zeta(c^2(w)-r^2)+\varepsilon)\frac{\xi^2}{r^2}+\frac{\pi^2}{(\pi-\theta_w)^2}
\frac{(c^2+\varepsilon)\eta^2}{r^2}\big)\rho_{\xi}\Big)_{\xi}\\[1.5mm]
&&+
\Big(\big((\zeta(c^2(w)-r^2)+\varepsilon)\frac{\xi\eta}{r^2}-(\frac{\pi^2}{(\pi-\theta_w)^2}
\frac{(c^2+\varepsilon)\xi\eta}{r^2}\big)\rho_{\eta}\Big)_{\xi}\\[1.5mm]
&&+
\Big(\big((\zeta(c^2(w)-r^2)+\varepsilon)\frac{\xi\eta}{r^2}-\frac{\pi^2}{(\pi-\theta_w)^2}
\frac{(c^2+\varepsilon)\xi\eta}{r^2}\big)\rho_{\xi}\Big)_{\eta}\\[1.5mm]
&&+\Big(\big((\zeta(c^2(w)-r^2)+\varepsilon)\frac{\eta^2}{r^2}+\frac{\pi^2}{(\pi-\theta_w)^2}
\frac{(c^2+\varepsilon)\xi^2}{r^2}\big)\rho_{\eta}\Big)_{\eta}+\xi\rho_{\xi}+\eta\rho_{\eta}
\end{array}
\end{equation}
in the $(\xi',\eta')$-coordinates, where we drop $'$ for simplicity
without confusion. The eigenvalues of \eqref{3.27a} are
$$
\lambda_1=\zeta(c^2-r^2)+\varepsilon,\qquad
\lambda_2=(\frac{\pi}{\pi-\theta_w})^2(c^2+\varepsilon).
$$
Note that the transformation from the $(\xi,\eta)-$coordinates to
$(\xi',\eta')-$coordinates is invertible and the $C^{\alpha}$--norms
are equivalent, since $
\det\big(\frac{D(\xi',\eta')}{D(\xi,\eta)}\big)\equiv\frac{\pi}{\pi-\theta_w}>0
$ for all $(r,\theta)\in\mathbb{R}^2$.

As for the proof of the equivalence of the two norms, we have two
cases:

\begin{figure}
  \centering
      \includegraphics[width=0.31\textwidth]{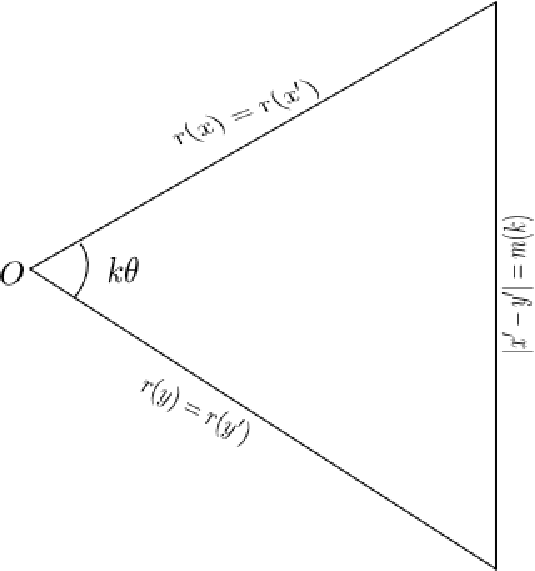}
  \caption{Scaling of the angles}
  \label{tr}
\end{figure}

\medskip
{\it Case 1}. If $\theta\geq \frac{\pi}{2}$  as in Fig. \ref{tr},
then
$$
\theta'=k\theta\geq \frac{\pi}{2} \qquad\mbox{with}\,\,\,
k=\frac{\pi-\theta_w}{\pi}.
$$
Since $1<k<2$, $|x-y|\geq\max\{r(x),r(y)\}$ and
$|x'-y'|\geq\max\{r(x'),r(y')\}$. Then the equivalence of the two
$C^{1,1}-$norms can be easily shown by setting $r(x)=r(x')$ and
$r(y)=r(y')$.

\medskip
{\it Case 2}. If $\theta<\frac{\pi}{2}$, then the distance between
two points in the $(\xi,\eta)-$coordinates and
$(\xi',\eta')-$coordinates is equivalent. By the cosine law, we
define
$$
m(k):=|x'-y'|^2=r(x)^2+r(y)^2-2r(x)r(y)\cos(k\theta),
$$
and then $ \frac{\partial m(k)}{\partial
k}=2kr(x)r(y)\sin(k\theta)>0. $ Thus,
$$
|x-y|=\sqrt{m(1)}\leq \sqrt{m(k)}=|x'-y'|\leq \sqrt{m(2)}\leq
2\sqrt{m(1)}=2|x-y|.
$$

Therefore, we can obtain the H\"{o}lder estimate of the solution at
$O$. Here $\gamma_V$ depends on the angle at the corner, a fixed
value that depends on the data $(\rho_0, \rho_1, \theta_w)$, and the
ellipticity ratio $\varepsilon$, but independent of $\gamma_1$,
$\alpha_1$, $K_0$, and $K_1$.

Finally, we can use the standard interior and boundary Schauder
estimates to obtain the local estimate \eqref{3.27}. The constant
$C_2$ depends on $\varepsilon$, the $C^{\alpha}-$norm of the
coefficients $a_{ij}$, and  the domain.
\end{proof}

Because the interior Schauder estimates can be further applied, a
solution in $C^{2+\alpha}_{\text{loc}}(\Omega')$ is actually in
$C^3_{\text{loc}}(\Omega)$.

\medskip
We next establish the H\"{o}lder gradient estimates on
$\Gamma_{\text{shock}}$. It is at this point that we need to derive
the basic estimates at point $P_2$ where the boundary operator $M$
is not oblique. In order to avoid handling the Neumann boundary
condition on the wedge boundary $\theta=\theta_w$ separately at each
step of this proof, we reflect $\Omega$ across the wedge boundary
$\theta=\theta_w$, without further comment, {\it i.e.}, $\Omega$
includes $\Sigma_0$, and let $\Gamma_{\text{shock}}$ stand for the
full $C^{1+\alpha_1}$--boundary in Lemma \ref{8} below. In addition,
we extend $\tilde{u}(2\theta_w-\theta)=u(\theta)$ for
$\theta\in(\theta_w,\theta_1)$ in a small neighborhood of
$\theta_w$. We still denote $\tilde{u}$ by $u$ for simplicity
without confusion.

\begin{lemma}\label{8}
Assume that $\Gamma_{\text{shock}}$ is given by
$\{(r(\theta),\theta)\}$ with
$r(\theta)\in\mathcal{K}^{\varepsilon,\delta}$ for some $\alpha_1$
and that $w\in\mathcal{W}$ for given $K_0$, $K_1$, $\alpha_0$, and
$\gamma_1$. Then there exists a positive constant $d_0$ such that,
for any $d\leq d_0$, the solution $u\in
C^1_{\text{loc}}(\Omega\cup\Gamma_{\text{shock}})\cup
C^3_{\text{loc}}(\Omega)$ to the linear problem
\eqref{3.23}--\eqref{3.24} satisfies \begin{equation}\label{3.28}
\|u\|_{1+\mu,\Gamma(d)\backslash B_{d}(P_1)} \leq
C(\varepsilon,\delta,\alpha_1,\gamma_1,K_1,d_0)\|u\|_0
\end{equation} for any $\mu<\min\{\gamma_1,\alpha_1\}$.
\end{lemma}

\begin{proof}
Away from a neighborhood $B_{d_0}(P_2)$ of $P_2$, the operator $M$
is oblique. Thus we can apply Theorem 6.30 in \cite{gt} to obtain
\eqref{3.28} in $\Gamma(d)\backslash \{B_{d_0}(P_1)\cup
B_{d_0}(P_2)\}$, with a constant $C$ depending on $\epsilon$,
$\alpha_1$, $\Omega$, $d_0$, and $K_0$. For the estimates near
$P_2$, the proof is adopted from \cite{sbe2}, which is similar. The
main idea is that, for a given solution $u$ to
\eqref{3.23}--\eqref{3.24}, we define \begin{equation}\label{3.29}
v=\frac{u}{1+\|Du\|_0}\quad \text{and}\quad
z=Mv=\sum_{i=1}^2\beta_{i}(\Xi)D_iv. \end{equation} For $d_0>0$
small enough, $O\notin B_{d_0}(P_2)$. Then we construct barrier
functions $\pm g$ for $z$ on $B:=
B_{d_0}(P_2)\cap\overline{\Omega}$, by finding a suitable positive,
increasing function $g$, $g(0)=0$, such that
$$
|z|\leq g.
$$
More precisely, $g(\zeta)=g_0\zeta^{\mu}$ for any $\mu<\gamma_2$.
This barrier function leads to \begin{equation}\label{3.46}
|D(z+g)|\leq \|(z+g)\|^{(1-\mu)}_{1+\gamma_2}d^{\mu-1}\leq
C(m)d^{\mu-1}\qquad\text{for}\ d<d_0, \end{equation} which leads to
$$
\|v\|_{1+\mu}\leq C.
$$
Finally, using the definition of $v$ in \eqref{3.29}, we apply the
interpolation inequality with small $\vartheta>0$ to obtain
\begin{equation} \|u\|_{1+\mu}\leq C\big(1+\|Du\|_0\big) \leq
C\big(1+\vartheta \|u\|_{1+\mu}+C_{\vartheta}\|u\|_0\big)
\end{equation} and thus \eqref{3.28} holds. Therefore, we obtain the
H\"{o}lder gradient estimate at $\Gamma_{\text{shock}}$ for the
solution $u$ of \eqref{3.23}. See \cite{sbe2} for more details.
\end{proof}

\smallskip
Now we focus on the existence of solutions in Theorem \ref{4} for
problem \eqref{3.23}--\eqref{3.24}. First we introduce two
definitions with some modification in comparison with \cite{gl1}.

\medskip
We say that problem \eqref{3.23}--\eqref{3.24} is locally solvable
if, for each $y\in\overline{\Omega}$, there exists a neighborhood
$O(y)$ such that,  for any $h\in C(\overline{N})$ with $N:=O(y)\cap
\{\overline{\Omega}\backslash(\{P_2\}\cup \Gamma_{\text{sonic}})\}$,
there exists a solution $v\in C^2(N)\cap C(\overline{N})$ of the
problem $L^{\varepsilon,+}v=0$ in $N\cap \Omega$, $Mv=0$ on
$N\cap\partial \Omega$, and $v=h$ on $\partial'N$, when $P_2\notin
\overline{N}(y)$;  or $L^{\varepsilon,+}v=0$ in $N\cap \Omega$, $M
v=0$ on $N\cap\partial \Omega$, $v=h$ on $\partial'N$, and
$v|_{P_2}=\bar{\rho}^{\varepsilon,\delta}$, when $P_2\in
\overline{N}(y)$. Here $\partial'N=\partial N\cap\Omega$. For
brevity, we denote this function $v$ by $(h)_y$ to emphasize its
dependence on $h$ and $y$.

A subsolution (supersolution) of \eqref{3.23}--\eqref{3.24} is a
function $w\in C(\overline{\Omega})$ with
$$
w(r(\theta_w),\theta_w)\le \bar{\rho}^{\varepsilon,\delta}\quad\,
(w(r(\theta_w),\theta_w)\ge  \bar{\rho}^{\varepsilon,\delta})
$$
such
that, for any $y\in\overline{\Omega}$, if $h\geq w$ ($h\leq w$) on
$\partial'N$, then $(h)_y\geq w$ ($(h)_y\leq w$) in $N$. The set of
all subsolutions (supersolutions) is denoted by $S^-$ ($S^+$).

\medskip
We now establish the existence of  solutions to problem
\eqref{3.23}--\eqref{3.24}.

\begin{lemma}\label{10}
Assume that $\Gamma_{\text{shock}}$ is given by
$\{(r(\theta),\theta)\}$ with
$r(\theta)\in\mathcal{K}^{\varepsilon,\delta}$ for some $\alpha_1$
and that $w\in\mathcal{W}$ for given $K_0$, $K_1$, $\alpha_0$, and
$\gamma_1$. Then there exist $\gamma_V,\alpha_{\Omega}\in(0,1)$ and
$d_0>0$, which are independent of $\gamma_1$ and $\alpha_1$, such
that there exists a solution
$$
u^{\varepsilon,\delta} \in C^{1+\mu}(\Gamma(d)\backslash
B_{d}(P_1))\cap C^{2+\alpha}(\Omega') \cap
C^{\gamma}(\Omega_{V'}(d))
$$
to the linear problem \eqref{3.23}--\eqref{3.24} for any $\alpha\leq
\alpha_{\Omega}$, $\mu<\min\{\gamma_1,\alpha_1\}$,
$\gamma\leq\gamma_V$, and $d\leq d_0$, which satisfies
\eqref{3.26}--\eqref{3.27} and \eqref{3.28}.
\end{lemma}

\begin{proof}
For fixed $\varepsilon,\delta>0$, we denote
$u^{\varepsilon,\delta}=u$ in the proof without confusion. We use
the Perron method to show the existence of a solution to problem
\eqref{3.23}--\eqref{3.24}.

It suffices to show the local existence at $P_2$. In fact, let $B_2$
be a sufficiently small neighborhood of $P_2$ with smooth boundary
such that $O\notin B_2$, $\beta_1\leq0$, and  $\beta_2<0$. Then we
study the local existence in the $(\xi,\eta)-$coordinates in $B_2$.
Reflect the region  $B_2$ across $\theta=\theta_w$ to obtain a new
region, which is still denoted by $B_2$. Then we introduce the
coordinate transform in a neighborhood of $P_2$:
\begin{equation}
\hat{\xi}=\hat{\xi}(r,\theta),\qquad \hat{\eta}=\hat{\eta}(r,\theta)
\end{equation}
such that
\begin{equation*}
\begin{array}{ll}
\hat{\xi}(r_w,\theta_w)=0, \qquad\qquad &\hat{\eta}(r_w,\theta_w)=0,\\
\frac{\partial(\hat{\xi}, \hat{\eta})}{\partial r}=(0, -1),\,
&\frac{\partial(\hat{\xi},  \hat{\eta})}{\partial\theta}
=-(\frac{1}{\beta_{2}}, \frac{\beta_1}{\beta_2}).
\end{array}
\end{equation*}
Let $\Gamma_{shock}:=\{(\hat{\xi},\hat{\eta})\,:\,
\hat{\eta}=f(\hat{\xi})\} =\{(r,\theta)\,:\, r=r(\theta)\}$ in
$\hat{B}_2$. Then $\hat{\eta}(r(\theta),
\theta)=f(\hat{\xi}(r(\theta), \theta))$ and hence $ f'(\hat{\xi})
=-(\beta_1-\beta_2r'(\theta))\geq 0, $ and the function
$f(\hat{\xi})$ is increasing in $\hat{\xi}$ on
$\Gamma_{\text{shock}}\cap \hat{B}_2$. Thus, from
$\frac{\partial\hat{\xi}}{\partial\theta}=-\frac{1}{\beta_{2}}>0$
and $\frac{\partial\hat{\xi}}{\partial r}=0$, we have
$$
f(\hat{\xi})\geq 0.
$$

We replace $\Omega$ by $\Omega_{\sigma}$ which is the
$\sigma-$distance from point $P_2$ upward, see Fig. \ref{ex}. On the
bottom straight boundary of $\Omega_{\sigma}$, we impose
$$
u=\bar{\rho}^{\varepsilon,\delta}\qquad\text{on the bottom of}\
\Omega_{\sigma}.
$$
\begin{figure}[h!]
   \centering
      \includegraphics[width=0.4\textwidth]{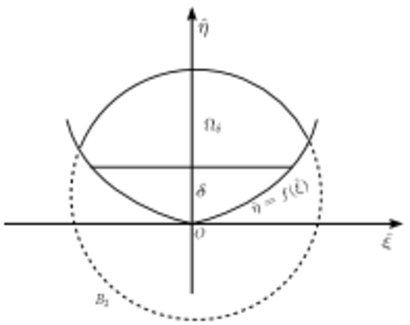}
  \caption{Domain with tip $P_2$ removed}
  \label{ex}
\end{figure}

Now we study the following boundary value problem:
\begin{equation}
\begin{cases}
\hat{L}^{\varepsilon,\delta}u=\sum_{i,j=1}^2\hat{a}_{ij}D_ju+\sum_{i=1}^2\hat{b}_iD_iu=0\qquad\,\,&\text{in}\, \Omega_{\sigma},\\
\hat{M}u=\partial_{\hat{\xi}}u=0\qquad\qquad\qquad\qquad &\text{on}\, \partial \Omega_{\sigma}\cap\Gamma_{\text{shock}},\\
u=h\qquad\qquad\qquad\qquad\qquad\qquad &\text{on}\, \partial B_2\cap\Omega,\\
u=\bar{\rho}^{\varepsilon,\delta}\qquad\qquad\qquad\qquad\qquad\qquad
&\text{on}\, \Sigma_{\sigma},
\end{cases}
\end{equation}
where
\begin{eqnarray*}
&&
\tilde{a}^{\varepsilon}_{11}=\frac{\hat{a}^{\varepsilon}_{11}}{\hat{\beta}^2_{2}},
\qquad   \tilde{a}^{\varepsilon}_{12}=\tilde{a}^{\varepsilon}_{21}
=-\frac{\hat{\beta}_{1}}{\hat{\beta}^2_{2}}\hat{a}^{\varepsilon}_{22},
\qquad \tilde{a}^{\varepsilon}_{22}=\hat{a}^{\varepsilon}_{11}
 +(\frac{\hat{\beta}_{1}}{\hat{\beta}_{2}})^2\hat{a}_{22},\\
&&\tilde{b}^{\varepsilon}_1=
\frac{\partial\hat{a}^{\e}_{11}}{\partial\hat{\eta}}
 -\frac{\hat{a}^{\e}_{22}}{\hat{\beta}^{2}_2}\frac{\partial \hat{\beta}_1}{\partial\hat{\xi}}+\frac{\hat{\beta}_1\hat{a}^{\e}_{22}}{\hat{\beta}^{2}_2}\frac{\partial \hat{\beta}_2}{\partial\hat{\eta}}+\frac{\hat{\beta}_1\hat{a}^{\e}_{22}}{\hat{\beta}^{3}_2}\frac{\partial \hat{\beta}_2}{\partial\hat{\xi}}-\frac{\hat{\beta}^2_1\hat{a}^{\e}_{22}}{\hat{\beta}^{3}_2}\frac{\partial \hat{\beta}_2}{\partial\hat{\eta}}+(\frac{\hat{\beta}_1}{\hat{\beta}_2})^2\frac{\partial \hat{a}^{\e}_{22}}{\partial\hat{\eta}}-\hat{b}^{\e},\\
&&
\tilde{b}^{\e}_2=-\frac{\hat{a}^{\e}_{22}}{\hat{\beta}^{3}_2}\frac{\partial
\hat{\beta}_2}{\partial\hat{\xi}}+\frac{\hat{a}^{\e}_{22}\hat{\beta}_1}{\hat{\beta}^{3}_2}
 +\frac{1}{\hat{\beta}^2}\frac{\partial \hat{a}^{\e}_{22}}{\partial\hat{\xi}}
 -\frac{\hat{\beta}_1}{\hat{\beta}^{2}_2}\frac{\partial \hat{a}^{\e}_{22}}{\partial\hat{\eta}}.
\end{eqnarray*}
Here $\hat{a}^{\e}_{ii}$, $\hat{b}^{\e}$, and $\hat{\beta}_i$,
$i=1,2$, are the coefficients of \eqref{3.23}--\eqref{3.24} in the
$(\hat{\xi}, \hat{\eta})-$coordinates, and $h$ is a continuous
function satisfying $\bar{\rho}^{\varepsilon,\delta}<h\leq\rho_1$.
Following Lieberman \cite{gl1}, there exists a solution
$$
u_{\sigma}\in C(\overline{\Omega}\cap\overline{\hat{B}_2}) \cup
C^{2,\alpha}(\Omega_{\sigma}\cap \hat{B}_2)
$$
for small enough $\hat{B}_2$. The maximum principle holds for
$u_{\sigma}$, which converges locally in $C^2(\Omega\cap \hat{B}_2)$
to a solution in $C^{2+\alpha}(\Omega\cap \hat{B}_2)$ as
$\sigma\rightarrow0_+$.

We now use a barrier function to obtain the continuity of $u$ at
$P_2$. We consider the auxiliary function \begin{equation}
v=\bar{\rho}^{\e,\delta}+c\big(1-e^{-l\,\hat{\eta}}\big),
\end{equation} where $c>0$ and $l>0$ are specified later. For the
oblique derivative boundary condition along
$\Omega_{\sigma}\cap\Gamma_{\text{shock}}$, we have the following
two cases:

\medskip
{\it Case 1}: $\tilde{\beta}\cdot\nnu>0$ and
$\tilde{M}(v-\bar{\rho}^{\e,\delta})\geq0$ when $\hat{\xi}>0$,

{\it Case 2}: $\tilde{\beta}\cdot\nnu<0$ and
$\tilde{M}(v-\bar{\rho}^{\e,\delta})\leq0$ when $\hat{\xi}<0$,

\medskip
\noindent where $\nnu$ denotes the outward normal to
$\Omega_{\sigma}$ at $\Omega_{\sigma}\cap\Gamma_{\text{shock}}$.

Moreover, it is easy to see that $v\geq \bar{\rho}^{\e,\delta}$ on
$\Sigma_{\sigma}$. Choose $C$ large enough such that $v\geq \sup|h|$
on $\partial \hat{B}_2\cap\Omega$. For the equation, we have
$$
\tilde{a}^{\e}_{ij}D_jv+\tilde{b}^{\e}_iD_iv=-ce^{-l \hat{\eta}}
(l^2\tilde{a}^{\e}_{11}-l\tilde{b}^{\e}_1),
$$
which is less than a negative constant by choosing
$l>\frac{\|\tilde{b}^{\e}\|_0}{\lambda}$, where $\lambda\leq
\tilde{a}^{\e}_{11}(\hat{\xi},\hat{\eta})$. Thus,
$$
\bar{\rho}^{\e,\delta}\leq u\leq v.
$$
Then $u$ is continuous at point $P_2$. The continuity of $u$ at the
other points follows from Lieberman's argument in \cite{g2,g4}.

By Lemma \ref{8}, we have $u\in
C^{1,\mu}(\hat{B}_2\cap\overline{\Omega})$.

In order to establish the global existence of solutions, it is
required to show
$$
\sup_\Omega(w^--w^+)=0,
$$
where $w^\pm$ are the supersolution and subsolution of problem
\eqref{3.23}--\eqref{3.24}, respectively.

In fact, we set $m:=\sup\limits_{\Omega}(w^--w^+)$. We assume that
$m>0$ in $\Omega$. Since $w^-(P_2)-w^+(P_2)\le 0$, there exists a
neighborhood  $\hat{B}_2(P_2)$ of $P_2$ such that $w^-(y)-w^+(y)<m$
for $y\in \hat{B}_2(P_2)$. Now we define
$$
\mathcal{Y}:=\{y\in\overline{\Omega} \, :\, w^-(y)-w^+(y)=m\}.
$$
Let $y_0\in \mathcal{Y}$ such that
$$
\text{dist}(y_0,P_2)=\min\limits_{y\in
\mathcal{Y}}\text{dist}(y,P_2).
$$
Let $\bar{w}^{\pm}$ be the lifts of $w^{\pm}$ in $M(y_0)$. We see
that $\bar{w}^--\bar{w}^+\leq m$ on $\partial'N$. The strong maximum
principle implies that either  $\bar{w}^--\bar{w}^+< m$ in $M$ or
$\bar{w}^--\bar{w}^+\equiv m$. Since
$\bar{w}^-(y_0)-\bar{w}^+(y_0)\geq w^-(y_0)-w^+(y_0)=m$, it follows
that $\bar{w}^--\bar{w}^+= m$ in $N$, and hence
$$
\bar{w}^--\bar{w}^+\equiv m \qquad\text{on}\,\, \partial'N,
$$
which contains the point of $\mathcal{Y}$ closer to $P_2$ than
$y_0$. This is a contradiction with the definition of $y_0$.

We refer to Lieberman \cite{g2} for handling the mixed case and both
points $P_1$ and $P_3$. As for the interior and the Dirichlet
boundary condition on the sonic arc $\Gamma_{sonic}$, they are
classical since the equation is uniformly elliptic for fixed $\e>0$
(see Gilbarg-Trudinger \cite{gt}).

With all of these, we then employ the Perron method to establish the
existence of a global solution.
\end{proof}

\subsection{Proof of Theorem \ref{4}: Regularized Nonlinear Fixed Boundary Problem}
We now establish the existence of solutions to the nonlinear problem
\eqref{3.1} with a fixed boundary.

\begin{lemma}\label{11}\label{2}\label{3}\label{13}
For $\e\in(0,\e_0)$ and $\delta\in(0,\delta_0)$, given
$r(\theta)\in\mathcal{K}^{\e,\delta}\subset C^{1+\alpha_1}$, there
exists a solution $\rho^{\e,\delta}\in
C_{(-\gamma_1)}^{2+\alpha}(\Omega^{\e,\delta})$ to problem
\eqref{3.1} and \eqref{3.4}--\eqref{3.5} with the oblique derivative
condition $M\rho^{\e,\delta}=0$ for some
$\alpha(\e,\delta),\gamma_1(\e,\delta)\in(0,1)$ such that
\begin{equation} \label{3.48}
\rho_0<\bar{\rho}^{\e,\delta}\leq\rho^{\e,\delta}\leq \rho_1.
\end{equation} Moreover, for some $d_0>0$, the solution $\rho^{\e,\delta}$
satisfies \begin{equation}\label{3.49}
\|\rho^{\e,\delta}\|_{\gamma,\Gamma(d_0)\cup B_{d_0}(P_1)}\leq K_2,
\end{equation} where $\gamma$ and $K_2$ depend on $\delta$, $\e$, $\gamma_V$,
and $K_1$, but are independent of $\alpha_1$. In addition, the
solutions satisfy the following three properties:
\begin{itemize}
\item[\rm (i)]{\it Ellipticity of  Eq. \eqref{3.1}}:\, $c^2(\rho^{\e,\delta})-r^2\geq0$ in $\overline{\Omega}^{\e,\delta}$;
\item[\rm (ii)]{\it \eqref{2.48} can always be integrated}:\, $\bar{c}(\rho^{\varepsilon,\delta},\rho_0)-r\leq0$ on $\Gamma^{\e,\delta}_{\text{shock}}$;
\item[\rm (iii)]{\it Local behaviour of the density near $\Gamma^{\e,\delta}_{\text{shock}}$ and the convexity of
$\Gamma^{\e,\delta}_{\text{shock}}$}:\, $\rho^{\e,\delta}$ is
monotone on $\Gamma^{\e,\delta}_{\text{shock}}$.
\end{itemize}
\end{lemma}
  The idea is to employ the fixed point theorem to prove the existence of the solution
  and then to find the barrier function to control the behaviour of the solution.
  The proof is long and similar to the one in \cite{sbe2} or \cite{e},
  while the main difference here is the singularity at origin for which we need to treat.
  Thus, we postpone the detailed proof which can be found in Appendix for self-containedness.

\subsection{Proof of Theorem \ref{4}: The Regularized Nonlinear Free Boundary Problem}
We now show the existence of a solution to the regularized free
boundary problem.

\begin{lemma}\label{12}
For each $\e\in(0,\e_0)$ and $\delta\in(0,\delta_0)$ with some
$\e_0>0$ and $\delta_0>0$, there exists a solution
$(\rho^{\e,\delta},r^{\e,\delta})
 \in C_{(-\gamma)}^{2+\alpha}(\Omega^{\e,\delta})
 \times C^{1+\alpha_1}([-\frac{\pi}{2},\theta_1))$
to the regularized free boundary problem \eqref{3.1}--\eqref{3.5}.
\end{lemma}

\begin{proof}
 For the notational simplicity,  we suppress
the $(\e,\delta)$-dependence in the proof.

For each $r(\theta)\in \mathcal{K}^{\e,\delta}\subset
C^{1+\alpha_1}([\theta_w,\theta_1])$, using the solution $\rho$ of
the nonlinear fixed boundary problem \eqref{3.1} and
\eqref{3.4}--\eqref{3.5} given by Lemma \ref{10}, we define the map
$J$ on $\mathcal{K}$, $\tilde{r}=Jr$, as in \eqref{2.48}:
\begin{equation}\label{3.57}
\tilde{r}(\theta)=r_1+\int_{\theta_1}^{\theta}g(r(s),s,\rho(r(s),s))ds.
\end{equation} There are two cases for the approximate shock position
$\tilde{r}(\theta)$:

\smallskip
{\it Case 1: $\tilde{r}(\theta_w)\geq c(\rho_0)+\delta$}. We check
that $J$ maps $\mathcal{K}$ into itself. It is easy to check that
$\tilde{r}(\theta)\in C^{1+\gamma_V}([\theta_w,\theta_1])\cap
C^{1+1}([\theta_w,\theta_1))$, from the definition of $\tilde{r}$
and by using Lemmas \ref{7}--\ref{8}, and property ($K_1$) follows
from \eqref{3.57}. By the definition of $g$ and
$\rho(P_2)=\bar{\rho}$, $\tilde{r}'(\theta)=0$ holds, which implies
property ($K_2$). Then it suffices to show that property ($K_4$)
holds, since the upper and lower bounds of $\rho$, Lemma \ref{3},
and ($K_4$) imply ($K_3$). From the expression of
$g(r(\theta),\theta,\rho(r(\theta),\theta))$ and the upper and lower
bounds of $\rho$, we have ($K_4$).

\smallskip
{\it Case 2: $\tilde{r}(\theta_w)< c(\rho_0)+\delta$}. Since
$\tilde{r}'(\theta)>0$ for $\theta\in(\theta_w,\theta_1)$ and
$r_1=c(\rho_1)>c(\rho_0)+\delta$, there exists a unique
$\theta_a\in(\theta_w,\theta_1)$ such that $\tilde{r}(\theta_a)=
c(\rho_0)+\delta$. Now, choosing $\tau$ to be determined later such
that $\tilde{r}(\theta_a+\tau)\leq c(\rho_0)+2\delta$ and letting
$x_1=\theta_a+\tau-\theta_w$, we modify the approximate shock
position on $\theta_w\leq\theta\leq \theta_a+\tau$ by defining
$$
\hat{r}(\theta)=c(\rho_0)+\delta+A(\theta-\theta_w)^3+B(\theta-\theta_w)^n
$$
with
$$
A=\frac{1}{(n-3)x^3_1}(na-bx_1), \qquad
B=\frac{1}{(n-3)x^n_1}(bx_1-3a),
$$
where $a=\tilde{r}(\theta_a+\tau)-c(\rho_0)-\delta$ and
$b=\tilde{r}'(\theta_a+\tau)$.

Choose $\tau$ small enough such that
$$
bx_1-3a>0,
$$
and then $n$ sufficiently large such that
$$
na-bx_1>0,
$$
where $n$ depends on $\delta$, but independent of the iteration.
Next, we choose $n$ and $\tau$ precisely. In fact, it is easy to see
that
$$
|b|\leq\frac{c(\rho_1)}{c(\rho_0)}\sqrt{c^2(\rho_1)-c^2(\rho_0)},
$$
and
$$
|bx_1|\leq C(\rho_0,\rho_1,\theta_1,\theta_w)
:=\frac{c(\rho_1)}{c(\rho_0)}\sqrt{c^2(\rho_1)-c^2(\rho_0)}(\theta_1-\theta_w).
$$

\smallskip
If $3\delta\leq bx_1$, we choose $\tau$ such that $a=\delta$ and
$n_1=\frac{C(\rho_0,\rho_1,\theta_1,\theta_w)}{\delta}+1$, which
depend only on $\delta$, $\rho_0$, and $\rho_1$.

\smallskip
If $3\delta> bx_1$, letting $\tau$ small enough, we can obtain new
$a$ and $b$ satisfying $3a=bx_1$, where we choose the biggest $\tau$
smaller than the old one such that $3a=bx_1$ holds. Note that
$bx_1>0$ and $\tilde{r}(\theta_a)=c(\rho_0)+\delta$. Thus, choosing
$n_2=4$, we have $A>0$ and $B=0$.

Let $n=\max(n_1,n_2)=n(\rho_0,\rho_1,\theta_1,\theta_w,\delta)$,
which is independent of the iteration process. Thus,
$\hat{r}(\theta)$, uniquely determined, is a strictly increasing
function on $[\theta_w,\theta_a+\tau]$. Furthermore, we have
$$
0=\hat{r}'(-\frac{\pi}{2})\leq {\hat{r}}'(\theta) \leq
{\hat{r}}'(\theta_a+\tau)=\tilde{r}'(\theta_a+\tau).
$$
We define
\begin{equation*}
Jr(\theta)=
\begin{cases}\tilde{r}(\theta) \qquad\,\, &\text{for}\  \theta\in[\theta_a+\tau,\theta_1],\\[1.5mm]
{\hat{r}}(\theta) \qquad\,\, &\text{for}\
\theta\in[\theta_w,\theta_a+\tau].
\end{cases}
\end{equation*}
It is easy to show that $Jr(\theta)$,
$\theta\in[\theta_w,\theta_1]$, satisfies properties
($K_1$)--($K_4$).

First, since $\tilde{r}(\theta)\in
C^{1+\gamma_V}([\theta_a+\tau,\theta_1])$, $\hat{r}(\theta)\in
C^{\infty}([\theta_w,\theta_a+\tau])$, and
 $(Jr)'(\theta)\in C([\theta_w,\theta_1])$, we have
$$
Jr(\theta)\in
 C^{1+\gamma_V}([\theta_w,\theta_1]).
$$

Next, for
 $\theta\in[\theta_w,\theta_a+\tau]$,
 $\hat{r}'(\theta)=3A(\theta-\theta_w)^2+nB(\theta-\theta_w)^{n-1}$.
Then
\begin{eqnarray*}
&&\hat{r}'(\theta_2)-\hat{r}'(\theta_3)\\
&&= 3A(\theta_2-\theta_w)^2-3A(\theta_3-\theta_w)^2  +nB(\theta_2-\theta_w)^{n-1}-nB(\theta_3-\theta_w)^{n-1}\\
&&=3A(\theta_2-\theta_3)(\theta_2+\theta_3-2\theta_w)+nB(\theta_2-\theta_3)
 \big(\sum_{j=0}^{n-2}
 C_{n-2}^{j}(\theta_2-\theta_w)^{n-2-j}(\theta_3-\theta_w)^{j}\big).
\end{eqnarray*}
Using the fact that $\theta_2-\theta_w,\theta_3-\theta_w\leq x_1$,
 and $A,B\geq 0$, we obtain
 \begin{eqnarray*}
|\hat{r}'(\theta_2)-\hat{r}'(\theta_3)|
&\leq& |\theta_2-\theta_3|^{\alpha}\big(6Ax_1^{2-\alpha}+C(n)Bx_1^{n-1-\alpha}\big)\\
&\leq&
C(n)\big(ax_1^{-1-\alpha}+bx_1^{-\alpha}\big)|\theta_2-\theta_3|^{\alpha}.
\end{eqnarray*}
Notice that $\tilde{r}'=\frac{r}{\bar{c}}\sqrt{r^2-\bar{c}^2}$,
$r\in C^{1+\gamma_V}$, and $\theta_a+\tau$ is uniformly away from
$\theta_1$, which means that
  $\rho\in C^{1+\mu}$. We obtain
$$
\tilde{r}'\leq C(\rho_0,\rho_1,\e,\delta)x_1^{1/2},
$$
which implies
$$
|\hat{r}'(\theta_2)-\hat{r}'(\theta_3)|
 \leq C(\rho_0,\rho_1,\e,\delta)|\theta_2-\theta_3|^{\alpha}
 \qquad \mbox{if}\,\,\, \alpha\leq \frac{1}{2}.
$$
Thus
$$
\|Jr\|_{C^{1+\alpha}([-\frac{\pi}{2},\theta_1])}\leq
C(\rho_1,\rho_2,\e,\delta)
$$
if $\alpha\leq \min\{\gamma_V, \frac{1}{2}\}$, which satisfies
($K_1$)--($K_4$).

Thus, we define a map
$$
J\, :\, \mathcal{K}^{\e,\delta}\rightarrow\mathcal{K}^{\e,\delta}
$$
by
$$
\tilde{r}=Jr.
$$
Obviously, $\mathcal{K}^{\e,\delta}$ is a convex and closed subset
of the Banach space $C^{\alpha_1}$, and $J$ is compact, if
$\alpha_1<\min\{\gamma_V, \frac{1}{2}\}$. In order to use the
Schauder fixed point theorem, we need to prove that $J$ is
continuous on $\mathcal{K}^{\e,\delta}$.

Assume that $r_m,r\in\mathcal{K}^{\e,\delta}$ for $m=1,2,\cdots$,
$r_m\rightarrow r$ as $m\to \infty$, and $\rho_m$ solves the fixed
boundary problem for $r_m$ for each fixed $m$. Then, by the standard
argument as in \cite{sbe2}, $\rho_m\rightarrow\rho$, which solves
the problem for $r$. Therefore, we have
$$
g(r_m(\theta),\theta,\rho_m(r_m(\theta),\theta))
\longrightarrow(r(\theta),\theta,\rho(r(\theta),\theta)) \qquad
m\to\infty,
$$
which implies $ Jr_m\rightarrow Jr $ as $m\to \infty$ at the point
where \eqref{3.2} holds for both $r_m$ and $r$. Then
$Jr_m\rightarrow Jr$ as $m\to\infty$, if $Jr$ belongs to Case 1. For
Case 2, due to the construction, we divide it into three subcases:
$$
3\delta< bx_1; \qquad 3\delta> bx_1; \qquad 3\delta=bx_1,
$$
where $b=\tilde{r}'(\theta_a+\tau)$, $x_1=\theta_a+\tau-\theta_w$,
and $\tilde{r}(\theta_a+\tau)=c(\rho_0)+2\delta$ depend only on $r$
and $\delta$. For any case, it is easy to deduce that
$$
(\tau_m, \theta_{a,m}) \rightarrow (\tau, \theta_a), \quad (A_m,
B_m)\rightarrow (A, B) \qquad\,\, \mbox{when}\,\, m\to \infty.
$$
Then $Jr_m\rightarrow Jr$, with the fact that
$$
Jr_m=c(\rho_0)+\delta+A_m(\theta-\theta_w)^3+B_m(\theta-\theta_w)^n
$$
for $\theta<\theta_{a,m}+\tau_m$, where $n$, $\theta_w$, and
$\rho_0$ are universal constants.

Then, for any fixed $\e,\delta>0$, we obtain the existence of a
solution $(\rho^{\e,\delta},r^{\e,\delta})$ to the free boundary
problem by the standard fixed point argument. Moreover, we have
$r^{\e,\delta}\in C^{1+\alpha}([\theta_w,\theta_1])$ for $\alpha\leq
\alpha_1$. This completes the proof.
\end{proof}

\subsection{Proof of Theorem \ref{4}: Completion}

  We note that Lemma \ref{12} implies that there exists a solution
$(\rho^{\e,\delta},r^{\e,\delta})$ such that
$r^{\e,\delta}\in\mathcal{K}^{\e,\delta}$. From Lemma \ref{11} and
the interior Schauder estimate, we note that
$\|\rho^{\e,\delta}\|_{C^{2,\alpha}_{loc}}\leq C$, and
$\rho^{\e,\delta}$ satisfies property \eqref{3.21}. By Lemma
\ref{2}, we have $c^2(\rho^{\e,\delta})\geq r^2$. This completes the
proof.

\section{Proof of Theorem \ref{2.1}: Existence of Solutions}

In this section, we study the limiting solution, as the elliptic
regularization parameter $\e$ and the oblique derivative boundary
regularization parameter $\delta$ tend to $0$. We start with the
regularized solutions of problem \eqref{3.1} and
\eqref{3.3}--\eqref{3.5}, whose existence is guaranteed by Theorem
\ref{4}. Denote by $\rho^{\e,\delta}$ a sequence of the regularized
solutions of the boundary value problem.

For the solutions of the regularized problems,
we first construct a uniform lower barrier to obtain the uniform
ellipticity in any compact domain contained by $\overline{\Omega}\backslash\Gamma_{\text{\rm sonic}}$, and also away from the points
in $\Gamma_{\text{\rm shock}}$ where their distance to the sonic circle of right
state $C_0$ goes to zero as $\delta$ tends to zero.

\begin{lemma}\label{15}
There exists a positive function $\varphi$, independent of $\e$ and
$\delta$, such that
$$
c^2(\rho^{\e,\delta})-(\xi^2+\eta^2)\geq \varphi \qquad \mbox{in
$\,\, \overline{\Omega}\backslash\Gamma_{\text{\rm
sonic}}$,}
$$
and
$$
\varphi \to 0
$$
as $\text{\rm dist}((\xi,\eta),\Gamma_{\text{\rm sonic}})\rightarrow0$,
or $\max\{\,\text{\rm dist}((\xi,\eta),\Gamma_{\text{\rm shock}}),\,\,\text{\rm dist}((\xi,\eta),C_0)\,\}\rightarrow0$.
\end{lemma}

\begin{proof}
For $0<R<1$ and $X_0=(\xi_0,\eta_0)\in \tilde{\Omega}$, as in
\cite{sbe2}, let
$$
\zeta(X)=1-\frac{(\xi-\xi_0)^2+(\eta-\eta_0)^2}{R^2}\qquad\mbox{for
$B_R(X_0)\cap \Gamma_{\text{sonic}}=\emptyset$}.
$$
We define
$$
\varphi(\xi,\eta)=\delta_0\big(\zeta(X)\big)^\tau,
$$
where $\delta_0$ and $\tau$ are two positive constants to be
specified later, and $0\leq \zeta\leq 1$ independent of $\e$. Then,
piecing together these $B_{\frac{3}{4}R_{X_0}}(X_0)$, $X_0\in
\overline{\Omega}\backslash\Gamma_{\text{sonic}}$, we can obtain a
local uniform lower barrier of $c^2(\rho^{\e,
\delta})-(\xi^2+\eta^2)$. That is,
$$
c^2(\rho^{\e,\delta})-\xi^2-\eta^2\geq
\varphi=\delta_0\zeta^{\tau}\qquad\text{in}\
B_{\frac{3}{4}R_{X_0}}(X_0)\cap\overline{\Omega_{\e}},
$$
where $\delta_0$ and $\tau$ are independent of $\e$ (though they may
depend on $R$). Moreover, $\delta_0$ tends to $0$ {as
$\text{dist}((\xi,\eta),\Gamma_{\text{sonic}})\rightarrow0$,
or $\max\{\,\text{\rm dist}((\xi,\eta),\Gamma_{\text{\rm shock}}),\,\,\text{\rm dist}((\xi,\eta),C_0)\,\}\rightarrow0$, so does
$\varphi$}. See \cite{sbe2} for more details.
\end{proof}

The proof of Lemma \ref{15} also implies that we can obtain the
uniform ellipticity of \eqref{3.1}, which is independent of $\e$ in
$B_{\frac{3}{4}R_{X_0}}(X_0)\cap\overline{\Omega_{\e}}$.

The uniform lower bound of $c^2-\xi^2-\eta^2$ independent of $\e$
implies that the governing equation \eqref{3.1} is locally uniformly
elliptic, independent of $\e$ and $\delta$, which allows us to apply
the standard local compactness arguments to obtain the limit $\rho$
locally in the interior of the domain.

We first consider the behaviour of shock position $r^{\e,\delta}$,
as $\e$ and $\delta$ tend to $0$. We divide the shock position into
three cases:

\smallskip
{\it Case 1}: $c(\rho_0)<r(\theta)\leq \bar{c}(\rho_1, \rho_0)$ for
all $\theta\in[\theta_w,\theta_1)$ and
$r'(\theta)=r\sqrt{\frac{r^2-\bar{c}^2}{\bar{c}^2}}$;

\smallskip
{\it Case 2}: $r(\theta_w)=c(\rho_0)$ and $c(\rho_0)<r(\theta)\leq
\bar{c}(\rho_1,\rho_0)$,
$r'(\theta)=r\sqrt{\frac{r^2-\bar{c}^2}{\bar{c}^2}}$ for all
$\theta\in(\theta_w,\theta_1)$;

\smallskip
{\it Case 3}: There exists  $\theta_a\in(\theta_w,\theta_1)$ such
that $r(\theta)\equiv c(\rho_0)$ for $\theta\in[\theta_w,\theta_a]$,
$r(\theta)>c(\rho_0)$, and
$r'(\theta)=r\sqrt{\frac{r^2-\bar{c}^2}{\bar{c}^2}}$ for
$\theta\in(\theta_a,\theta_1)$.

\begin{lemma}\label{16}
There exist functions $r(\theta)\in C^1([\theta_w,\theta_1])$ and
$\rho\in C_{loc}^{2+\alpha}(\Omega)\cap C(\overline{\Omega})$,
satisfying one of the three cases stated above, such that
$$
r^{\e,\delta}\rightarrow r \,\,\,\,\,\mbox{in $\,\,\,\, C([\theta_w,\theta_1])$},
\qquad\,\, \rho^{\e,\delta}\rightarrow\rho\,\,\,\,\,\mbox{in \,\,\,
$C^{2+\alpha}_{loc}$},
$$
and $(\rho,r)$ is a solution of the free boundary problem
\eqref{3.1}--\eqref{3.5}.
\end{lemma}

\begin{proof}
For $\e,\delta>0$, it follows from Lemma \ref{12} that
$$
r^{\e, \delta}\in C^{1+\alpha}([\theta_w,\theta_1]), \qquad \|r^{\e,
\delta}\|_{C^{1}([\theta_w,\theta_1])}\leq C,
$$
where $C$ is independent of $\e$ and $\delta$. Thus, by the
Ascoli-Arzela theorem, there exists a subsequence converging
uniformly to a function $r(\theta)$ in
$C^{\alpha}([\theta_w,\theta_1])$ as $\e,\delta\rightarrow0$ for any
$\alpha<1$. By the local ellipticity ({\it cf.} Lemma \ref{15}) and
the standard interior Schauder estimate, there exists a function
$\rho\in C^{2+\alpha}_{loc}$ such that
$\rho^{\e,\delta}\rightarrow\rho$ in any compact subset contained by
$\bar{\Omega}\backslash(\Gamma_{\text{sonic}}\cup
\Gamma_{\text{shock}})$, satisfying $Q\rho=0$ in $\Omega$.

For $(r(\theta_0),\theta_0)\in\Gamma_{\text{sock}}$ with
$r(\theta_0)>c(\rho_0)$, there exist a neighborhood of $\theta_0$
and a constant $\delta^{\star}>0$ independent of $\e$ and $\delta$
such that $r^{\e,\delta}\geq c(\rho_0)+\delta^{\star}$  for  $\e$
and $\delta$  small enough. It follows from $c(\rho^{\e,\delta})\geq
r^{\e,\delta}\geq c(\rho_0)+\delta^{\star}$ that
$$
\rho^{\e,\delta}>\rho_0+\delta^{\star}.
$$
Thus, we obtain the uniform ellipticity locally, as well as the
uniform negativity of  $\mathbf{\beta}\cdot\nnu$ locally. Hence, we
can pass the limit to obtain $\rho\in C^{1+\alpha}$ and
\begin{eqnarray*}
M\rho=0 \qquad \mbox{on $\Gamma_{\text{shock}}$ near
$(r(\theta_0),\theta_0)$}
\end{eqnarray*}
such that $r'(\theta)=\frac{r}{\bar{c}}\sqrt{r^2-\bar{c}^2}$.

\smallskip
Then the remainder is to show the case that
$(r(\theta_0),\theta_0)\in\Gamma_{\text{shock}}$ and
$r(\theta_0)=c(\rho_0)$.

First, it follows from Lemma \ref{13} that
$$
c(\rho_0)\leq r^{\e,\delta}(\theta)\leq
c(\rho^{\e,\delta}(r^{\e,\delta}(\theta),\theta))\leq
c(\rho^{\e,\delta}(r^{\e,\delta}(\theta_0),\theta_0))
$$
for $\theta\in[\theta_w,\theta_0]$, and
$$
c(\rho_0)\leq
\bar{c}(\rho^{\e,\delta}(r^{\e,\delta}(\theta_0),\theta_0),
\rho_0)\leq r^{\e,\delta}(\theta_0).
$$
Thus,
$$
\rho^{\e,\delta}(r^{\e,\delta}(\theta_0),\theta_0)\rightarrow\rho_0.
$$
Therefore, $r(\theta)\equiv c(\rho_0)$ for
$\theta\in[\theta_w,\theta_0]$.

\smallskip
Next we prove the continuity of solutions up to the boundary where
$r(\theta)=c(\rho_0)$. First, we prove that $r\in C^1$. Still from
Lemma \ref{13}, we obtain that
$$
\rho(r(\theta),\theta)\rightarrow\rho_0 \qquad\mbox{if
$\theta\rightarrow\theta_0$ from the right}.
$$
On the other hand,
$$
r'(\theta)=r(\theta)\sqrt{\frac{r^2(\theta)-\bar{c}^2}{\bar{c}^2}}\qquad
\text{for $\theta>\theta_0$},
$$
which implies that $r'(\theta)\rightarrow 0$ as
$\theta\rightarrow\theta_0$ from the right-hand side, and it holds
obviously from the left-hand side. If we define $r'(\theta_0)=0$,
then $r\in C^1$.

Note that the equation for $u=c^2(\rho)$ is
\begin{eqnarray}
Q(u)&:=&(c^2-r^2+\e)u_{rr}+\frac{c^2+\e}{r^2}u_{\theta\theta}
+\frac{c^2+(\gamma-2)(r^2-\e)}{(\gamma-1)
c^2}(u_r)^2  \label{4.11a}\\
& & +\frac{1}{(\gamma-1)
r^2}(u_{\theta})^2+\frac{c^2-2r^2+\e}{r}u_r  \nonumber \\[2mm]
& =& 0. \nonumber
\end{eqnarray}

We prove the most complicated case $\theta_0=\theta_a$ first, and
the other cases will be discussed later.

We construct a family of barrier functions $\{\Psi_{\tau}\}$ with
parameter $\tau$. For any $m>0$, there exists $\delta_1(m)>0$ such
that $r'(\theta)<m$ for $|\theta-\theta_a|<\delta_1(m)$. This
implies that
$$
|r(\theta)-r(\theta_a)|<m\delta_1(m) \qquad \mbox{for}\,\,\,\,
|\theta-\theta_a|<\delta_1(m),
$$
where $\delta_1(m)\rightarrow 0$ as $m\rightarrow 0$.

Let $m<1$ and $m\delta(m)=\frac{\tau}{2}$ ($\tau$ will be specified
later). We have
$$
\begin{array}{lll}
\rho_0\leq\rho(r(\theta),\theta)
&\leq& \rho(r(\theta_a+\delta_1(m))
,\theta_a+\delta_1(m)) \leq
(\bar{c}_{\rho_0})^{-1}(r(\theta_a+\delta_1(m)))\\[1.5mm]
 &\leq& \rho_0+\frac{Cm}{2}.
 \end{array}
$$
For $\e,\delta$ small enough, we obtain
$$
\rho_0\leq\rho^{\e,\delta}(r^{\e,\delta}(\theta),\theta)
\leq\rho_0+Cm
$$
and
$$
0\leq r^{\e,\delta}(\theta)-c(\rho_0)<\tau \qquad \mbox{for
$|\theta-\theta_a|<\delta_1(m)$},
$$
where $C$ depends only on $\gamma$ and $\rho_0$. We define
$$
\Psi_{\tau}=\Psi_{\tau}^{\e,\delta}
=c^2(\rho_0+Cm)+A(c(\rho_0)+\tau-r)^{\alpha}+B(\theta-\theta_a)^2
$$
in
$$
Q^{\e,\delta}=\{(r,\theta)\,
 :\, |r-c(\rho_0)|\leq\delta_2,|\theta-\theta_a|\leq\delta_1(m)\}\cap\Omega^{\e,\delta},
$$
where $\delta_2>\tau$ will be chosen later.

Choose
$$
B=\frac{c^2(\rho_1)-c^2(\rho_0)}{\delta_1^2(m)}, \qquad A=
A_1=\frac{c^2(\rho_1)-c^2(\rho_0)}{\delta_2^{\alpha}}.
$$
Since $\rho_0\leq\rho^{\e,\delta}\leq\rho_1$ and
$\rho^{\e,\delta}\leq\rho_0+Cm$ on
$\Gamma_{\text{shock}}^{\e,\delta}\cap\partial Q^{\e,\delta}$, we
have
$$
\Psi_{\tau}\geq c^2(\rho^{\e,\delta}) \qquad\mbox{on}\,\, \partial
Q^{\e,\delta}.
$$
Thus, we have
\begin{equation}\label{4.11}
\begin{array}{lll}
Q(\Psi^{\e,\delta}_{\tau})
&=&A\alpha(\alpha-1)\big(c^2(\rho^{\e,\delta})-r^2\big)\big(c(\rho_0)-r+\tau\big)^{\alpha-2}
+\frac{2Bc^2}{r^2}\\[2mm]
&&+\big(1-\frac{\gamma-2}{(\gamma-1) c^2}(c^2-r^2)\big)A^2\alpha^2
\big(c(\rho_0)-r+\tau\big)^{2\alpha-2}\\[2mm]
&&+\frac{4B^2}{(\gamma-1) r^2}(\theta-\theta_a)^2
 -\frac{A\alpha(c^2-2r^2)}{r}\big(c(\rho_0)-r+\tau\big)^{\alpha-1}.
\end{array}
\end{equation}

Consider \eqref{4.11} in $Q^{\e,\delta}\cap\{(r,\theta)\,:\,
c^2(\rho^{\e,\delta})-\Psi^{\e,\delta}_{\tau}\geq0\}$. Since
$$
c(\rho_{0}+Cm)\geq
c(\rho^{\e,\delta}(r^{\e,\delta}(\theta),\theta))\geq
r^{\e,\delta}(\theta)\geq r,
$$
we have
\begin{equation*}
\begin{array}{lll}c^2(\rho^{\e,\delta})-r^2+\e
&\geq& c^2(\rho^{\e,\delta})-\Psi^{\e,\delta}_{\tau}
+c^2(\rho_0+Cm)-r^2+A\big(c(\rho_0)+\tau-r\big)^{\alpha}\\[2mm]
&\geq& A\big(c(\rho_0)+\tau-r\big)^{\alpha}.
\end{array}
\end{equation*}
For $\alpha<1$, \eqref{4.11} implies
\begin{equation*}
\begin{array}{lll}
Q(\Psi^{\e,\delta}_{\tau}) &\leq
&A^2\alpha\Big(\big(2-\frac{\gamma-2}{(\gamma-1)
c^2}(c^2-r^2)\big)\alpha-1\Big)
\big(c(\rho_1)-r+\tau\big)^{2\alpha-2}\\[2mm]
&&-\frac{c^2-2r^2}{r}A\alpha\big(c(\rho_1)-r+\tau\big)^{\alpha-1}
  +\frac{4}{(\gamma-1) r^2}B^2(\theta-\theta_a)^2+2B\frac{c^2}{r^2}.
\end{array}
\end{equation*}

Moreover, let
$$
A>A_1, \qquad B=\frac{C(\rho_0,\rho_1)}{\delta_1^2(m)}.
$$
If $\alpha<\frac{1}{2+C(\rho_0,\rho_1,\gamma)}$ and $\delta_2+\tau$
is small enough, we have
$$
\begin{array}{lll}
Q(\Psi^{\e,\delta}_{\tau})&\leq&
C(\rho_0,\rho_1)\big((2+C(\rho_0,\rho_1,\gamma))\alpha-1\big)
A^2\big(c(\rho_1)+\tau-r\big)^{2\alpha-2}\\[1.5mm]
&&+\frac{C(\rho_1,\rho_2)}{\delta_1^2(m)}.
\end{array}
$$
Then there exists a constant $A_2(\delta_2,m,\rho_0,\rho_1)$ such
that
$$
Q(\Psi^{\e,\delta}_{\tau})\leq 0 \qquad \mbox{for}\,\,\, A>A_2.
$$
In fact, if $r<c(\rho_0)$, we choose $\delta_2=\sqrt{m}\delta_1(m)$
to obtain
$$
c(\rho_0)-r+\tau\leq 2\sqrt{m}\delta_1(m),
$$
and let
$$
A_2^{(1)}=\frac{C(\rho_0,\rho_1,\alpha)m^{\frac{1-\alpha}{2}}}{\delta_1^{\alpha}(m)},
\qquad
A_1=\frac{C(\rho_0,\rho_1)}{m^{\frac{\alpha}{2}}\delta^{\alpha}_1(m)}.
$$
If $r\geq c(\rho_0)$,
$$
c(\rho_0)+\tau-r\leq\tau,
$$
and we let
$$
A_2^{(2)}=\frac{C(\rho_0,\rho_1,\alpha)m^{1-\alpha}}{\delta_1^{\alpha}(m)}.
$$
Set $ A=\max\{A_1,A^{(1)}_2,A^{(2)}_2\}. $ Then
$\rho^{\e,\delta}\leq\Psi^{\e,\delta}_{\tau}$ in $Q^{\e,\delta}$.
Passing to the limits $\delta,\e\rightarrow0$, we obtain
$$
\rho\leq\Psi_{\tau} \qquad\mbox{in the domain}\,\,\,
Q(m,\sqrt{m}\delta_1(m)):=\cap_{\delta,\e>0}Q^{\e,\delta}.
$$

With these barrier functions, we can show that $\rho$ is continuous
at $(r(\theta_a),\theta_a)$. In fact, for every $\e_1>0$, there
exists $m>0$ such that
$$
c^2(\rho_0+Cm)-c^2(\rho_0)<\frac{\e_1}{3}.
$$
For this $m$, we can choose $A$, $B$, and $\tau$ such that
$$
c^2(\rho)\leq\Psi_{\tau}\leq\frac{\e_1}{3}+c^2(\rho_1)
+A\big(c(\rho_0)-r+\tau\big)^{\alpha}+B(\theta-\theta_a)^2.
$$
Choose the neighborhood of $(r(\theta_a),\theta_a)$ small enough so
that
$$
A(c(\rho_0)-r+\tau)^{\alpha}\leq A(2\tau)^{\alpha} \leq
C(\rho_0,\rho_1,\alpha)m^{\frac{\alpha}{2}}.
$$
Then, choosing $m$ small again, we have
$$
c^2(\rho)\leq\frac{2\e_1}{3}+c^2(\rho_0)+B(\theta-\theta_a)^2.
$$
Finally, we choose a small neighborhood such that
$$
c^2(\rho_0)\leq c^2(\rho)\leq\e_1+c^2(\rho_0).
$$
Thus, we obtain our claim that $\rho$ is  continuous at
$(r(\theta_a),\theta_a)$, that is, the results hold for this case.

\smallskip
As for the case $\theta\in[\theta_w,\theta_a)$, we can choose
arbitrary $\tau>0$, which is independent of the neighborhood of
$\theta$. This fact makes the similar proof of this case much easier
for all sufficiently small $\e$ and $\delta$, and we omit the
details here.
\end{proof}

Next, we discuss the wave strength at the sonic circle $r\equiv
c(\rho_0)$ and conclude that Case 3 in Lemma \ref{16} does not
actually occur.

\begin{lemma}\label{5}
Let $r(\theta)$ be monotone increasing in $\theta$ on
$\Gamma_{\text{\rm shock}}$ and $\rho>\rho_0$ in the subsonic
region. Then $r(\theta)>c(\rho_0)$ for $\theta_w<\theta\leq
\theta_1$.
\end{lemma}

\begin{proof}We divide the proof into five steps.

\smallskip
1. We show our claim by contradiction. More precisely, if there
exists $\bar{\theta}$ such that $r(\bar{\theta})=c(\rho_0):=c_0$.
Then, using the monotonicity of $r(\theta)$,
$$
r(\theta)\equiv c_0\qquad\mbox{for $\theta_w\leq\theta\leq
\bar{\theta}$}.
$$

2. For $\theta_0\in[\theta_w,\bar{\theta}]$, we define
$$
w_1=c_0^{2}+A_1(c_0-r)^{\frac{1}{2}}-B_1(c_0-r)^{\beta_1}+D_1(\theta-\theta_0)^2,
$$
where  $A_1, B_1,D_1>0$ and $\frac{1}{2}<\beta_1<1$, which will be
specified later to prove that $\rho\in C^{\frac{1}{2}}$ near this
boundary point.

\smallskip
Using \eqref{4.11a} with the coefficient of $u_{rr}$ replaced by
$u-r^2$, we have
\begin{equation}
\begin{array}{lll}
\hat{Q}(w_1)&=&\Big(-(\beta^2_1-\frac{1}{4})A_1B_1(c_0-r)^{\beta_1-\frac{3}{2}}+O_1\Big)\\[2mm]
&& +\Big(\beta_1(2\beta_1-1)B^2_1
(c_0-r)^{2\beta_1-2}+O_2\Big)\\[2mm]
&& -\frac{(\gamma-2)}{4(\gamma-1) c^2}A^2_1(c^2-r^2)(c_0-r)^{-1}\\[2mm]
&&+\Big(-\frac{A_1D_1}{4}(c_0-r)^{-\frac{3}{2}}(\theta-\theta_0)^2+O_4\Big),
\end{array}
\end{equation}
where
\begin{equation*}
\begin{array}{lll}
O_1&=&-\frac{A_1}{2}c_0(c_0-r)^{-\frac{1}{2}}-2c_0\beta_1(\beta_1-1)B_1(c_0-r)^{\beta_1-1}
+\frac{A_1}{4}(c_0-r)^{\frac{1}{2}}\\[1.5mm]
&&+\beta_1(\beta_1-1)B_1(c_0-r)^{\beta_1}+\frac{2c^2}{r^2}D_1+\frac{A_1r}{2}(c_0-r)^{-\frac{1}{2}}
-\beta_1 B_1r(c_0-r)^{\beta_1-1}\\[1.5mm]
&&+\frac{(\gamma-2)\beta_1 A_1B_1}{(\gamma-1) c^2}(c^2-r^2)(c_0-r)^{\beta_1-\frac{3}{2}}-\frac{A_1}{2r}(c^2-r^2)(c_0-r)^{\frac{1}{2}}\\[1.5mm]
&&+\frac{\beta_1 B_1}{r}(c^2-r^2)(c_0-r)^{\beta_1-1},\\[1.5mm]
O_2&=&-\frac{(\gamma-2)\beta_1^2B^2_1}{(\gamma-1) c^2}(c^2-r^2)(c_0-r)^{2\beta_1-2},\\[2mm]
O_4&=&\beta_1(\beta_1-1)B_1D_1(c_0-r)^{\beta_1-2}(\theta-\theta_0)^2+\frac{4D_1}{(\gamma-1)
r^2}(\theta-\theta_0)^2.
\end{array}
\end{equation*}

Notice that there exists $0<\alpha<\frac{1}{2}$ such that
$c^2-r^2\leq (c_0-r)^{\alpha}$ for $c_0-r>0$ small. Thus,
$$
\Big|\frac{(\gamma-2)}{4(\gamma-1)
c^2}A^2_1(c^2-r^2)(c_0-r)^{-1}\Big|\leq
C(\rho_0,\rho_1)A^2_1(c_0-r)^{\alpha-1}.
$$
We can choose a proper constant $\alpha$ such that
$\beta_1-\frac{3}{2}<\alpha-1$, {\it i.e.},
$\alpha>\beta_1-\frac{1}{2}$.

On one hand, let $c_0-r>0$ be small enough so that
$$
 (\beta^2_1-\frac{1}{4})A_1B_1(c_0-r)^{\beta_1-\frac{3}{2}}
 >3C(\rho_0,\rho_1)A^2_1(c_0-r)^{\alpha-1},
$$
which implies \begin{equation}\label{6.3}
B_1>\frac{3C(\rho_0,\rho_1)}{\beta^2_1-\frac{1}{4}}A_1(c_0-r)^{\alpha-\beta_1+\frac{1}{2}}:=
A_1C(\rho_0,\rho_1,\beta_1)(c_0-r)^{\alpha-\beta_1+\frac{1}{2}}.
\end{equation} On the other hand, if $c_0-r>0$ is sufficiently small, we have
$$
(\beta^2_1-\frac{1}{4})A_1B_1(c_0-r)^{\beta_1-\frac{3}{2}}
>3\beta_1(2\beta_1-1) B^2_1(c_0-r)^{2\beta_1-2},
$$
which implies \begin{equation}\label{6.4}
A>\frac{(2\beta^2_1-\beta_1)B_1}{\beta^2_1-\frac{1}{4}}(c_0-r)^{\beta_1-\frac{1}{2}}
:= C(\beta_1)B_1(c_0-r)^{\beta_1-\frac{1}{2}}. \end{equation}
Moreover, we have
$$
C(\rho_0,\rho_1,\beta_1)(c_0-r)^{\alpha-\beta_1+\frac{1}{2}}
<C(\beta_1)(c_0-r)^{\frac{1}{2}-\beta_1}
$$
when $r\in [\bar{r}, c_0]$, and $\bar{r}$ is close to $c_0$.

Choose proper constants $A_1$ and $D_1$ such that
$$
w_1>c^2_0+\frac{1}{2}A_1(c_0-r)^{\frac{1}{2}}+D_1(\theta-\theta_0)^2>c^2
$$
at the boundary of a relatively neighborhood $N_1$ of
$(c_0,\theta_0)$ to $\Omega$. Choose $B_1$ sufficiently small such
that
$$
\hat{Q}(w_1)<0
$$
and
$$
C(\rho_0,\rho_1,\beta_1)(c_0-r)^{\alpha-\beta+\frac{1}{2}}<\frac{B_1}{A_1}<\min\Big\{C(\beta),
\frac{1}{4(\beta_1-\beta^2_1)}\Big\}(c_0-r)^{\frac{1}{2}-\beta_1}
\qquad \mbox{in $N_1$}.
$$
This implies that \eqref{6.3} and \eqref{6.4} hold.

Obviously, we have
$$
\partial_{rr}w_1<0 \qquad\mbox{in}\,\, N_1
$$
if \eqref{6.3} and \eqref{6.4} hold. If  $S_1=\{(r,\theta)\in N_1\,
:\, c^2>w_1\}\neq\emptyset$, we have $Q(w_1)\leq\hat{Q}(w_1)<0$ in
$S_1$. Thus,
$$
0<Qu-Q(w_1).
$$
Using the maximum principle, $ u\leq w_1, $ which contradicts with
$c^2>w_1$. Thus
$$
c^2\leq w_1 \qquad\mbox{in}\,\, N_1.
$$

3. We define
$$
w_2=c_0^{2}+A_2(c_0-r)^{\frac{1}{2}}+B_2(c_0-r)^{\beta_2}-D_2(\theta-\theta_0)^2,
$$
where  $A_2,B_2,D_2>0$ and $\frac{1}{2}<\beta_2<1$, all of which
will be specified later to prove that $C^{\frac{1}{2}}$ is optimal.
Through a simple algebraic calculation, we have
\begin{equation}
\begin{array}{lll}
\hat{Q}(w_2)&=&\big((\beta^2_2-\frac{1}{4})A_2B_2(c_0-r)^{\beta_2-\frac{3}{2}}+\overline{O}_1\big)
+\big(\beta_2(2\beta_2-1)B^2_2(c_0-r)^{2\beta_2-2}+\overline{O}_2\big)\\[2mm]
&&+\big(-\beta_2(\beta_2-1)B_2D_2(c_0-r)^{\beta_2-2}(\theta-\theta_0)^2+\overline{O}_3\big)\\[2mm]
&& +\big(\beta_2B_2r(c_0-r)^{\beta_2-1}+\overline{O}_4\big)
 +\frac{1}{4}A_2D_2(c_0-r)^{-\frac{3}{2}}(\theta-\theta_0)^2\\[2mm]
&& +2\beta_2(\beta_2-1)B_2c_0(c_0-r)^{\beta_2-1},
\end{array}
\end{equation}
where
\begin{eqnarray*}
\overline{O}_1&=&-\frac{(\gamma-2)}{4(\gamma-1)
c^2}A^2_2(c^2-r^2)(c_0-r)^{-1}
-\frac{A_2}{2}c_0(c_0-r)^{-\frac{1}{2}}+\frac{A_2}{4}(c_0-r)^{\frac{1}{2}}
\\
&&-\beta_2(\beta_2-1)B_2(c_0-r)^{\beta_2}-\frac{2c^2}{r^2}D_2+\frac{A_2r}{2}(c_0-r)^{-\frac{1}{2}}
-\frac{A_2}{2r}(c^2-r^2)(c_0-r)^{\frac{1}{2}}
\\
&&-\frac{(\gamma-2)\beta_2 A_2B_2}{(\gamma-1) c^2}(c^2-r^2)(c_0-r)^{\beta_2-\frac{3}{2}},\\
\overline{O}_2&=&-\frac{(\gamma-2)\beta^2_2B^2_2}{(\gamma-1) c^2}(c^2-r^2)(c_0-r)^{2\beta_2-2},\\
\overline{O}_3&=&-\frac{4D_2}{(\gamma-1) r^2}(\theta-\theta_0)^2,\\
\overline{O}_4&=&-\frac{\beta_2 B_2}{r}(c^2-r^2)(c_0-r)^{\beta_2-1}.
\end{eqnarray*}

Let $D_2$ be large enough such that $c^2>w_2$ for some
$\theta=\theta_a,\theta_b$. We choose $\tilde{r}<c_0$ such that
\begin{equation*}
\begin{array}{lll}
c^2&>&c^2_0+2A_2(c_0-\tilde{r})^{\frac{1}{2}}-D_2(\theta-\theta_0)^2\\[1.5mm]
&\geq&c_0+A_2(c_0-\tilde{r})^{\frac{1}{2}}+B_2(c_0-\tilde{r})^{\beta_2}-D_2(\theta-\theta_0)^2.
\end{array}
\end{equation*}
The second inequality holds, provided that
$\frac{B_2}{A_2}\leq(c_0-{\tilde{r}})^{\frac{1}{2}-\beta_2}$.
Choosing $\beta_2>\frac{7}{8}$, we have
$$
\frac{1}{2}\beta_2B_2r(c_0-r)^{\beta_2-1}+2\beta_2(\beta_2-1)B_2c_0(c_0-r)^{\beta_2-1}
\leq0 \qquad\mbox{for}\,\,\, \frac{c_0}{2}<r<c_0,
$$
and
$$
\hat{Q}(w_2)>0.
$$
Then, if  $S_2=\{(r,\theta)\in N_1\, :\, c^2<w_2\}\neq\emptyset$, we
have
$$
Q(w_2)\geq\hat{Q}(w_2)>0 \qquad\mbox{in}\,\,\, S_2.
$$
Thus, $ Qu-Q(w_2)<0. $ Using the maximum principle, $c\leq w_2$,
which contradicts with $c^2<w_1$. Thus
$$
c^2\geq w_2\qquad\mbox{in}\,\,\, N_2.
$$

4. We now show that
$$
c^2>c_0^{2}+A_3(c_0-r)^{\frac{1}{2}}+B_3(c_0-r)^{\beta_3}=:w_3
$$
in a relative neighborhood of $(r_0,\theta_0)$, where $A_3$ and
$B_3$ are positive constants to be specified later, so that the
$C^{\frac{1}{2}}$--regularity is optimal.

Since $c^2\geq w_2$, we can  choose $\bar{\theta}_a$ and
$\bar{\theta}_b$ such that
$$
c^2\geq
c^2_0+A_2(\overline{\theta}_a,\overline{\theta}_a)(c_0-r)^{\frac{1}{2}}+B_2(c_0-r)^{\beta_2}
\qquad \text{for}\ N_3\subset N_2.
$$
Thus, there exist positive constants $A_3$, $B_3$, and $\beta_3$
such that
$$
w_3\leq c^2.
$$
It is easy to see that
\begin{eqnarray*}
\hat{Q}(w_3)
&=&\big((\beta^2_3-\frac{1}{4})A_3B_3(c_0-r)^{\beta_3-\frac{3}{2}}+{\tilde{O}}_1\big)
+\big(\beta_3(2\beta_3-1)B^2_3
(c_0-r)^{2\beta_3-2}+{\tilde{O}}_2\big)\\
&&+\big(\beta_3B_3r(c_0-r)^{\beta_3-1}+{\tilde{O}}_4\big)
+2\beta_3(\beta_3-1)B_3c_0(c_0-r)^{\beta_3-1},
\end{eqnarray*}
where
\begin{eqnarray*}
{\tilde{O}}_1&=&-\frac{(\gamma-2)}{4(\gamma-1)
c^2}A^2_3(c^2-r^2)(c_0-r)^{-1}-\frac{A_3}{2}c_0(c_0-r)^{-\frac{1}{2}}+\frac{A_3}{4}(c_0-r)^{\frac{1}{2}}
\\
&&+\frac{A_3r}{2}(c_0-r)^{-\frac{1}{2}}
-\frac{A_3}{2r}(c^2-r^2)(c_0-r)^{\frac{1}{2}}
\\
&&-\frac{(\gamma-2)\beta_3 A_3B_3}{(\gamma-1) c^2}(c^2-r^2)(c_0-r)^{\beta_3-\frac{3}{2}},\\
{\tilde{O}}_2&=&\beta_3(\beta_3-1)B_3(c_0-r)^{\beta_3}
-\frac{(\gamma-2)\beta^2_3B^2_3}{(\gamma-1) c^2}(c^2-r^2)(c_0-r)^{2\beta_3-2},\\
{\tilde{O}}_3&=&-\frac{\beta_3 B_3}{r}(c^2-r^2)(c_0-r)^{\beta_3-1}.
\end{eqnarray*}
Similarly, we can show that $c^2\geq w_3$ in $N_3$.

Thus, $\frac{1}{2}A_3(c_0-r)^{\frac{1}{2}}\leq c^2-c^2_0\leq 2A_1
(c_0-r)^{\frac{1}{2}}$ in $N_1\cap N_3$. This implies
$$
a(c_0-r)^{\frac{1}{2}}\leq v:=\rho-\rho_0\leq
A(c_0-r)^{\frac{1}{2}}\qquad\mbox{in $N_1\cap N_3$}
$$
for some constants $a$ and $A$, so the optimal regularity of $\rho$
is $C^{\frac{1}{2}}$ near the sonic circle.

\smallskip
5. We introduce the coordinates: $ (x, y)=(c_{0}-r,\
\theta-\theta_{w}) $ and set $ v=c^2-c^2_0. $ Thus, rewriting the
equation for $c^2$ in the divergence form, we have \begin{equation}
Qv=\big(a_{11}(v+2c_0x-x^2)v_{x}\big)_{x}+b_1v_{x}+(a_{22}v_{y})_{y}=0,
\end{equation} where
$a_{11}=\frac{c^{\frac{2(2-\gamma)}{\gamma-1}}}{\gamma-1}$,
$a_{22}=\frac{c^{\frac{2}{\gamma-1}}}{(\gamma-1) r^2}$, and
$b_1=\frac{c^{\frac{2}{\gamma-1}}}{(\gamma-1) r}$.

Scale $v$ in $N_1\cap N_3$ by defining \begin{equation}
u(S,T)=\frac{1}{S^{\frac{1}{5}}}v(S^{-\frac{12}{5}},y_{0}+S^{-\frac{14}{5}}T),
\end{equation} for $(S^{-\frac{12}{5}},y_{0}+S^{-\frac{14}{5}}T)\in N_1\cap
N_3$. Then $u$ satisfies the following governing equation:
\begin{equation}\label{4.9} Qu=(\tilde{a}_{11}u_{S})_S+(\tilde{a}_{12}u_{T})_S
+(\tilde{a}_{21}u_{S})_T+(\tilde{a}_{22}u_{T})_T+(\tilde{b}_{2}u)_T
+\tilde{c}_{1}u_{S}+\tilde{c}_{2}u_{T}+\tilde{d}_{2}u=0,
\end{equation}
where
\begin{eqnarray*}
\tilde{a}_{11}&=&a_{11}\big(S^{\frac{7}{5}}u+2c_0S^{-\frac{6}{5}}-S^{-\frac{18}{5}}\big),\\
\tilde{a}_{12}&=&\frac{14T}{5S}a_{11}\big(S^{\frac{7}{5}}u+2c_0S^{-\frac{6}{5}}-S^{-\frac{18}{5}}\big),\\
\tilde{a}_{21}&=&\frac{14T}{5S}a_{11}\big(S^{\frac{7}{5}}u+2c_0S^{-\frac{6}{5}}-S^{-\frac{18}{5}}\big),\\
\tilde{a}_{22}&=&\frac{144}{25}a_{22}+\frac{189T^2}{25S^2}a_{11}
  \big(S^{\frac{7}{5}}u+2c_0S^{-\frac{6}{5}}-S^{-\frac{18}{5}}\big),\\
\tilde{b}_{2}&=&\frac{14T}{25S^2}a_{11}\big(S^{\frac{7}{5}}u+2c_0S^{-\frac{6}{5}}-S^{-\frac{18}{5}}\big)
   =\tilde{b}_{22}TS^{-2},\\
\tilde{c}_{1}&=&\frac{S^{\frac{7}{5}}u(2-\gamma)S^{-\frac{11}{5}}}{5(\gamma-1)^2
}c^{\frac{2(3-2\gamma)}{\gamma-1}}\big(S^{\frac{7}{5}}u+2c_0S^{-\frac{6}{5}}-S^{-\frac{18}{5}}\big)\\
&& -\frac{a_{11}(4c_0-S^{-\frac{12}{5}})}{5S^{\frac{11}{5}}}-\frac{12b_1}{5S^{\frac{11}{5}}}=\tilde{c}_{11}S^{-\frac{11}{5}},\\
\tilde{c}_{2}&=&\frac{168T}{25S^2}a_{11}\big(S^{\frac{7}{5}}u+2c_0S^{-\frac{6}{5}}-S^{-\frac{18}{5}}\big)
-\frac{189Ta_{11}}{25S^{2}}-\frac{168b_1T}{25S^{\frac{16}{5}}}=\tilde{c}_{22}S^{-2}T,\\
\tilde{d}&=&\frac{S^{\frac{7}{5}}u(2-\gamma)S^{-\frac{16}{5}}}{25(\gamma-1)^2
}c^{\frac{2(3-2\gamma)}{\gamma-1}}\big(S^{\frac{7}{5}}u+2c_0S^{-\frac{6}{5}}-S^{-\frac{18}{5}}\big)\\
&&-\frac{13a_{11}(2c_0-S^{-\frac{12}{5}})}{25S^{\frac{16}{5}}}
-\frac{12a_{11}S^{\frac{7}{5}}u(c_0+\frac{6}{5}S^{-\frac{6}{5}})}{25S^{\frac{23}{5}}}-\frac{12b_1}{25S^{\frac{16}{5}}}
=\tilde{d}_{1}S^{-\frac{16}{5}}.
\end{eqnarray*}
From the optimal continuity,
$$
0<a\leq S^{\frac{7}{5}}u\leq A,
$$
we have
$$
0<C^{-1}\leq
\lambda_1,\lambda_2,\tilde{b}_{22},\tilde{c}_{11},\tilde{c}_{22},\tilde{d}_1\leq
C
$$
if $S^{-1}$ and $T$ are sufficiently small. Here $\lambda_1$ and
$\lambda_2$ are the eigenvalues of the matrix
$(\tilde{a}_{ij})_{2\times 2}$, so the equation is uniformly
elliptic for $u$ in the $(S,T)$-coordinates.

Let $x_0^{-1}<S\leq x^{-\frac{5}{4}}_0$  with $x_0$ small enough.
Then, using Theorem 8.20 in \cite{gt}, we have
\begin{eqnarray*}
ax^{\frac{7}{5}}_0 \leq u(x_0^{-1},0) &\leq&
\sup\limits_{x_0^{-1}\leq
S\leq x^{-{5}/{4}}_0}u(S,T)\\
&\leq&  C\inf\limits_{x_0^{-1}\leq S\leq x^{-{5}/{4}}_0}u(S,T)\leq
Cu(x^{-\frac{5}{4}}_0,0)\leq CAx^{\frac{7}{4}}_0,
\end{eqnarray*}
where $C\leq C(n)^{(\frac{\Lambda}{\lambda}+\nu R)}$ in \cite{gt} is
independent of $x_0$, since
$(\Lambda,\lambda)=(\lambda_1,\lambda_2)$,
$R=x_0^{-\frac{5}{4}}-x_0^{-1}\leq x_0^{-\frac{5}{4}}$, and
$\nu:=\max\limits_{x_0^{-1}\leq S\leq x^{-5/4}_0}
\{\tilde{b}_{2},\tilde{c}_{1},\tilde{c}_{2},\sqrt{\tilde{d}}\}\leq
Cx_0^{\frac{8}{5}}$. This implies that $x^{-\frac{7}{20}}_0\leq C$,
which is a contradiction if $x_0$ is sufficiently small. This
completes the proof.

\end{proof}

Next, we consider $\Gamma_{\text{shock}}$ in the
$(\xi,\eta)$--coordinates to obtain finer properties.

\begin{lemma} \label{convex}
For the free boundary $\Gamma_{\text{\rm shock}}=\{(\xi,\eta(\xi))\,
:\, \xi_w<\xi<\xi_1\}$ determined by \eqref{3.1}--\eqref{3.5},
$$
\eta(\xi)\in C^{2}([\xi_w,\xi_1));
$$
moreover, $\eta(\xi)$ is strictly convex for $\xi\in [\xi_w,
\xi_1)$.
\end{lemma}

\begin{proof}
We define \begin{equation}\label{3.66}
F(\xi,\eta)=\xi^2+\eta^2-r^2(\theta(\xi,\eta))=0\qquad\text{on}\
\Gamma_{\text{shock}}. \end{equation} It is easy to check that
$$
F_{\eta}=(2\eta-2rr'\theta_{\eta})|_{\xi=\xi_w}=2\eta(\xi_w)\neq 0.
$$
By the implicit function theorem, there exists $\eta=\eta(\xi)$ such
that \eqref{3.66} holds locally on $\Gamma_{\text{shock}}$ near
$\xi=\xi_w$. That is, there exists $\bar{\xi}>0$ such that
$(\xi,\eta(\xi))\in\Gamma_{\text{shock}}$ for
$\xi_w<\xi\leq\bar{\xi}$.

\smallskip
Recall that $\eta'(\xi)=f(\xi,\eta(\xi),\rho(\xi,\eta(\xi)))$. Then
$$
\eta''=f_{\xi}+f_{\eta}\eta'+f_{\rho} \rho' \qquad \mbox{for}\,\,\,
\xi\in(\xi_w,\bar{\xi}).
$$

Notice that
$$
\eta'(\xi)=f(\xi,\eta,\bar{c})=\frac{\xi\eta+\bar{c}\sqrt{\xi^2+\eta^2-\bar{c}^2}}{\xi^2-\bar{c}^2},
$$
then
\begin{eqnarray*}
&&f_{\xi}=\frac{\eta}{\xi^2-\bar{c}^2}+\frac{\bar{c}\xi}{(\xi^2-\bar{c}^2)\sqrt{\xi^2+\eta^2-\bar{c}^2}}
-\frac{2\xi(\xi\eta+\bar{c}\sqrt{\xi^2+\eta^2-\bar{c}^2})}{(\xi^2-\bar{c})^2};\\
&&f_{\eta}=\frac{\xi}{\xi^2-\bar{c}^2}+\frac{\eta\bar{c}}{(\xi^2-\bar{c}^2)\sqrt{\xi^2+\eta^2-\bar{c}^2}}.
\end{eqnarray*}
Thus, we have
\begin{align*}
f_{\xi}+f_{\eta}\eta'=&\frac{\eta}{\xi^2-\bar{c}^2}+\frac{\bar{c}\xi}{(\xi^2-\bar{c}^2)\sqrt{\xi^2+\eta^2-\bar{c}^2}}
-\frac{2\xi^2\eta}{(\xi^2-\bar{c})^2} -\frac{2\xi\bar{c}\sqrt{\xi^2+\eta^2-\bar{c}^2}}{(\xi^2-\bar{c})^2}\\
                      &+\frac{\xi^2\eta}{(\xi^2-\bar{c}^2)^2} +\frac{\xi\bar{c}\sqrt{\xi^2+\eta^2-\bar{c}}}{(\xi^2-\bar{c})^2}
                       +\frac{\xi\eta^2\bar{c}}{(\xi^2-\bar{c}^2)^2\sqrt{\xi^2+\eta^2-\bar{c}^2}} +\frac{\eta\bar{c}^2}{(\xi^2-\bar{c})^2}\\
                      =&\frac{\eta(\xi^2-\bar{c}^2-2\xi^2+\xi^2+\bar{c}^2)}{(\xi^2-\bar{c}^2)}
+\frac{\xi\bar{c}(\xi^2-\bar{c}^2-2\xi^2-2\eta^2+2\bar{c}^2+\xi^2+\eta^2-\bar{c}^2+\eta^2)} {(\xi^2-\bar{c}^2)^2\sqrt{\xi^2+\eta^2-\bar{c}^2}}\\
                      =&0.
\end{align*}

Therefore, the sign of $\eta''$ is determined entirely by the sign
of $f_{\rho}$ and $\rho'$. Note that $\rho$ is increasing,
$\rho'>0$, and $\frac{d\bar{c}^2}{d\rho}>0$. Moreover, we have
\begin{eqnarray}\label{3.67}
\quad \frac{\partial
f}{\partial\bar{c}^2}&=&\frac{-2\xi\eta\bar{c}\sqrt{\xi^2+\eta^2-\bar{c}^2}
+2\eta^2(\xi^2+\eta^2-\bar{c}^2)+(\xi^2+\eta^2)(\bar{c}^2-\eta^2)}
{\bar{c}\big(\xi\eta-\bar{c}\sqrt{\xi^2+\eta^2-\bar{c}^2}\big)^2\sqrt{\xi^2+\eta^2-\bar{c}^2}}\nonumber\\
&=&\frac{\big(\xi\bar{c}-\eta\sqrt{\xi^2+\eta^2-\bar{c}^2}\big)^2}{\bar{c}\big(\xi\eta-\bar{c}
\sqrt{\xi^2+\eta^2-\bar{c}^2}\big)^2\sqrt{\xi^2+\eta^2-\bar{c}^2}}.
\end{eqnarray}

If $\xi \eta \le 0$, it is clear from \eqref{3.67} that $
\frac{\partial f}{\partial\bar{c}^2}>0. $

If $\xi \eta>0$, from \eqref{3.67}, we have
\begin{eqnarray*}
\frac{\partial f}{\partial\bar{c}^2} &=&\frac{
(\xi^2+\eta^2)^2(\bar{c}^2-\eta^2)^2\big(\xi\eta+\bar{c}\sqrt{\xi^2+\eta^2-\bar{c}^2}\big)^2}
{\bar{c}\sqrt{\xi^2+\eta^2-\bar{c}^2}(\xi^2-\bar{c}^2)(\eta^2-\bar{c}^2)\big(\xi \bar{c}+\eta\sqrt{\xi^2+\eta^2-\bar{c}^2}\big)^2}\\
&=&\frac{
(\xi^2+\eta^2)^2\big(\xi\eta+\bar{c}\sqrt{\xi^2+\eta^2-\bar{c}^2}\big)^2}
{\bar{c}(\bar{c}^2-\xi^2)\sqrt{\xi^2+\eta^2-\bar{c}^2}\big(\xi
\bar{c}+\eta\sqrt{\xi^2+\eta^2-\bar{c}^2}\big)^2}
>0.
\end{eqnarray*}
These imply that $\eta=\eta(\xi)$ is strictly convex for $\xi \in
[\xi_w,\xi_1)$.

\end{proof}

\smallskip
Lemma \ref{convex} yields that problem \eqref{3.1}--\eqref{3.5} is
equivalent to the following free boundary problem in the
self-similar coordinates:
\begin{enumerate}
\item[(i)] Equation:
\begin{equation}\label{3.58}
L\rho=\sum_{i,j=1}^2D_i\big(a_{ij}(\xi,
\eta,\rho)D_j\rho\big) +\sum_{i=1}^2b_i(\xi,\eta)D_i\rho=0\qquad
\text{in}\,\, \Omega
\end{equation}
with
\begin{eqnarray*}
&& a_{11}(\xi,\eta,\rho)=c^{2}(\rho)-\xi^2, \quad
a_{22}(\xi,\eta,\rho)=c^{2}(\rho)-\eta^2,\\
&& a_{12}(\xi,\eta,\rho)=a_{21}(\xi,\eta,\rho)=-\xi\eta,\quad
b_1(\xi,\eta)=\xi,\quad b_2(\xi,\eta)=\eta.
\end{eqnarray*}

\item[(ii)] The shock equation:
$$
\frac{d\eta}{d\xi}=f(\xi,\eta,\rho)=\frac{\xi\eta+\bar{c}\sqrt{\xi^{2}+\eta^{2}-\bar{c}^2}}
{\xi^{2}-\bar{c}^2}\qquad\text{with}\ \eta(\xi_1)=\eta_1,
$$
with the boundary condition on $\Gamma_{\text{shock}}$:
\begin{equation}\label{3.58a} N\rho=\sum_{i=1}^2\beta_iD_i\rho =0\qquad
\text{on}\
 \Gamma_{\text{shock}}=\{\eta=\eta(\xi)\, :\, 0\leq\xi\leq\xi_1\},
 \end{equation}
where $\beta_1$ and $\beta_2$ are the following functions of
$(\xi,\eta)$, $\rho$, and $\eta'$:
\begin{eqnarray}\label{3.59}
\beta_{1}&=&(\xi^{2}+\eta^{2})(-\eta'\xi+\eta)\big(c^{2}(\rho)+\bar{c}^2(\rho,\rho_0)\big)\\
&&-2\bar{c}^2(\rho,\rho_0)\big(-\eta'\xi(c^{2}+\eta^{2})+(\eta-\eta(\eta')^{2}-\xi\eta')(c^{2}-\xi^{2})
     \big) \nonumber
\end{eqnarray}
and
\begin{eqnarray}\label{3.60}
\beta_{2}&=&\eta'(\xi^{2}+\eta^{2})(\eta-\eta'\xi)\big(c^{2}(\rho)+\bar{c}^2(\rho,\rho_0)\big)\\
&&-2\bar{c}^2(\rho,\rho_0) \big(  (\eta'\eta-\xi-\xi
(\eta')^2)(c^{2}-\eta^{2})+\eta'\eta(c^{2}+\xi^{2})\big). \nonumber
\end{eqnarray}

\item[(iii)] The remaining boundary conditions:
\begin{equation}\label{3.62} \rho=\rho_2\,\,\, \text{on}\
\Gamma_{\text{sonic}},\qquad \rho_{\nnu}=0\,\,\, \text{on}\
\Gamma_0,\qquad \rho(P_2)=\bar{\rho}, \end{equation} where $\nnu$ is
the outward normal to $\Omega$ at $\Gamma_0$.
\end{enumerate}
It is easy to check that \eqref{3.58a} is the oblique derivative
boundary condition along $\Gamma_{\text{shock}}$.

\smallskip
With Lemma \ref{convex}, we can show that Case 1 is the only case
for the solutions, which implies that we can obtain the finer
regularity near $P_2$.

\begin{lemma}\label{10a}
Suppose that $(\rho,r)$ is the solution to the free boundary problem
\eqref{3.1}--\eqref{3.5}. Then the shock does not meet the circle
$r=r_0$ at the wedge.
\end{lemma}

\begin{proof}
The main idea of the proof is the same as that in Lemma \ref{5}, and
the only main difference is that the domain to be considered is a
sector instead of a ball. We only list the major procedure and the
difference here. We show our claim by contradiction. Otherwise,
$r(\theta_w)= c_0$.

First, let $\eta=r\cos(\theta-\theta_w)$ and consider
$$
\phi=c_0^{2}+A_1(c_0-\eta)^{\frac{1}{2}}-B_1(c_0-\eta)^{\beta_1}+C_1(\theta-\theta_0)^2,
$$
where $\theta\in[\theta_w,\theta_w+\delta]$, $\delta>0$ small
enough,  $A_1, B_1,C_1>0$ and $\frac{1}{2}<\beta_1<1$, all of which
will be specified later to prove that $\rho\in C^{\frac{1}{2}}$ near
this boundary point.

Since $0\leq c_0^2-\eta^2=(c_0^2-r^2)+r^2\sin^2(\theta-\theta_w)$ on
$\Gamma_{\text{shock}}$ from its convexity indicated in Lemma
\ref{convex}, we have
$$
0\leq r^2-c_0^2\leq r^2\sin^2(\theta-\theta_w)\leq
C(\theta-\theta_w)^2 \qquad\mbox{on $\Gamma_{\text{shock}}$}
$$
for some constant $C>0$. This implies that $ \bar{c}^2-c^2_0\leq
r^2-c^2\leq r^2-c_0^2\leq C(\theta-\theta_w)^2 $. Then
$$
c^2-c_0^2\leq C(\theta-\theta_w)^2,
$$
since $c^2$ and $\bar{c}^2$ are both functions of $\rho$. We can
choose $C_1>0$ so large that $c^2\leq \phi$ on
$\Gamma_{\text{shock}}$. Then, as in the proof of Lemma \ref{5}, we
can now show that $\phi$ is an upper barrier of $\rho$, {\it i.e.},
$c^2\leq \phi$ in $N_1$, which implies
$$
0\leq c^2-c_0^2\leq
A_1(c_0-\eta)^{\frac{1}{2}}+C_1(\theta-\theta_w)^2.
$$

Next, for a lower barrier of $\rho$, as the proof of Lemma \ref{5},
we can show that there exist a neighborhood $N_2$ of
$(r_w,\theta_w)$ and a constant $A_2>0$ such that
$$
c^2-c_0^2\geq A_2(c_0-r)^{\frac{1}{2}}\qquad \mbox{in
$N_2\cap\{(r,\theta)\,:\,r\leq c_0\}$}.
$$
The only new here is the boundary $r=c_0$, which is obvious. This
implies that
$$
a(c_0-r)^{\frac{1}{2}}\leq v:=\rho-\rho_0\leq A(c_0-r)^{\frac{1}{2}}
\qquad\mbox{in $N_1\cap N_2\cap
 V$},
$$
where $V$ is an upward sector containing the wedge, with the vertex
at $P_2$ and the angle smaller than $\frac{\pi}{2}$, for some
constants $a$ and $A$ depending on $V$. This implies that the
optimal regularity along the wedge is $C^{\frac{1}{2}}$ near the
sonic circle.

With this optimal regularity in hand, we introduce the coordinates:
\begin{equation} x=c_{0}-r,\,\, y=\theta-\theta_{w},\,\, v=c^2-c^2_0. \end{equation} Thus,
rewriting the equation for $c^2$ in the divergence form, we have
\begin{equation}
Qv=\big(a_{11}(v+2c_0x-x^2)v_{x}\big)_{x}+b_1v_{x}+(a_{22}v_{y})_{y}=0,
\end{equation} where
$a_{11}=\frac{c^{\frac{2(2-\gamma)}{\gamma-1}}}{\gamma-1}$,
$a_{22}=\frac{c^{\frac{2}{\gamma-1}}}{\gamma-1}\frac{1}{r^2}$, and
$b_1=\frac{c^{\frac{2}{\gamma-1}}}{\gamma-1}\frac{1}{r}$.

As in the proof of Lemma \ref{5}, scale $v$ in $N_1\cap N_3\cap V$
by defining \begin{equation}
u(S,T)=\frac{1}{S^{\frac{1}{5}}}v(S^{-\frac{12}{5}},S^{-\frac{14}{5}}T)
\end{equation} for $(S^{-\frac{12}{5}},S^{-\frac{14}{5}}T)\in N_1\cap N_3\cap
V$. Moreover, $u$ satisfies the governing equation \eqref{4.9}.

From the optimal continuity, $0<a\leq S^{\frac{7}{5}}u\leq A$. Then,
exactly following the proof of Lemma \ref{5}, we obtain a
contradiction when $x_0$ is small. This completes the proof.
\end{proof}

Finally, we establish the Lipschitz continuity for the solution near
the degenerate sonic boundary.

\begin{lemma}\label{17}
The solution $\rho$ to the free boundary problem
\eqref{3.58}--\eqref{3.62} is Lipschitz continuous up to the
boundary $\Gamma_{\text{sonic}}$.
\end{lemma}

\begin{proof}
On one hand, since $\rho\leq \rho_1$ in $\Omega$, we have
$$
c^2(\rho)-\xi^2-\eta^2<c^2(\rho_1)-\xi^2-\eta^2.
$$
On the other hand, it follows from Lemma \ref{15}
 that
$$
c^2(\rho)-\xi^2-\eta^2>\xi^2+\eta^2-c^2(\rho_1)\qquad \mbox{in
$\Omega$}.
$$
Then we have
\begin{equation*}
\begin{array}{lll}
|c^2(\rho)-c^2(\rho_1)|&\leq&|c^2(\rho)-\xi^2-\eta^2|+|c^2(\rho_1)-\xi^2-\eta^2|\\[1.5mm]
&\leq&2|c^2(\rho_1)-\xi^2-\eta^2|\\[1.5mm]
&\leq&4\, c(\rho_1)|c(\rho_1)-\sqrt{\xi^2+\eta^2}|,
\end{array}
\end{equation*}
which implies that $\rho$ is Lipschiz continuous up to the
degenerate boundary $\Gamma_{\text{sonic}}$.

\end{proof}

\medskip
\begin{proof}[\bf Proof of the Existence Part of  Theorem \ref{1}]
 The above seven lemmas, {\it i.e.}, Lemmas \ref{15}--\ref{17}, show that there exists
 a solution
 $$
 (\rho,r)\in C^{2+\alpha}(\Omega)\cap C^{\alpha}(\overline{\Omega})
\cap C^{0,1}(\Omega\cup\Gamma_{\text{sonic}})\times
C^{2+\alpha'}((\theta_w,\theta_1))\cap C^{1,1}([\theta_w,\theta_1]),
$$
which satisfies \eqref{2.40a}--\eqref{2.47a}. This completes the
proof of the existence part.
\end{proof}

\section{Proof of Theorem \ref{2.1}: Optimal regularity near the sonic boundary}

In this section, we prove that the Lipschitz continuity is the
optimal regularity for $\rho$ across the sonic boundary
$\Gamma_{\text{sonic}}$, as well as at the intersection point $P_1$
between $\Gamma_{\text{sonic}}$ and $\Gamma_{\text{shock}}$. In \S
4, we have shown that the solution $\rho$ to the free boundary
problem  \eqref{3.58}--\eqref{3.62} is Lipschitz continuous in
$\Omega$ up to the degenerate boundary $\Gamma_{\text{sonic}}$. Now
we employ the approach introduced in Bae-Chen-Feldman \cite{bcm}
with the aid of the estimates in \S 4 to analyze the finer behavior
of $\rho$ near the sonic circle $r=r_1:=c(\rho_1)$.

\smallskip
For $\varepsilon\in(0, \frac{c_{1}}{2})$, we denote by
$$
\Omega_{\varepsilon}:=\Omega\cap\{(r,\theta):0<c_{1}-r<\varepsilon\},
$$
the $\varepsilon-$neighborhood of the sonic circle
$\Gamma_{\text{sonic}}$ within $\Omega$. In $\Omega_{\varepsilon}$,
we introduce the coordinates: \begin{equation}\label{5.1} (x,
y)=(c_{1}-r,\, \theta-\theta_{1}). \end{equation} One of our main
observations is that it is more convenient to study the regularity
in terms of the difference between $c^2(\rho_1)$ and $c^2(\rho)$:
\begin{equation} \psi:=c^2(\rho_1)-c^2(\rho), \end{equation} since $\psi$ and $\rho$ have
the same regularity in $\Omega_{\varepsilon}$.

It follows from \eqref{1.8b} that $\psi$ satisfies
\begin{equation}
\begin{array}{lll} \label{5.3}
\mathcal{L}_{1}\psi&:=&(2c_{1}x-\psi+O_{1})\psi_{xx}+(c_{1}+O_{2})\psi_{x}-(1+O_{3})\psi^{2}_{x}
\\[2mm]
&& +(1+O_{4})\psi_{yy}-(\frac{1}{(\gamma-1)
c^2_{1}}+O_{5})\psi^{2}_{y}=0\qquad\quad
 \text{in}\ Q^{+}_{r,R}
\end{array}
\end{equation}
in the $(x,y)-$coordinates, where \begin{equation}\label{5.4}
\begin{array}{ll}
O_{1}(x,\psi)=-x^{2},  \qquad &
O_{2}(x,\psi)=-3x+\frac{\psi}{c_{1}},\\[1.5mm]
O_{3}(x,\psi)=-\frac{\gamma-2}{\gamma-1}(2c_{1}x-\psi-x^{2}),
\qquad & O_{4}(x,\psi)=\frac{c_{1}^{2}-\psi}{(c_{1}-x)^{2}}-1,\\[1.5mm]
O_{5}(x,\psi)=\frac{1}{(c_{1}-x)^{2}}-\frac{1}{c_{1}^{2}}.
\end{array}
\end{equation}
Moreover, $\psi$ satisfies \begin{equation}\label{5.5} \psi>0\qquad
\text{in}\ Q^{+}_{r,R} \end{equation} and the following Dirichlet
boundary condition: \begin{equation}\label{5.6} \psi=0\qquad
\text{on}\,\,
\partial Q^{+}_{r,R}\cap\{x=0\},
\end{equation} where
$Q^{+}_{r,R}:=\{(x,y):x\in(0,r),|y|<R\}\subset\mathbb{R}^{2}$, with
$R=\theta_w-\theta_1$, since we can extend $\psi(x,y)$ from
$\Omega_{\varepsilon}$, by defining $\psi(x,y)=\psi(x,-y)$ for
$(x,y)\in\Omega_{\varepsilon}$, and extend the domain
$\Omega_{\varepsilon}$ with respect to $y$. Thus, without further
comment, we study the behavior of $\psi$ in $Q^{+}_{r,R}$.

It is easy to see that the terms $O_{i}(x,y),\ i=1,\cdots,5$, are
continuously differentiable and
\begin{eqnarray}
\qquad \qquad \frac{|O_{1}(x,y)|}{x^{2}}+
 \frac{|O_{k}(x,y)|}{x}
+\frac{|DO_{1}(x,y)|}{x}+|DO_{k}(x,y)|\leq N \quad \text{for}\
k=2,\cdots,5, \label{5.8}
\end{eqnarray}
in $\{x>0\}$ for some constant $N$ depending only on $c_1$ and
$\gamma$. Inequality \eqref{5.8} implies that the terms
$O_{i}(x,y),\ i=1,\cdots,5$, are small. Thus, the main terms of
\eqref{5.3} form the following equation:
 \begin{equation}\label{5.10}
 (2c_{1}-\psi)\psi_{xx}+c_{1}\psi_{x}-\psi^{2}_{x}
+\psi_{yy}-\frac{1}{(\gamma-1) c^2_{1}}\psi^{2}_{y}=0\qquad
\text{in}\ Q^{+}_{r,R}. \end{equation} It follows from Lemmas \ref{15}
and \ref{10a} that
\begin{equation}\label{5.11} 0\leq\psi\leq
2(c_1-\vartheta)x,
\end{equation}
where $\vartheta$ depends only on
$\rho_1$ and $\gamma$. Then equation \eqref{5.10} is uniformly
elliptic in every subdomain $\{x>\delta\}$ with $\delta>0$. The same
is true for \eqref{5.3} in $Q^{+}_{r,R}$ if $r$ is sufficiently
small.

\begin{remark}
If $\hat{r}$ is sufficiently small, depending only on $c_{1}$ and
$\gamma$, then \eqref{5.8} and \eqref{5.11} imply that \eqref{5.3}
is uniformly elliptic with respect to $\psi$ in $Q^{+}_{r,R}\cap
\{x>\delta\}$ for any $\delta\in(0,\frac{\hat{r}}{2})$. We will
always assume such a choice of $\hat{r}$ hereafter.
\end{remark}

\subsection{First-order lower bound of $\psi$}
In order to prove that $C^{0,1}$ is the optimal regularity of $\psi$
across the sonic boundary, our idea is to construct a positive
subsolution of  \eqref{5.3} and \eqref{5.5}--\eqref{5.6} first,
which provides our desired lower bound of $\psi$.

\begin{lemma}\label{5.2}
Let $\psi$ be a solution of the Dirichlet problem \eqref{5.3} and
\eqref{5.5}--\eqref{5.6}. Then there exist $\hat{r}>0$ and $\mu>0$,
depending only on $c_{1}$, $\gamma$, $\theta_w$, and
$\inf\limits_{Q^{+}_{\hat{r},R}\cap \{x>\hat{r}/2\}}\psi$, such
that, for all $r\in (0,\frac{\hat{r}}{2}]$,
\begin{equation}\label{5.12a} \psi(x,y)\geq\mu c_{1} x\ \
\qquad\text{in}\ Q^{+}_{r, \frac{15R}{16}}. \end{equation}
\end{lemma}

\begin{proof}
In the proof below, without further comment, all the constants
depend only on the data, {\it i.e.}, $c_{1}$, $\hat{r}$, $\gamma$,
$\theta_w$, and $\inf\limits_{Q^{+}_{\hat{r},R}\cap
\{x>{\hat{r}}/{2}\}}\psi$, unless otherwise is stated.

Fix $y_{0}$ with $|y_{0}|\leq \frac{15R}{16}$. We now prove that
\begin{equation}\label{5.12} \psi(x,y_{0})\geq \frac{5}{8}\mu
x\qquad\text{for}\ x\in(0,r). \end{equation}

Without loss of generality, we may assume that $R=2$ and $y_{0}=0$;
otherwise, we set $\tilde{\psi}(x,y)=\psi(x,y_{0}+\frac{R}{32}y)$
for all $(x,y)\in Q^{+}_{\hat{r},2}$. Then $\tilde{\psi}(x,y)\in
C(\overline{Q^{+}_{\hat{r},R}})\cap C^{2}(Q^{+}_{\hat{r},R})$
satisfies \eqref{5.3} with \eqref{5.8} and \eqref{5.11} in
$Q^{+}_{\hat{r},2}$, with some modified constants $N$, $\vartheta$,
and $O_{i}$, depending only on the corresponding quantities in the
original equation and on $R$. Moreover,
$$
\inf\limits_{Q^{+}_{\hat{r}2}\cap \{x>\hat{r}/2\}}\tilde{\psi}
=\inf\limits_{Q^{+}_{\hat{r},R}\cap \{x>\hat{r}/2\}}\psi.
$$
Then \eqref{5.12} for $\psi$ follows from \eqref{5.12} for
$\tilde{\psi}$ with $y_0=0$ and $R=2$. Thus we keep the original
notation with $y_{0}=0$ and $R=2$. That is, it suffices to prove
that \begin{equation}\label{5.13} \psi(x,0)\geq \frac{5}{8}\mu
x\qquad\text{for}\ x\in(0,r). \end{equation}

By the Harnack inequality, we conclude that, for any
$r\in(0,\frac{\hat{r}}{2})$, there exists $\sigma=\sigma(r)>0$,
depending only on $r$ and the data $c_{1}$, $\hat{r}$, $\gamma$,
$\theta_w$, and $\inf\limits_{Q^{+}_{\hat{r},R}\cap
\{x>\hat{r}/2\}}\psi$, such that \begin{equation}
\psi\geq\sigma\quad\text{on}\ Q^{+}_{\hat{r},3/2}\cap\{x>r\}.
\end{equation} Let $r\in(0,\frac{\hat{r}}{2})$ and \begin{equation}\label{5.15}
0<\mu_0\leq\min\{\frac{\sigma(r)}{r},c_1\}, \end{equation} where $r$
will be chosen later. Define
\begin{equation}\label{5.16}
g(y)=
\begin{cases}
\mu(y+1)^2,\qquad&-1\leq y<-\frac{1}{2},\\
\mu(2y^4-2y^2+\frac{5}{8}),\qquad&-\frac{1}{2}\leq y\leq\frac{1}{2},\\
\mu(y-1)^2, \qquad& \frac{1}{2}<y\leq 1.
\end{cases}
\end{equation}
Set $w(x,y)=\mu x g(y)$ with $g\in C^2([-1,1])$. Then, using
\eqref{5.15} and \eqref{5.16}, we obtain that, for all $x\in(0,r)$
and $|y|<1$,
\begin{equation*}
\begin{cases}
w(0,y)=0\leq \psi(0,y),\\
w(r,y)\leq\frac{5}{8}\mu r\leq \psi(r,y),\\
w(x,\pm1)=0\leq\psi(x,\pm1).
\end{cases}
\end{equation*}
Therefore, we have \begin{equation} w\leq \psi \qquad\text{on}\
\partial Q^+_{r,1}.  \nonumber \end{equation}

Next, we show that $w(x,y)$ is a strict subsolution
$\mathcal{L}_{1}w(x,y)>0$ in $Q^+_{r,1}$, if the parameters are
appropriately chosen. In fact,
\begin{equation*}
\begin{array}{lll}
&&\mathcal{L}_{1}w(x,y)\\[2mm]
&&=\big(c_1 g(y)-g^2(y)\big)\\[2mm]
&&\quad +x\Big(g''(y)-\frac{1}{(\gamma-1) c^2_{1}}x(g'(y))^2
+\frac{O_2}{x}g(y)-\frac{O_3}{x}g^2(y)+O_4 g'(y)-xO_5(g'(y))^2\Big).
\end{array}
\end{equation*}

On one hand, for $1-|y|<\varepsilon_0$ with $\varepsilon_0$ small
enough, we can see
\begin{eqnarray*}
&&g''(y)-\frac{1}{(\gamma-1)
c^2_{1}}x(g'(y))^2+\frac{O_2}{x}g(y)-\frac{O_3}{x}g^2(y)
+O_4g'(y)-xO_5(g'(y))^2\\
&&\geq g''(y)-\frac{1}{(\gamma-1)
c^2_{1}}x(g'(y))^2-Nx(g(y)+1)g(y)-xN g'(y)
+Nx^2(g'(y))^2\\
&& =:h(x,y).
\end{eqnarray*}
It is easy to see that $h(x,y)$ is continuous with respect to $x$,
$h(0,y)=0$, and that there exists $r_1>0$ such that $h(x,y)>0$ for
$r<r_1$.

On the other hand, for $1-|y|>\varepsilon_0$,
\begin{equation*}
\begin{array}{lll}
&&\mathcal{L}_{1}w(x,y)\\[2mm]
&&\geq x\Big(g''(y)-\frac{1}{(\gamma-1)
c^2_{1}}x(g'(y))^2+\frac{O_2}{x}g(y)-\frac{O_3}{x}g^2(y)
  +O_4g'(y)-xO_5(g'(y))^2\Big)\\[2mm]
&&\quad +\mu\varepsilon^2_0 (c_1-\frac{5}{8}\mu).
\end{array}
\end{equation*}
Then there exists $r_2>0$ such that the above inequality is
positive.

We claim
$$
\sup\limits_{Q^+_{r,1}}(w-\psi)\leq \sup\limits_{ \partial
Q^+_{r,1}}(w-\psi)\leq0,
$$
whenever $0<r<r_0:=\min\{r_1,r_2\}$ and $\mu\in(0,\mu_0]$.
Otherwise, there exists a point $(x_0,y_0)\in Q^+_{r,1}$ such that
\begin{eqnarray*}
0&<&(\mathcal{L}_{1}w-\mathcal{L}_{1}\psi)(x_0,y_0)\\
&=&(2c_{1}x-\psi+O_{1})(w-\psi)_{xx} +(c_{1}+O_{2})(w-\psi)_{x}-(1+O_{3})(w+\psi)_{x}(w-\psi)_{x}\\
&&+(1+O_{4})(w-\psi)_{yy}-\big(\frac{1}{(\gamma-1)
c^2_{1}}+O_{5}\big)(w+\psi)_{y}(w-\psi)_{y}\leq0,
  \nonumber
\end{eqnarray*}
where we have used the fact that $w_{xx}=0$, which is a
contradiction. Hence, we obtain our claim: \begin{equation}
\psi(x,y)\geq w(x,y)=xf(y)\qquad\text{in}\  Q^+_{r,1}.  \nonumber
\end{equation} In particular, \begin{equation} \psi(x,0)\geq\frac{5}{8}\mu x
\qquad\text{for}\quad x\in[0,r].  \nonumber \end{equation} This
implies \eqref{5.12}. Then \eqref{5.12a} holds by modifying $\mu$,
which is still denoted by $\mu$. This completes the proof.
\end{proof}

\subsection{$C^{1,\alpha}$--Estimate of $\psi$}
If $\psi$ satisfies \eqref{5.3}, \eqref{5.5}--\eqref{5.6}, and
\eqref{5.11}, it is expected that $\psi$ is very close to $c_{1}x$,
which is a solution of \eqref{5.10}. More precisely, we now prove
\begin{equation} |\psi(x,y)-c_{1}x|\leq Cx^{1+\alpha}\qquad \text{for all}\
(x,y)\in
 Q^{+}_{\hat{r},\frac{7R}{8}}   \nonumber
\end{equation} for some constant $C$.

To prove this, we study the function: \begin{equation}\label{5.26}
W(x,y):=c_{1}x-\psi(x,y). \end{equation} By \eqref{5.3}, $W$
satisfies
\begin{eqnarray}
&&\mathcal{L}_{2}W=(c_{1}x+W+O_{1})W_{xx}-(c_{1}-O_{2}-2c_{1}O_{3})W_{x}+(1-O_{3})W^{2}_{x} \label{5.27}\\
&&\qquad\qquad +(1+O_{4})W_{yy}-(\frac{1}{(\gamma-1)
c^2_{1}}-O_{5})W^{2}_{y}\nonumber\\
&&\qquad\,\,\,\, = c_{1}O_{2}+c^{2}_{1}O_{3}
\qquad\qquad\qquad\quad \text{in}\ Q^{+}_{\hat{r},R}, \nonumber\\
&&W(0,y)=0\qquad\qquad\qquad\qquad\qquad\quad \text{on}\ \partial Q^{+}_{\hat{r},R}\cap\{x=0\}, \label{5.28}\\
&&-(c_{1}-\vartheta)x\leq W(x,y)\leq c_{1}x\qquad\qquad\text{in}\
Q^{+}_{\hat{r},R}. \label{5.29}
\end{eqnarray}

Then we establish the following two estimates.

\begin{proposition}\label{18}
Let $c_{1}$, $\hat{r}$, $R$, and $\vartheta$ be the same as in Lemma
{\rm \ref{5.2}}. Then, for any $\alpha\in(0,1)$, there exist
positive constants $r$ and $A$, which depend only on $N$, $c_{1}$,
$\hat{r}$, $R$, $\vartheta$, and $\alpha$, such that, if $W\in
C(\overline{Q^{+}_{\hat{r},R}})\cap C^{2}(Q^{+}_{\hat{r},R})$
satisfies \eqref{5.27}--\eqref{5.29}, then
\begin{equation}\label{5.44} W(x,y)\leq Ax^{1+\alpha}\qquad \
\text{in}\ Q^{+}_{r,\frac{3R}{4}}. \end{equation}
\end{proposition}

\begin{proof}
 The main idea of the proof is the same as that in \cite{bcm}, and
we only list the major procedure and the difference here.

First, we prove that there exist $\alpha_{1}\in (0,\frac{1}{2})$ and
$r_{1}>0$ such that, if $W\in C(\overline{Q^{+}_{\hat{r},R}})\cap
C^{2}(Q^{+}_{\hat{r},R})$ satisfies \eqref{5.27}--\eqref{5.29}, then
\begin{equation} W(x,y)\leq
\frac{c_{1}(1-\mu_{1})}{r^{\alpha}}x^{1+\alpha}\qquad \text{in}\
Q^{+}_{r,\frac{7R}{8}},  \nonumber \end{equation} whenever
$\alpha\in (0,\alpha_{1}]$, $r\in (0,r_{1}]$, and
$\mu_{1}<\min\{\mu,\frac{1}{2}\}$, where $\mu$ is the constant
determined by Lemma \ref{5.2}.

As in \cite{bcm}, we first note that, without loss of generality, we
may assume that $R=2$ and $y_{0}=0$. Then it suffices to prove that
\begin{equation} W(x,0)\leq
\frac{c_{1}(1-\mu_{1})}{r^{\alpha}}x^{1+\alpha}\qquad\text{for}\
x\in(0,r)  \nonumber \end{equation} for some $r\in (0,r_{0})$ and
$\alpha\in (0,\alpha_{1})$, under the assumptions that
\eqref{5.27}--\eqref{5.29} hold in $Q^{+}_{\hat{r},2}$. For any
given $r\in (0,r_{0})$, let
$$
v=A_{1}x^{1+\alpha}(1-y^{2})+B_{1}xy^{2}
$$
with $A_{1}r=c_1(1-\mu_{1})$ and $B_{1}=c_{1}(1-\mu_{1})$. Then we
obtain \begin{equation} W\leq v \qquad\text{on}\ \partial
Q^{+}_{r_{0},1}, \nonumber \end{equation} and \begin{equation}
\mathcal{L}_{2}v-\mathcal{L}_{2}W-v_{xx}(v-W)<0 \qquad \text{in}\
Q^{+}_{r,1},  \nonumber \end{equation}
 whenever $r\in (0,r_{1}]$ and
$\alpha\in (0,\alpha_{1}]$ so that
\begin{equation}\label{A-32}
(2\alpha-1)(\alpha+1)c_1A_1<-\frac{\mu_1}{2},
\end{equation}
and
\begin{equation}\label{A-33}
r_1<\min\Big\{\Big(\frac{\mu_1}{4c_1}\Big)^{\frac{1}{\alpha}},
\Big(\frac{B_1c_2-B_1^2}{C}\Big)^{\frac{1}{\alpha}}, r_0\Big\}.
\end{equation}
Then \begin{equation} W\leq v\qquad \text{in}\ Q^{+}_{r,1}.
\nonumber \end{equation}

Next, we generalize the result for any $\alpha\in (0,1)$, which
suffices to show that for the case $\alpha>\alpha_{1}$. Fix any
$\alpha\in( \alpha_{1},1)$ and set the following comparison
function:
$$
v=\frac{c_{1}(1-\mu_{1})}{r^{\alpha_{1}}_{1}r^{\alpha-\alpha_{1}}}x^{1+\alpha}(1-y^{2})
+\frac{c_{1}(1-\mu_{1})}{r^{\alpha_1}_{1}}x^{1+\alpha_{1}}y^{2}.
\nonumber
$$
Then, as before, we can prove
$$
W\leq v\qquad \text{on}\ \partial Q^{+}_{r,1}\qquad \text{for}\
r\in(0,r_{1}].
$$
and
$$
\mathcal{L}_{2}v-\mathcal{L}_{2}W-v_{xx}(v-W)<0.
$$
Then it is easy to prove that this proposition holds with
$$
A=\frac{c_{1}(1-\mu_{1})}{r^{\alpha_1}_{1}r^{\alpha-\alpha_{1}}}.
$$
\end{proof}

\begin{proposition}\label{19}
Let $c_{1}$, $\hat{r}$, $R$, $\vartheta$, and $O_{i}$ be the same as
in Lemma {\rm 5.2}.  Then, for any $\alpha\in(0,1)$, there exist
positive constants $r$ and $B$, depending on $N$, $c_{1}$,
$\hat{r}$, $R$, $\vartheta$, and $\alpha$, so that,  if $W\in
C(\overline{Q^{+}_{\hat{r},R}})\cap C^{2}(Q^{+}_{\hat{r},R})$
satisfies \eqref{5.27}--\eqref{5.29}, we have
\begin{equation}\label{5.50} W(x,y)\geq- Bx^{1+\alpha}\qquad \text{in}\
Q^{+}_{r,\frac{3R}{4}}. \end{equation}
\end{proposition}

\begin{proof}Similar to the proof of Proposition \ref{18},
it suffices to prove that, with the assumption $R=2$,
$$
W(x,0)\geq- \frac{c_{1}-\vartheta}{r^{\alpha}}x^{1+\alpha}\qquad
\text{for}\ x\in(0,r)
$$
for some $r>0$ and $\alpha\in (0,\alpha_{2})$. For this, we use the
comparison function:
$$
v(x,y):=-Lx^{1+\alpha}(1-y^{2})-Kxy^{2},\qquad \text{with}\
Lr^{\alpha}=K=\frac{c_{1}-\vartheta}{r^{\alpha}}.
$$
It is easy to check that
$$
W\geq v \qquad \mbox{on $\partial Q^{+}_{r,1}$ $\,\,\,$ for
$r\in(0,r_{1}]$}.
$$
Then we follow the same procedure as in \cite{bcm}, except that
$\mathcal{L}_{2}v>\mathcal{L}_{2}W$, to find that the conditions for
the choice of $\alpha,\ r>0$ are inequalities \eqref{A-32} and
\eqref{A-33} with $(\mu_{1}, r_{1})$ replaced by $(\beta, r_{2})$,
respectively, and with an appropriate constant $C$.

We claim that
$$
\min\limits_{ Q^+_{r,1}}(W-v)\geq \min\limits_{\partial
Q^+_{r,1}}(W-v)\geq 0.
$$
Otherwise, there exists a point $(x_0,y_0)\in Q^+_{r,1}$ such that
$(W-v)(x_0,y_0)<0$ and
\begin{equation}
\begin{array}{lll}0&>&(\mathcal{L}_{2}W-\mathcal{L}_{2}v)(x_0,y_0)\\[2mm] &=&(c_{1}x+W+O_{1})(W-v)_{xx}-(c_{1}-O_{2}-2c_{1}O_{3})(W-v)_{x}\\[2mm]
&& +(1-O_{3})(W+v)_{x}(v-W)_{x}+(1+O_{4})(W-v)_{yy}\\[2mm]
&& -(\frac{1}{(\gamma-1) c^2_{1}}-O_{5})(W+v)_{x}(W-v)_{x}+v_{xx}(W-v)\\[2mm]
&\geq&0\qquad\text{in}\ Q^+_{r,1},
 \end{array}
 \end{equation}
which is a contradiction. This completes the proof for the case
$\alpha\leq\alpha_2$.

For the case $\alpha\in(\alpha_{2},1)$, we set the comparison
function:
$$
u_{-}(x,y):=-\frac{c_{1}-\vartheta}{r^{\alpha_{2}}_{2}r^{\alpha-\alpha_{2}}}x^{1+\alpha}(1-y^{2})
-\frac{c_{1}-\vartheta}{r^{\alpha_{2}}_{2}}x^{1+\alpha_{2}}.
$$
Then, using the argument as before, we can choose $r>0$
appropriately small such that
$$
\mathcal{L}_{2}u_{-}-\mathcal{L}_{2}W>0
$$
holds for all $(x,y)\in Q^{+}_{r,1}$.
\end{proof}

\begin{lemma}\label{lem:5.5}
Let $\psi\in C(\overline{Q^{+}_{\hat{r},R}})\cap
C^{2}(Q^{+}_{\hat{r},R})$ be a solution of the Dirichlet problem
\eqref{5.3} and \eqref{5.5}--\eqref{5.6}. Then $\psi\in
C^{1,\alpha}(\overline{Q^{+}_{{\hat{r}}/{2},{R}/{2}}})$ for any $
\alpha\in(0,1)$ with
$$
\psi_{x}(0,y)=c_{1},\quad \psi_{y}(0,y)=0\qquad \text{for any}\
|y|\leq\frac{R}{2}.
$$
\end{lemma}

\begin{proof}
The proof is quite similar to that in \cite{bcm}, and the main
difference is the scaling due to the different equations. For fixed
$z_{0}=(x_{0},y_{0})\in Q^{+}_{{r}/{2},{R}/{2}}$, rescale $W$ in
$R_{z_{0}}$ by defining \begin{equation}
W^{(z_{0})}(S,T)=\frac{1}{x^{1+\alpha}_{0}}W(x_{0}+\frac{x_{0}}{8}S,y_{0}
+\frac{\sqrt{x_{0}}}{8}T)\qquad \text{for}\ (S,T)\in Q_{1},\nonumber
\end{equation} where $Q_{h}=(-h,h)^{2}$ for $h>0$. Keep this in mind, we can
prove this lemma easily by following \cite{bcm} step by step. Thus
we omit the detail of proof here.
\end{proof}

Now, following the procedure in \cite{bcm} step by step with the aid
of the results above, we can obtain the next theorem.

\begin{theorem}
Let $\rho\in C^{2+\alpha}(\Omega)\cap C(\overline{\Omega})$ be the
solution of the free boundary problem \eqref{2.40a}--\eqref{2.47a}
in \S 4. Then $\rho$ cannot be $C^{1}$ across the degenerate sonic
boundary $\Gamma_{\text{sonic}}$.
\end{theorem}

\smallskip
We now study more detailed regularity of $\rho$ near the sonic
circle. From now on, we use a localized version of
$\Omega_{\varepsilon}$: For a given neighborhood
$\mathcal{N}(\Gamma_{\text{sonic}})$ of $\Gamma_{\text{sonic}}$ and
$\varepsilon>0$, define
$$
\Omega_{\varepsilon}:=\Omega\cap\mathcal{N}(\Gamma_{\text{sonic}})\cap\{x<\varepsilon\}.
$$
Since $\mathcal{N}(\Gamma_{\text{sonic}})$ is fixed in the following
theorem, we do not specify the dependence of $\Omega_{\varepsilon}$
on $\mathcal{N}(\Gamma_{\text{sonic}})$.

Finally, we show the regularity part of Theorem \ref{1}.
\begin{theorem}  Let $\rho$ be
the solution of the free boundary problem \eqref{3.58}--\eqref{3.62}
established in \S {\rm 4} and satisfy the properties: There exists a
neighborhood $\mathcal{N}(\Gamma_{\text{sonic}})$ of
$\Gamma_{\text{sonic}}$ such that, for $\psi:=c^{2}_{1}-c^2(\rho)$,

\medskip
\noindent {\rm (a)} $\psi$ is $C^{0,1}$ across part
$\Gamma_{\text{sonic}}$ of the degenerate sonic boundary;

\medskip
\noindent {\rm (b)} there exists $\vartheta_{0}>0$  so that, in the
coordinates \eqref{5.1},
 \begin{equation}\label{5.64}
 |\psi|\leq(2c_{1}-\vartheta_{0})x\qquad \text{in}\
\Omega\cap\mathcal{N}(\Gamma_{\text{sonic}}). \end{equation} Then we
have
\begin{itemize}
\item[\rm (i)] There exists $\varepsilon_{0}>0$ such that $\psi$ is
$C^{1,\alpha}$ in $\Omega$ up to $\Gamma_{\text{\rm sonic}}$ away
from point $P_1$ for any $\alpha\in(0,1)$. That is, for any
$\alpha\in(0,1)$ and
$(\xi_{0},\eta_{0})\in\overline{\Gamma_{\text{\rm sonic}}}\backslash
P_1$, there exists $K<\infty$ depending only on $\rho_{0}$,
$\rho_{1}$, $\gamma$, $\varepsilon_{0}$, $\alpha$,
$\|\psi\|_{C^{0,1}}$, and $d=\text{\rm
dist}((\xi_{0},\eta_{0}),\Gamma_{\text{\rm sonic}})$ so that
$$
\|\psi\|_{1,\alpha;\overline{B_{d/2}}(\xi_{0},\eta_{0})\cap\Omega_{{\varepsilon_{0}}/{2}}}\leq
K;
$$

\item[\rm (ii)] For any
$(\xi_{0},\eta_{0})\in\Gamma_{\text{\rm sonic}}\backslash P_1$, $
\lim\limits_{\begin{subarray}{l}(\xi,\eta) \rightarrow
(\xi_{0},\eta_{0})\\(\xi,\eta)\in\Omega
\end{subarray}}
D_{r}\psi=c_{1}; $

\item[\rm (iii)] The limit
$\lim\limits_{\begin{subarray}{l}(\xi,\eta)\rightarrow
P_1\\(\xi,\eta)\in\Omega\end{subarray}}D_{r}\psi$ does not exist.
\end{itemize}
\end{theorem}

The proof is quite similar to the one in \cite{bcm}, which can be
achieved by following the proof of Theorem 4.2 in \cite{bcm} step by
step with the aid of the estimates obtained above. Hence we omit the
proof here.

\section{Proof of Theorem \ref{2.1}:  Global Solutions}

Finally, we show that the solution established above is a global
solution indeed, valid through the sonic circle $\Gamma_{\rm
sonic}$, as claimed in Theorem \ref{2.1}.

\smallskip
Since $\rho$ is only Lipschitz continuous across the sonic circle,
we treat the solution in the weak sense: For every $\zeta\in
C_c^{\infty}(\Omega_-)$, with $\Omega_-$ denoting the region of the
left state,
$$
\int_{\Omega_-}\big((c^2-r^2)\rho_r\zeta_r
+\frac{c^2}{r^2}\rho_{\theta}\zeta_{\theta}
-\frac{c^2}{r}\rho_r\zeta\big)\, \mathrm{d}r\mathrm{d}\theta=0.
$$
Notice that $\rho$ is Lipschitz continuous across the sonic circle.
Then, due to the Green theorem, the integrand is equal to $0$ if and
only if
 \begin{equation}
 [\big((c^2-r^2)\rho_r,\frac{c^2}{r^2}\rho_{\theta}\big)\cdot
 \nnu]=0 \qquad\text{on }\Gamma_{\text{sonic}},
 \end{equation}
where the bracket $[\cdot]$ denotes the difference of the quantity
between two sides of the sonic circle, and $\nnu$ is the normal
direction. It is obvious because from the facts that
$(\rho_r,\rho_{\theta})=(-c_1,0)$ up to the sonic circle from the
subsonic domain obtained in Lemma \ref{lem:5.5},
$(\rho_r,\rho_{\theta})=(0,0)$ from  the supersonic domain and the
fact that $c^2-r^2=0$ on the sonic circle. This completes the proof
of Theorem \ref{2.1}.

\section{Existence and Regularity of Global Solutions of the Nonlinear Wave System}

In our main theorem, Theorem \ref{1}, we have constructed a global
solution $\rho$ of the second-order equation \eqref{3.58} in
$\Omega$, combining this function with $\rho=\rho_1$ in state (1)
and $\rho=\rho_0$ in state (0). That is, we have obtained the global
density function $\rho$ that is piecewise constant in the supersonic
region, which is Lipschitz continuous across the degenerate sonic
boundary $\Gamma_{\text{sonic}}$ from $\Omega$ to state (1).

To recover the momentum components, $m$ and $n$, we can integrate
the second and third equation in \eqref{1.5a}. These can be also
written in the radial variable $r$, \begin{equation} \label{6.1}
\frac{\partial m}{\partial r}=\frac{1}{r} p(\rho)_{\xi}, \qquad
\frac{\partial n}{\partial r}=\frac{1}{r}p(\rho)_{\eta},
\end{equation} and integrated from the boundary of the subsonic region
toward the origin.

Note that we have proved that the limit of $D\rho$ does not exist at
$P_1$ as $(\xi,\eta)$ in $\Omega$ tends to $(\xi_1,\eta_1)$, but $|D
c(\rho)|$ has a upper bound. Thus, $p(\rho)$ is Lipschitz, which
implies that $(m, n)$ are at least Lipschitz across the sonic circle
$\Gamma_{\rm sonic}$.

Furthermore, $(m,n)$ have the same regularity as $\rho$ inside
$\Omega$ except the origin $r=0$. However, $(m, n)$ may be
multi-valued at the origin $r=0$.

In conclusion, we have

\begin{theorem}\label{6.1a}
Let the wedge angle $\theta_w$ be between $-\pi$ and $0$. Then there
exists a global solution $(\rho, m, n)(r,\theta)$ with the free
boundary $r=r(\theta), \theta\in [\theta_w, \theta_1]$, of {\rm
Problem 2} such that
$$
(\rho, m, n)\in C^{2+\alpha}(\Omega), \quad \rho\in
C^{\alpha}(\overline{\Omega}), \quad r\in
C^{2+\alpha}([\theta_w,\theta_1))\cap C^{1,1}([\theta_w,\theta_1]),
$$
and $(\rho, m, n)=(\rho_1, m_1, 0)$ in the domain $\{\xi<\xi_1,
r>r_1\}$ and $(\rho_0, 0,0)$ in the domain $\{\xi>\xi_1,
\eta>\eta_1\}\cup\{r>r(\theta), \theta\in [\theta_w, \theta_1]\}$.
Moreover, the solution $(\rho,m,n)(r,\theta)$ with the free boundary
$r=r(\theta)$ satisfies the following properties:
\begin{enumerate}
\item[\rm (i)] $\rho> \rho_0$ on the shock $\Gamma_{\rm shock}$, that is, the shock $\Gamma_{\rm shock}$
is separated from the sonic circle $C_0$ of state $(0)$;

\item[\rm (ii)] The shock $\Gamma_{\rm shock}$ is convex in the self-similar coordinates $(\xi, \eta)$ and
strictly convex up to point $P_1$, except point $P_2$;

\item[\rm (iii)] The solution $(\rho,m,n)$ is $C^{1,\alpha}$ up to $\Gamma_{\rm sonic}$ and
Lipschitz continuous across $\Gamma_{\text{\rm sonic}}$;

\item[\rm (iv)] The Lipschitz regularity of the solution
across $\Gamma_{\rm sonic}$ and at $P_1$ from the inside is optimal;

\item[\rm (v)] The momentum components $(m,n)$ may be multi-valued at the origin.
\end{enumerate}
\end{theorem}

\section*{Appendix: Proof of Lemma \ref{2}}
For self-containedness, we illustrate a stretched proof of Lemma
\ref{2} in the following:

\smallskip
\begin{proof}
For the notational simplicity, we write $\rho=\rho^{\e,\delta}$
throughout the proof.

1. The existence part of the proof is similar to that in
\cite{sbe2}. The main idea is that, for any function $w\in
\mathcal{W}$, we define a mapping
$$
T:\mathcal{W}\subset  C_{(-\gamma_1)}^2\rightarrow C_{(-\gamma_1)}^2
$$
by $Tw=\rho$, where $\rho$ is the solution to the linear regularized
fixed boundary problem \eqref{3.23}--\eqref{3.24} solved in Lemma
\ref{10}. By Lemma \ref{10}, $T$ obviously maps $\mathcal{W}$ into a
bounded set in $C_{(-\gamma_V)}^{2+\alpha}$, where $\gamma_V$ is the
value given by Lemma \ref{10}. Since $\gamma_V$ is independent of
$\gamma_1$, we may take $\gamma_1=\frac{\gamma_V}{2}$ so that
$T(\mathcal{K})$ is precompact in $C_{(-\gamma_1)}^2$.

Next, it is easy to verify that $Tw$ satisfies (W1) and (W3) in
Definition \ref{6} by the boundary conditions, the maximum
principle, and the standard interior and boundary H\"{o}lder
estimates ({\it cf.} Theorems 8.22 and 8.27 in \cite{gt}). In order
to show that $T$ maps $\mathcal{W}$ into itself, the remaining task
is to show that $Tw$ satisfies (W2) in Definition \ref{6}. To
achieve this, it suffices to find $K>0$ such that \begin{equation}
\label{3.51} \sup\limits_{\delta>0} \big(\delta^{2-\gamma_1}
\|\rho\|_{2,\overline{\Omega}\backslash\{\Gamma(\delta)\cup
\Omega_V(\delta)\}}\big)<K, \end{equation} under the assumption that
$\|w\|^{(-\gamma_1)}_2\leq K$. Note that Lemma \ref{8} gives us a
local bound for the weighted norm of $\rho$ on $\Gamma(d_0)$ of the
form \begin{equation}\label{3.52} d^{2-\gamma_1}\|\rho\|_2\leq
d^{1-\gamma_1+\mu}C, \end{equation} which holds for all $d<d_0$,
where $C$ depends on $K$, $\alpha_1$, and $\gamma_1$. To show
\eqref{3.51},  we make the $L^\infty$--estimate by considering
separately the domains in
$\overline{\Omega}\backslash\{\Gamma(\delta)\cup \Omega_V(\delta)\}$
for which $\delta>\tilde{d}$, with $\tilde{d}\leq d_0$ to be
specified later, and the domains for which $\delta\leq \tilde{d}$.

In the domains of the first kind,
$\overline{\Omega}\backslash\{\Gamma(\delta)\cup \Omega_V(\delta)\}$
with $\delta>\tilde{d}$, the solution is smooth, and trivially its
$C^2$--norm bound is independent of $K$ by the uniform H\"{o}lder
estimate, the interpolation inequality ({\it cf.} Lemma 6.32,
\cite{gt}), and the bootstrap iteratively.

Finally, we estimate
$\delta^{2-\gamma_1}\|\rho\|_{2,\overline{\Omega}\backslash\{\Gamma(\delta)\cup
\Omega_V(\delta)\}}$ with $\delta\leq\tilde{d}$. We divide the
subdomain $\overline{\Omega}\backslash\{\Gamma(\delta)\cup
\Omega_V(\delta)\}$ into two parts: The part for which
$\delta>\tilde{d}$ and its complement. The supremum over the
subdomain for which $\delta>\tilde{d}$ has been calculated above.
Next, we use the estimates for the behavior of the solution near
$\Gamma_{\text{shock}}$ to obtain the supremum over the complement.
By the interpolation inequality, let $\gamma_1=\frac{\gamma_V}{2}$,
we can obtain
$$
d^{2-\gamma_1}\|\rho\|_2\leq K_V \qquad \mbox{for all}\,\, d<d_V,
$$
where $K_V$ is independent of $K$. Therefore, we can choose
$\tilde{d}\leq \frac{\min\{d_0,d_V\}}{2}$ in \eqref{3.52} small
enough that $\tilde{d}^{1-\gamma_1+\mu}C\leq K$. Therefore,
\eqref{3.51} is satisfied, and we have chosen the parameters $K$,
$K_0$, and $\alpha_0$ defining $\mathcal{W}$ so that $T$ maps
$\mathcal{W}$ into itself.

Now, by the Schauder fixed point theorem, there exists a fixed point
$\rho$ such that $T\rho=\rho\in  C_{(-\gamma_1)}^{2}$. Then $\rho$
is a solution of the boundary value problem \eqref{3.1} and
\eqref{3.3}--\eqref{3.5} and meets the estimates listed in the
lemma.

\smallskip
2. We now show the three properties listed in this lemma for the
fixed boundary nonlinear problem \eqref{3.1} and
\eqref{3.3}--\eqref{3.5}. First we prove property (i):
\begin{equation}\label{3.5a} c^2(\rho^{\e,\delta})-r^2\geq0\qquad\text{in}\
\overline{\Omega}^{\e,\delta}, \end{equation} by the maximum principle.

On contrary, we assume that there exists a nonempty set
$D=\{(\xi,\eta)\in\overline{\Omega}\, : \, c^{2}(\rho)-r^{2}<0\}$.
Then it is easy to check that $P_2\notin D$. Since $O\notin D$,
$$
D\subset\Omega_{s}:=\{X\in \overline{\Omega}\backslash V\,:\,
r^2>\bar{c}^2(\rho,\rho_0)\},
$$
where $V$ is the set of the corner points of $\Omega$.

Firstly, inside $\Omega_{s}$, multiplying
$(\gamma-1)\rho^{\gamma-2}$ both sides of the equation
$Q^{\e,+}\rho=0$, and denoting $c^{2}(\rho)=\rho^{\gamma-1}=u$, we
have
\begin{equation}
\begin{array}{lll}
Lu&=&(\gamma-1)\rho^{\gamma-2} Q^{\e,+}\rho\\[2mm]
&=&\sum_{i=1}^2
a^{\e}_{ii}(D_{ii}u-\frac{\gamma-2}{\gamma-1}\frac{1}{\rho^{\gamma-1}}|D_{i}u|^{2})
+\zeta'(c^2-r^2)(c^2-r^2)_ru_r
+\frac{1}{r^2}u^2_{\theta}+b^{\e}u_r\\
&=&0.
\end{array}
\end{equation}
We note that $\frac{\e}{2}\leq a^{\e}_{11}\leq\e$ due to the cut-off
function $\zeta$ in $D$. We evaluate $Lr^{2}$ in $D$:
\begin{equation}
\begin{array}{lll}Lr^{2}&\geq&-2\e \big|1-\frac{2(\gamma-2)}{\gamma-1}\frac{1}{\rho^{\gamma-1}}r^{2}\big|
+\zeta'(c^2-r^2)(c^2-r^2)_ru_r+2c^2\\[2mm]
&\geq&2r^2_0-\frac{2\e}{\rho_0^{\gamma-1}}\big|\rho_0^\gamma-1-\frac{2|\gamma-2|}{\gamma-1}r^2_0\big|>0
\end{array}
\end{equation}
with small
$\e<\e_0:=\frac{\rho_0^{\gamma-1}r^2_0}{\big|\rho_0^{\gamma-1}-\frac{2|\gamma-2|}{\gamma-1}r^2_0\big|}$
when $(c^2-r^2)_r=0$. Then it means that  the minimum point of
$c^2-r^2$ can not obtained in $D$.

Secondly, along $\Gamma_{\text{shock}}\cap D$, Multiplying
$(\gamma-1)\rho^{\gamma-2}$ over the equation $M\rho=0$, we have the
boundary condition for $u$:
$$
0=(\gamma-1)\rho^{\gamma-2}M\rho=\tilde{M}u=\sum_{i=1}^2
\beta_{i}D_{i}u.
$$
At the same time, we have
\begin{equation}
\tilde{M}r^2=2r\beta_1=2rr'\big(c^2(r^{2}-\bar{c}^2)-3\bar{c}^2(c^{2}-r^{2})\big)>0
\qquad\text{on $\Gamma_{\text{shock}}\cap D$},
\end{equation}
where we have used the fact that $r^2\geq c^2\geq \bar{c}^2$ in
$\Omega_s$. Thus it means that the minimum point of $c^2-r^2$ can
not obtained along $\Gamma_{\text{shock}}\cap D$.

Thirdly, on $\Gamma_0\cap D$,
$$
(\gamma-1)\rho^{\gamma-2}\frac{\partial\rho}{\partial\nnu}-\frac{\partial
r^2}{\partial\nnu}=0,
$$
which is a contradiction due to the Hopf maximum principle.
Therefore, there is no minimum point, which implies that the set
$D=\emptyset$. This completes the proof of property ${\rm (i)}$. We
remark here that property ${\rm (i)}$ guarantees the ellipticity of
our nonlinear system, so that we can remove the cut-off function.

3. We can show property ${\rm (ii)}$, {\it i.e.},
$$
r-\bar{c}(\rho,\rho_0)\geq0 \qquad \mbox{on} \,\,\,
\Gamma_{\text{shock}}.
$$
The proof is similar to \cite{e} based on Lemma \ref{2}. The main
idea is to assume that there exists a non-empty set $ B=\{X\in
\overline{\Gamma_{\text{shock}}}:\bar{c}(\rho, \rho_0)-r>0\} $ and a
point $X\in B$ such that
$$
\max\limits_{\overline{B}}\big(\bar{c}^2(\rho,
\rho_0)-r^{2}\big)=(\bar{c}^{2}-r^{2})(X)=m>0.
$$
It is clear that $X\neq P_1,P_2$. Therefore, if $X$ exists, then
$X\in \Gamma_{\rm shock}\setminus\{P_1, P_2\}$. Then $X$ can be
either a local maximum point or a saddle point in $\Omega\cup
\Gamma_{\text{shock}}$. We show that both cases can not occur, which
implies that such $X$ does not exist. The case that $X$ is a local
maximum point is proved by the maximum principle. For the more
complicated case that $X$ is a saddle point, then multiplying
$(\bar{c}^2)'$ both sides of $Q^{\e}\rho=0$ yields
$$
L\bar{c}^2=\sum_{i=1}^2
a^{\e}_{ii}D_{ii}(\bar{c}^2)+a_1(\bar{c}^2)^2_r
+a_2(\bar{c}^2)^{2}_{\theta}+\tilde{b}^{\e}(\bar{c}^2)_r,
$$
where $a_1=-a^{\e}_{11}+\frac{a}{(\bar{c}^2)'}$ and
$a_2=-a^{\e}_{22}+\frac{a}{r^2(\bar{c}^2)'}$.

Since $X$ is a saddle point, we can construct a barrier function
$\psi$ so that $X=(r_{x},\theta_{x})$ is a maximum point along the
normal direction.

We define $d:=r_x-r+r'(\theta_x)(\theta-\theta_x)$ and a set
$$
W:=\{(r,\theta)\in\Omega:d>0\}\cap\{(r,\theta)\in\Omega:\bar{c}^2-r^{2}>m\}.
$$

Set $u=\bar{c}^2-r^{2}-m$, and let
$$
w:=\frac{1}{\mu_0}(e^{\mu_0u}-1),  \qquad \mu_0>0.
$$
Choose $\mu_0=\frac{\max\{a_i\}}{e_0}$, where $a^{\e}_{11},
a^\e_{22}\geq e_0>0$ in $W$,  and $\mu_0$ and $e_0$ are independent
of $\e$. Thus, we find $\psi(d)$ to be
$$
\psi=\frac{m_0b_1+d_1f_1}{b_1}\frac{1-e^{-b_1d/e_0}}{1-e^{-b_1d_1/e_0}}-\frac{f_1}{b_1}d,
$$
which satisfies the boundary condition:
$$
\psi(0)=0, \qquad\psi(d_1)=m_0,
$$
with $m_0=\frac{e^{\mu_0 u_{\text{max}}}-1}{\mu_0}$, where
 $u_{\text{max}}=\max\limits_{W}u=\max\limits_{W}(\bar{c}^2-r^{2}-m)$ and $d_1>0$.
Here $b_1=\max\limits_{\overline{W}}(4a_1r+b)$, $f_1=\max\limits_{W}
e^{\mu_0 u}(2a^{\e_0}_{11}+4r^2a_1+2r\tilde{b}^{\e})^+$, and
$\varepsilon\leq\varepsilon_0$.

Hence, in the set $W$, using the maximum principle and
\eqref{mu-def}, at $X$, we finally have
\begin{eqnarray*}
0&\geq&2r\beta_1-\psi'\mu d_r(X)
+\beta_2(\psi'd\theta+\beta_2\psi'r'dr)(X)\\[2mm]
&=&2r\beta_1+\psi'(0)(\beta_1-r'\beta_2)\\[2mm]
&=&2r\beta_1+\psi'(0)\mu.
\end{eqnarray*}
On the other hand, by the Taylor series expansion, we can show that,
for sufficiently small $d_1$ and $\vartheta=O(d_1)$, we have
$$
2r\beta_1+\psi'(0)\mu>0,
$$
which is a contradiction. Therefore, there is no such $X$, which
implies that the set $B=\emptyset$.

\smallskip
4. We study the monotonicity of $\rho$ along the shock boundary
$\Gamma_{\text{shock}}$, which will be used to describe the behavior
of $\rho^{\e,\delta}$ and $r^{\e,\delta}$ near the shock
$\Gamma^{\e,\delta}_{\text{shock}}$ when $\e,\delta$ tend to zero,
and the convexity of the shock in the $(\xi,\eta)$-coordinates.

The proof is technical, which can be followed as in \cite{sbe2},
with the main difference that we only need the uniform
$C^{\alpha}$--regularity. We only list the major procedure and the
difference. For simplicity, we write $\rho=\rho^{\e,\delta}$ below.
To prove the monotonicity, we argue by contradiction.

First, we examine the $C^{\alpha}$--function $\rho$ restricted to
$\Gamma_{\text{shock}}$. Without confusion, we may order the points
along
 $\Gamma_{\text{shock}}$ by $\theta$ and
 refer to the intervals along $\Gamma_{\text{shock}}$ by the label.
Then the lack of monotonicity implies that there exist points
$\Theta_1$ and $\Theta_2$ on $\Gamma_{\text{shock}}$, with
$P_2<\Theta_1<\Theta_2<P_1$, at which
$\rho(\Theta_1)>\rho(\Theta_2)$. Thus we immediately deduce that

\smallskip
\par
(a). In $(P_2, \Theta_2)$, there exists $\tilde{C}$ with
$\rho(\tilde{C})=\max\limits_{[P_2, \Theta_2]}\rho$;
\par
(b). In $(\tilde{C}, P_1)$, there exists $D$ with
$\rho(D)=\min\limits_{[\tilde{C}, P_1]}\rho$.

\noindent We want to identify points $C$ and $D$ on
$\Gamma_{\text{shock}}$ with $C<D$ such that

\smallskip
(i) $\rho(P_2)\leq \rho\leq\rho(C)$ on $[P_2, C]$;

\smallskip
(ii)  $\rho(C)\geq \rho\geq\rho(D)$ on $[C, D]$;

\smallskip
(iii) $\rho(D)\leq \rho\leq\rho(P_1)$ on $[D, P_1]$.

\smallskip
\noindent Now, property (ii) may not hold with $C=\tilde{C}$ because
$\rho(\tilde{C})$ is the maximum value of $\rho$ only at the
interval $[P_2, D]$, and we may have $D>\Theta_2$. Then, if there is
a point in $(P_2, \Theta_2)$ at which $\rho>\rho(\tilde{C})$, we let
$C$ to be the point. Otherwise, we choose $C=\tilde{C}$. Thus, all
the three  properties hold.

Now we look at the function $\rho$ in $\Omega$. The idea is to
partition $\Omega$ into three subdomains by two curves $\Gamma_C$
and $\Gamma_D$ from $C$ and $D$ to points $A$ and $B$ respectively
on $\Gamma_0$, in such a way $\rho(A)>\rho(B)$ that we can deduce
that there is a point $m$ on $\Gamma_0$ at which $\rho$ obtains a
maximum on either the subdomain $\Omega_A$ or the domain $\Omega_B$,
thus violating the Hopf maximum principle. This is also the case
even if it happens to be the origin $O$. It suffices to show that
$\rho(m)$ is the maximum value of $\rho$ on the boundary of
$\Omega_A$ or $\Omega_B$.

We now construct the Lipschitz curves on which $\rho$ has certain
monotone property. That is,
\begin{eqnarray*}
&&\rho(A)\geq \rho\geq \rho(C)-\mu\quad\, \text{on}\ \Gamma_C,
\qquad\ \rho(A)>\rho(C),
\\
&&\rho(B)\leq \rho\leq \rho(D)+\mu\quad\, \text{on}\ \Gamma_D,\qquad
\,\, \rho(B)<\rho(D),
\end{eqnarray*}
for certain number $\mu>0$. We specify
$$
\mu=\frac{1}{4}\min\{\rho(C)-\rho(D), \rho_1-\rho(C),
 \rho(D)-\bar{\rho}\}.
$$
Since $\rho\in C^{\alpha}(\overline{\Omega}), $ we have
$$
|\rho(X_1)-\rho(X_2)|\leq M|X_1-X_2|^{\alpha}
$$
for some $M>0$ and $X_1,X_2\in\overline{\Omega}$. Now, on any ball
with radius $r>0$,
$$
\text{Osc}(\rho)\leq 2Mr^{\alpha}.
$$
Let $R=(\frac{\mu}{2M})^{-\alpha}$. We have
$$
\text{Osc}_{B_R\cap \Omega}(\rho)\leq\mu.
$$
Now $\Gamma_C$ can be constructed as follows ({\it cf.} Fig.
\ref{mon}): In $B_R(C)\cap\Omega$, let $X_1$ be a point at which
$\rho$ attains its maximum value in $\overline{B_R(C)}$. Then the
first segment of $\Gamma_C$ is a straight line from $C$ to $X_1$
and, on the segment, we have
$$
\rho(X)\geq \rho(C)-\mu, \qquad \rho(X)\leq\rho(X_1).
$$

\begin{figure}[h!]
\centering
      \includegraphics[width=0.6\textwidth]{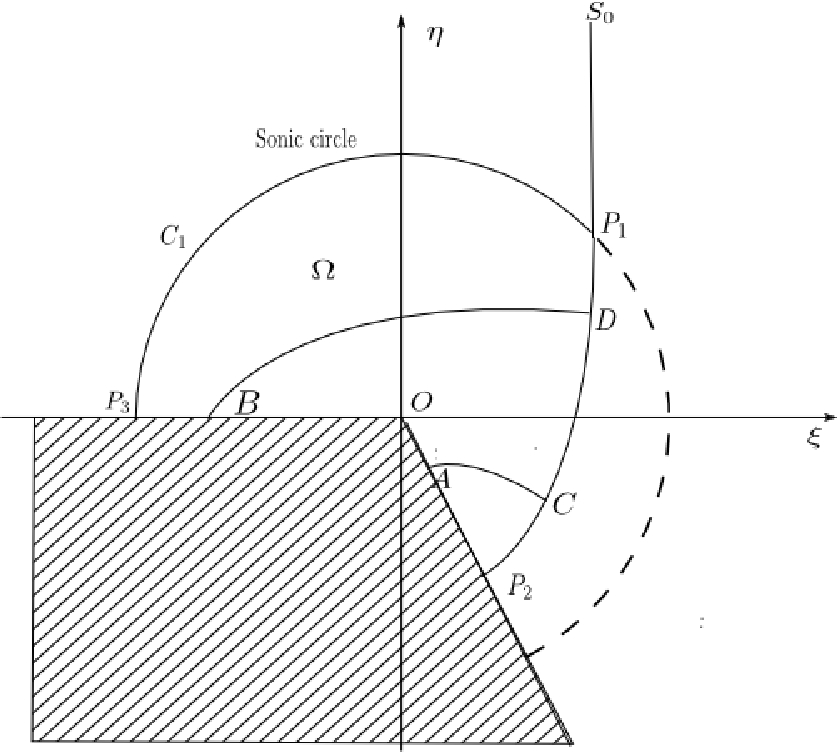}
  \caption{Hypothetical Curves}
  \label{mon}
\end{figure}

Now we continue inductively, forming a sequence of the line segments
with corners at $\{X_i\}$ (take $X_0=C$), along which $\rho(X)\geq
\rho(C)-\mu$ and $\rho(X_1)<\rho(X_2)<\cdots$. Since the domain
$\Omega$ is finite, this process must end at finite steps when we
reach a point $X_L=B\in\partial \Omega$. Similarly, we construct
$\Gamma_{D}$, with termination point $A\in \partial\Omega$.

We now locate $A$ and $B$. We note that the two curves cannot cross
each other. Furthermore, $\Gamma_{C}$ cannot terminate at
$\Gamma_{\text{sonic}}$ where $\rho>\rho_1-\mu>\rho(D)+\mu$. For the
same reason, it can not come back to $\Gamma_{\text{shock}}$ in
$[P_2,C]$ or $[C,D]$ where $\rho\leq \rho(C)$. Finally,  $A$ cannot
lie in the segment $[P_2,C]$ of $\Gamma_{shock}$. Hence, $A$ has to
end on $\Gamma_0$. Similarly, $B$ cannot lie on $\Gamma_{\rm shock}$
where $\rho\geq \rho(D)$ in the interval $[D, P_1)$ and must lie on
$\Gamma_0$ (see Fig. \ref{mon}).

Now we reach to our final contradiction. Since $\rho(A)$ is larger
than $\bar{\rho}$ and $\rho(B)$, there is a point $m$ along the
boundary $P_2OB$ at which $\rho$ attains a maximum. Assume first
that $m$ is not the origin, then $m$ can not be a local maximum for
the domain $\Omega$ by the Hopf lemma. However, along the entire
boundary of the domain $P_2CDBP_2$, $\rho\leq\rho(m)$, which implies
that it is a
maximum. This is a contradiction. Now, if $m$ coincides
with $O$, the similar minimum point $X$ resembling $B$ can not
coincide with $O$. We can find that there is no place for such $X$
either. Thus, this is also a contradiction. We conclude that $\rho$
is monotone along $\Gamma_{\text{shock}}$ from $P_2$ to $P_1$.

\end{proof}

\bigskip
\noindent
{\bf Acknowledgement}.
The authors would like to thank Mikhail Feldman for helpful
discussions and suggestions. The research of Gui-Qiang Chen was
supported in part by the National Science Foundation under Grants
DMS-0935967 and DMS-0807551, the UK EPSRC Science and Innovation
Award to the Oxford Centre for Nonlinear PDE (EP/E035027/1), the
NSFC under a joint project Grant 10728101, and the Royal
Society--Wolfson Research Merit Award (UK). Xuemei Deng's research
was supported in part by China Scholarship Council No. 2008631071
and by the EPSRC Science and Innovation Award to the Oxford Centre
for Nonlinear PDE (EP/E035027/1). The research of Wei Xiang was
supported in part by China Scholarship Council No.
 2009610055 and by the EPSRC Science and
Innovation award to the Oxford Centre for Nonlinear PDE
(EP/E035027/1).



\end{document}